\documentclass[10pt]{article}
\usepackage{amssymb}
\usepackage{amsthm}
\usepackage{amsmath}
\usepackage{enumerate}
\usepackage{pdflscape}
\usepackage{capt-of}
\usepackage{pbox}

\usepackage{hyperref}
\usepackage[left=2cm,right=2cm,top=1.5 cm,bottom=1.5 cm]{geometry}
\hypersetup{
         bookmarks=true   % show bookmarks bar?
    unicode=true,          % non-Latin characters in Acrobat‚Äôs bookmarks
    pdftoolbar=true,        % show Acrobat‚Äôs toolbar?
    pdfmenubar=true,        % show Acrobat‚Äôs menu?
    pdffitwindow=false,     % window fit to page when opened
    pdfstartview={FitH},    % fits the width of the page to the window
    pdftitle={removal_lemma},    % title
    pdfauthor={Lluis Vena},     % author
    pdfsubject={Removal Lemma for Groups},   % subject of the document
    pdfcreator={-},   % creator of the document
    pdfproducer={-}, % producer of the document
    pdfkeywords={} , % list of keywords
    pdfnewwindow=true,      % links in new window
    colorlinks=true,       % false: boxed links; true: colored links
    linkcolor=blue,          % color of internal links
    citecolor=black,        % color of links to bibliography
    filecolor=magenta,      % color of file links
    urlcolor=blue          % color of external links
}

\usepackage{url}
\usepackage[dvistyle, colorinlistoftodos]{todonotes}

\newtheorem{theorem}{Theorem}
\newtheorem{lemma}[theorem]{Lemma}
\newtheorem{corollary}[theorem]{Corollary}
\newtheorem{definition}[theorem]{Definition}
\newtheorem{proposition}[theorem]{Proposition}
\newtheorem{remark}[theorem]{Remark}
\newtheorem{observation}[theorem]{Observation}
\newtheorem{claim}{Claim}

\newcommand{\Z}{\mathbb Z}

\begin{document}

\title{On the removal lemma for linear configurations in finite abelian groups}
\author{Llu\'{i}s Vena\thanks{Department of Mathematics, University of Toronto, Canada. \newline Supported by a University of Toronto Graduate Student Fellowship.
\newline Charles University, Czech Republic. \newline Supported by ERC-CZ project LL1201 CORES. \newline E-mail: {\tt lluis.vena@gmail.com}}}%-Cros}
\date{\today}

\maketitle

\begin{abstract}
We present a general framework to represent discrete configuration systems using hypergraphs. This representation allows one to transfer combinatorial removal lemmas to their analogues for configuration systems. These removal lemmas claim that a system without many configurations can be made configuration-free by removing a few of its constituent elements. As applications of this approach we give; an alternative proof of the removal lemma for permutations by Cooper~\cite{cooper06}, a general version of a removal lemma for linear systems in finite abelian groups, an interpretation of the mentioned removal lemma in terms of subgroups, and an alternative proof of the counting version of the multidimensional Szemer\'edi theorem in abelian groups with generalizations.
\end{abstract}

%\listoftodos\relax

%\tableofcontents

%%%%%%%%%%%%%%%%%%%%%%%%%%%%%%%%%%%%%%%%%%%%%%%%%%%%%%%%%%%%%%%%%%%%%%%%
%%%%%%%%%%%%%%%%%%%%%%    Introduction    %%%%%%%%%%%%%%%%%%%%%%%%%%
%%%%%%%%%%%%%%%%%%%%%%%%%%%%%%%%%%%%%%%%%%%%%%%%%%%%%%%%%%%%%%%%%%%%%%%%

\section{Introduction} \label{s.intro}

In 1976 Ruzsa and Szemer\'edi introduced the Triangle Removal Lemma \cite{ruzsze78}, which roughly states that in a graph with not many triangles (copies of the complete graph on $3$ vertices), we can remove a small proportion of the edges to leave the graph with no copies of the triangle at all.\footnote{Recall that a (hyper)graph is composed by an ordered pair of sets; the vertex set and the set of edges, which are subsets of vertices. A hypergraph is said to be $k$-uniform if all the edges are subsets of $k$ vertices. A graph is a $2$-uniform hypergraph. A copy of a triangle is an injective map from $\{1,2,3\}$ to the vertex set of the graph that sends an edge of the triangle to an edge in the graph. To be precise, Ruzsa and Szemer\'edi showed the $(6,3)$--Theorem, which states the following: in any $3$-uniform hypergraph with $\delta n^2$ edges, there are $6$ vertices spanning (or inducing) $3$ edges if $n>n_0(\delta)$.} They gave a short proof of Roth's theorem \cite{rot53} as an application: any subset of the integers with positive (upper) density contains non-trivial $3$-term arithmetic progressions.\footnote{The $k$-term arithmetic progressions are the configurations $\{a,a+d,\ldots,a+(k-1)d\}$ with $a\in\Z$ and $d\in \Z_{\geq 0}$} 
The case for $k$-term arithmetic progressions, conjectured by Erd\H{o}s and Tur\'an \cite{tur41},  was established by Szemer\'edi in 1975 and is now called Szemer\'edi's Theorem.

	Both the Triangle Removal Lemma and Szemer\'edi's Theorem use Szemer\'edi's Regularity Lemma \cite{sze78} in their original proofs.\footnote{Szemer\'edi's Regularity Lemma states that the vertex set of any graph can be partitioned into finitely many parts, such that between most of the pairs, we have a quasi-random bipartite graph.} However, Szemer\'edi's Theorem does not seem to easily follow from the Triangle Removal Lemma (or the general Removal Lemma for Graphs \cite{erdfranrod86,fur95}.) Indeed, while the Triangle Removal Lemma follows almost immediately  from the Regularity Lemma, the proof of Szemer\'edi's Theorem is more involved and the Regularity Lemma is used in only one, yet crucial, step.

	In \cite{frarod02} Frankl and R\"odl showed that a $k$-uniform hypergraph version of the Removal Lemma suffices to establish the existence of non-trivial $(k+1)$-term arithmetic progressions in subsets of the integers with positive density.\footnote{The Removal Lemma for Hypergraphs states that if a given $k$-uniform hypergraph $K$ has few copies of a fixed $k$-uniform hypergraph $H$, then it can be made free of copies of $H$ by removing few edges. See Theorem~\ref{t.rem_lem_edge_color_hyper}.} The argument is similar to the one used by Ruzsa and Szemer\'edi \cite{ruzsze78} to show the $3$-term case using the Triangle Removal Lemma.\footnote{In this paper, the argument of Ruzsa and Szemer\'edi can be found in the proof of Theorem~\ref{t.multi_szem_ab_gr} but the construction of the hypergraph and the use of the Removal Lemma for Hypergraphs
has been substituted by the removal lemma for homomorphisms 
Theorem~\ref{t.rem_lem_ab_gr}.} This simple proof of Szemer\'edi's Theorem, 
along with the many applications that the Regularity Lemma for Graphs has (see
the surveys 
\cite{komsim96,komshosimsze02}), were the main motivations to extend the 
Regularity and the Removal lemmas from graphs to hypergraphs. This extension was 
done by several authors following different approaches: a combinatorial approach 
from Nagle, R\"odl, Schacht, and Skokan \cite{nagrodsch06,rodsko04}, an approach
using quasirandomness by Gowers \cite{gow07}, a probabilistic approach from Tao \cite{tao06}, and a non-standard measure-theoretic approach by Elek and Szegedy \cite{elesze12}.

\subsection{Arithmetic removal lemmas}

In \cite{gre05}, Green used Fourier analysis techniques to establish a regularity lemma and a removal lemma for linear equations in finite abelian groups. These statements are analogous to their combinatorial counterparts. The Group Removal Lemma \cite[Theorem~1.5]{gre05} ensures that for every $\varepsilon>0$ and every positive integer $m$, there exists a $\delta=\delta(\epsilon,m)>0$ such that, for any finite abelian group $G$, if
\begin{equation} \label{eq.1}
	x_1+\cdots+x_m=0
\end{equation} 
has less than $\delta |G|^{m-1}$ solutions with $x_i\in X_i\subset G$, then we can, by removing less than $\varepsilon |G|$ elements in each $X_i$, create sets $X_i''$ for which there is no solution to (\ref{eq.1}) with $x_i\in X_i''$ for all $i\in[1,m]$. 

Kr\'al', Serra and the author \cite{kraserven09} gave an alternative proof of the above removal lemma, showed by Green as \cite[Theorem~1.5]{gre05}, using the removal lemma for directed graphs \cite{alosha04}. With this alternative approach, the result was extended to any finite group, eliminating the need of commutativity.

While a removal lemma for linear systems for some $0/1$-matrices was shown to hold in \cite{kraserven09} using graphs, the work of Frankl and R\"odl \cite{frarod02} suggested that the hypergraph setting might provide the right tools to extend the removal lemma for one equation to a linear system.
% \footnote{Recall that the $k$-term arithmetic progressions can be codified as the solutions to a linear system with $(k-2)$ equations and $k$ variables. See equation \ref{eq.k-AP} on page \pageref{eq.k-AP}.}
Indeed,  Shapira~\cite{sha10}, and independently Kr\'al', Serra, and the author \cite{kraserven12}, used the Removal Lemma for Hypergraphs to obtain a removal lemma for linear systems over finite fields and proved a conjecture by Green \cite[Conjecture~9.4]{gre05} regarding a removal lemma for linear systems in the integers. 
A partial result for finite fields was obtained by Kr\'al', Serra, and the author \cite{kraserven08}, and also independently by Candela \cite{can_tesis_09}.

%Recently, additional versions of arithmetic removal lemmas have been given.

In addition to showing the removal lemma for finite fields,
%In \cite{sha10}, besides showing the removal lemma for finite fields, 
Shapira~\cite{sha10} raised the issue of whether an analogous result holds for linear systems over finite abelian groups. In \cite{ksv13}, %partially solving a question by Shapira in \cite{sha10},
 Kr\'al', Serra, and the author answered the question affirmatively provided that the determinantal\footnote{The determinantal of order $i$ of an integer matrix $A$ is the greatest common divisor of all the $i\times i$ submatrices of $A$. In this paper, the term \emph{determinantal} is used to refer to the determinantal of maximal order. See 
\cite[Chapter~II, Section~13]{new72} or 
Section~\ref{s.union_of_systems} for additional details and more of its properties.} of the integer matrix that defines the system is coprime to the cardinality of the finite abelian group. See \cite[Theorem~1]{ksv13} or Theorem~\ref{t.rem_lem_ksv13} for the result.

In a different direction, Candela and Sisask \cite{cansis13} proved that a removal lemma for integer linear systems holds over certain compact abelian groups. 
The main result in \cite{cansis13} has been recently extended by Candela, Szegedy and the author \cite{canszeven14+} to any compact abelian group provided that the integer matrix has determinantal $1$.

\paragraph{Previous combinatorial arguments.} 
The proof schemes of the previous arithmetic removal lemmas in \cite{kraserven09,kraserven12,ksv13,sha10} are inspired by the approaches of \cite{ruzsze78} and \cite{frarod02}, and can also be found in \cite{gow07,sol04,sze10,tao12}. Concisely, the main argument involves 
constructing a
 pair of graphs (or hypergraphs) $(K,H)$ so as to make it possible 
to transfer the removal lemma from the graph/hypergraph combinatorial setting to an arithmetic context. The pair $(K,H)$ is said to be a representation of the system and usually satisfies 
 the property that each copy of $H$ in $K$ is associated with a solution of our system. Moreover, these copies of $H$
should be evenly distributed throughout the edges of $K$.
The notion of representability of a  system by a hypergraph has been formalized by Shapira in \cite[Definition~2.4]{sha10}.

\subsection{Main notions and results} \label{s.main_notion_and_results}

This paper is built around two main pieces. The first one is the generalization of the combinatorial representability notion introduced by Shapira in \cite{sha10}. This generalization is stated as Definition~\ref{d.rep_sys}. The second piece is a removal lemma for homomorphisms of finite abelian groups; given a finite abelian group $G$ and a homomorphism $A$ from $G^m$ to $G^k$, if a set $X=X_1\times\cdots\times X_m$ does not contain many $\mathbf{x}\in X$ with $A\mathbf{x}=0$, then $X$ can be made solution-free by removing a small proportion of each $X_i$. The detailed statement can be found below as Theorem~\ref{t.rem_lem_ab_gr}. Additionally, we provide an interpretation of Theorem~\ref{t.rem_lem_ab_gr} in terms of subgroups, which is stated as Theorem~\ref{t.rem_lem_subgroups}, and can be seen as a removal lemma for finite abelian subgroups.

\subsubsection{Systems of configurations and representability}

\paragraph{Systems of configurations.}

Let us introduce the notion of a (finite) system of configurations.
Let $m$ be a positive integer and let $G$ be a set. A system of degree $m$ over $G$ consists of a pair $(A,G)$, where $A$ is a property on the configurations of $G^m$, $A: G^m\to \{0,1\}$. If
$x\in G^m$ is such that $A(x)=1$, then $x$ is said to be a \emph{solution} to $(A,G)$. In this paper all the sets $G$ considered are finite.

\paragraph{Representable systems.}

We introduce as Definition~\ref{d.rep_sys} a more general notion of representability for systems than the one given by Shapira in \cite[Definition~2.4]{sha10}. Although rather technical, Definition~\ref{d.rep_sys} asks for the existence of a pair of hypergraphs $(K,H)$, associated to a finite system  of configurations $(A,G)$,
with the following summary of properties (in parenthesis appear references to the properties described in Definition~\ref{d.rep_sys}):
\begin{itemize}
	\item Each copy of $H$ in $K$ is associated to a solution of $(A,G)$ (domain and range of $r$ in RP\ref{prop_rep2}.) The edges are associated with elements of $G$ (RP\ref{prop_rep1}, third point.) 
	\item Given a solution of $(A,G)$, there are many copies of $H$ in $K$ associated to such solution (cardinality of $r^{-1}(\mathbf{x},q)$ in RP\ref{prop_rep2}.) Those copies are well (evenly) distributed through the edges associated to the elements that configure the solution (RP\ref{prop_rep3} and RP\ref{prop_rep4}.)
	\item The number of vertices in $H$ is bounded (first point in RP\ref{prop_rep1}.)
\end{itemize}
All these points are sufficient to prove a removal lemma for representable systems, which has Theorem~\ref{t.rep_sys_rem_lem} as the precise statement.

Let us mention that the representations used in \cite{kraserven09,kraserven12,ksv13,sha10} can be seen to give a representation according to Definition~\ref{d.rep_sys}. In \cite{ksv13}, the authors used Shapira's definition, with an extra post-processing, to show a removal lemma for integer linear systems with determinantal $1$ and where the sets to be removed are small with respect to the total group, \cite[Theorem~1]{ksv13}.
The additional features of Definition~\ref{d.rep_sys} with respect to the representation notion given by Shapira in \cite{sha10}
 allow us to extend the result \cite[Theorem~1]{ksv13} (Theorem~\ref{t.rem_lem_ksv13} in this paper) to Theorem~\ref{t.rem_lem_ab_gr} in the following way. 
\begin{itemize}
	\item The set $Q$ is used to remove the determinantal condition and to extend the result to any homomorphism $A$ with domain $G^m$ and image in $G^k$.
	\item The vector $\gamma$ or, more precisely, the vector of proportions $(1/\gamma_1, \ldots, 1/\gamma_m)$, allows us to claim that the $i$-th removed set is an $\epsilon$-proportion of the projection of the whole solution set onto the $i$-th coordinate. The projection of the solution set is, in general, smaller than the whole abelian group $G$.\footnote{\label{f.footnote_7}For instance, the projection onto the first coordinate of the solution set of the equation $x_1+2(x_2+x_3)=0$, with $x_i\in \Z_6$, is isomorphic to $\Z_3$.} See Theorem~\ref{t.rep_sys_rem_lem} for further details.
\end{itemize}

\subsubsection{Removal lemma for representable systems.}

Let $I\subset \mathbb{N}$ be a set of indices. Consider $(\mathcal{A},\mathcal{G},m)=\{(A_i,G_i)\}_{i\in I}$ a family of systems of degree $m$. 
Let $S(A,G)$ denote the set of solutions for $(A,G)\in (\mathcal{A},\mathcal{G},m)$ and let $S(A,G,X)$ denote the subset of $\mathbf{x}\in S(A,G)$ with $\mathbf{x}\in X\subset G^m$.

Let $\Gamma=\{\gamma(i)\}_{i\in I}=\{(\gamma_1(i),\ldots,\gamma_m(i))\}_{i\in I}$ be a family of $m$-tuples of positive real numbers indexed by $I$. 
The family $(\mathcal{A},\mathcal{G},m)$ of systems is said to be $\Gamma$-representable, and the system $(A_i,G_i)$ is said to be associated with $\gamma_i$, if Definition~\ref{d.rep_sys} in Section~\ref{s.representables_systems} holds. The representability property suffices to show a removal lemma for configuration systems.

\begin{theorem}[Removal lemma for representable systems] \label{t.rep_sys_rem_lem}
	Let $(\mathcal{A},\mathcal{G},m)$ be a $\Gamma$-representable family of systems. Let $(A,G)$ be an element in the family associated to $\gamma=(\gamma_1,\ldots,\gamma_m)$. Let $X_1,\ldots,X_m$ be subsets of $G$ and let $X=X_1\times \cdots \times X_m$.
	
	For every $\varepsilon>0$ there exists a $\delta=\delta(m,\varepsilon)>0$, universal for all the members of the family, such that if 
	$$|S(A,G,X)|<\delta |S(A,G)|,$$ 
	then there are sets $X_i'\subset X_i$ with $|X_i'|<\varepsilon |G|/ \gamma_i$ for which $$ S\left(A,G,X\setminus X'\right)=\emptyset,$$ where
$X\setminus X'=(X_1\setminus X'_1)\times \cdots \times (X_m\setminus X'_m)$.
\end{theorem}

 Let us notice that, if 
a family of systems is $\Gamma$-representable, then the conclusions of Theorem~\ref{t.rep_sys_rem_lem} holds with smaller $\gamma_j(i)$'s as the restrictions on the theorem decrease with the $\gamma$'s. In the representability notion of Shapira~\cite{sha10}, as well as in other works like \cite{kraserven09,kraserven12,ksv13}, the $\gamma_j(i)$ are all $1$.
Thus, the notion of representability Definition~\ref{d.rep_sys}
is an extension of \cite[Definition~2.4]{sha10}.

\subsubsection{Removal lemma for finite abelian groups.}

Let $G$ be a finite abelian group. Given $\mathbf{b}\in G^k$, a homomorphism $A$ from $G^m$ to $G^k$ % $m\geq k$, 
induces a property $(A,\mathbf{b})$ in $G^m$ given by: $\mathbf{x}\in S((A,\mathbf{b}),G)$ if and only if $A(\mathbf{x})=b$. Let $(\mathcal{A},\mathcal{G},m)$ be the family of systems given by the homomorphisms $(A,\mathbf{b})$ with fixed $m$. These are the configuration systems
that we consider for most of the paper, especially in Section~\ref{s.equiv_and_rep} and onwards.

The set of homomorphisms $A:G^m\to G^k$ are in bijection with $k\times m$ homomorphism matrices $(\vartheta_{i,j})$ for some homomorphisms $\vartheta_{i,j}:G\to G$ depending on $A$.\footnote{See \cite[Section~13.10]{vandW91-2}} In particular, given $\mathbf{b}=(b_1,\ldots,b_k)^{\top}\in G^k$, $(x_1,\ldots,x_m)\in S\left(\left(A,\mathbf{b}\right),G\right)$ if and only if
\begin{displaymath}
\left(
	\begin{array}{ccc}
		\vartheta_{1,1} & \cdots & \vartheta_{1,m}  \\
		\vdots  &\ddots & \vdots \\
		\vartheta_{k,1} &  \cdots & \vartheta_{k,m} \\
		\end{array}\right)
		\left(\begin{array}{c}
		x_1 \\
		\vdots \\
			x_m \\
			\end{array}\right)	
	=	\left(\begin{array}{c}
			b_1 \\
			\vdots \\
				b_k \\
				\end{array}\right)
				\iff
\sum_{i=1}^m \vartheta_{j,i}(x_i)=b_j, \; \forall j\in [1,k].
\end{displaymath}
 Thus, we may use the term \emph{$k\times m$ homomorphism system on $G$} to refer to the system induced by a homomorphism from $G^m$ to $G^k$.

Let $S_i((A,\mathbf{b}),G)$, $i\in [1,m]$ denote the projection of the solution set $S((A,\mathbf{b}),G)$ to the $i$-th coordinate of $G^m$, $S_i((A,\mathbf{b}),G)=\pi_i(S((A,\mathbf{b}),G))$. The solution set $S((A,\mathbf{b}),G)$ is denoted by $S(A,G)$ when $\mathbf{b}=0$ or understood by the context.

In the sections \ref{s.proof_rl-lsg-1} and \ref{s.representability_product_cyclics}
of this paper, we show that the family of homomorphisms of finite abelian groups
is $\Gamma$-representable with $\gamma_i=|G|/|S_i((A,\mathbf{b}),G)|$ when $m\geq k+2$. Hence Theorem~\ref{t.rep_sys_rem_lem}, together with the additional comments to the construction presented in Section~\ref{s.finish_rem_lem_dkA1}, implies Theorem~\ref{t.rem_lem_ab_gr}.

\begin{theorem}[Removal lemma for linear systems over abelian groups] \label{t.rem_lem_ab_gr}
	 Let $G$ be a finite abelian group and let $m, k$ be two positive integers. Let $A$ be a group homomorphism from $G^m$ to $G^k$.
	Let $\mathbf{b}\in G^k$. Let $X_i\subset G$ for $i=[1,m]$, and $X=X_1\times \cdots \times X_m$. 
	
	For every $\varepsilon>0$ there exists a $\delta=\delta(m,\varepsilon)>0$ such that, if
 $$\left|S((A,\mathbf{b}),G,X)\right|<\delta \left|S((A,\mathbf{b}),G)\right|,$$ then there are sets $X_i'\subset X_i\cap S_i((A,\mathbf{b}),G)$ with $|X_i'|<\varepsilon |S_i((A,\mathbf{b}),G)|$ and $$S\left((A,\mathbf{b}),G, X\setminus X'\right)=\emptyset, \text{ where } X\setminus X'=(X_1\setminus X'_1)\times \cdots \times (X_m\setminus X'_m).$$

Also, let $I\subset [1,m]$ be such that $X_i\supset S_i((A,\mathbf{b}),G)$ for $i\in I$. The previous statement holds with the extra condition that $X_i'=\emptyset$ for $i\in I$.
\end{theorem}

Let us state the known arithmetic removal lemma for finite abelian groups \cite[Theorem~1]{ksv13}.
\begin{theorem}[Removal lemma for finite abelian groups, Theorem~1 in \cite{ksv13}]\label{t.rem_lem_ksv13}
		 Let $G$ be a finite abelian group and let $m, k$ be two positive integers. Let $A$ be a $k\times m$ integer matrix with determinantal coprime with the order of the group $|G|$.\footnote{The determinantal can be assumed to be $1$. This restriction implies that $\left|S((A,\mathbf{b}),G)\right|= |G^{m-k}|$, which is not true in the general case, as Observation~\ref{o.number_of_solutions}, or footnote \ref{f.footnote_7} shows.}
		%, $m\geq k$. 
		Let $\mathbf{b}\in G^k$. Let $X_i\subset G$ for $i=[1,m]$, and $X=X_1\times \cdots \times X_m$. For every $\varepsilon>0$ there exists a $\delta=\delta(m,\varepsilon)>0$ such that, if
	 $$\left|S((A,\mathbf{b}),G,X)\right|<\delta \left|S((A,\mathbf{b}),G)\right|=\delta |G^{m-k}|,$$ then there are sets $X_i'\subset X_i\cap S_i((A,\mathbf{b}),G)$ with $|X_i'|<\varepsilon |G|$ and $$S\left((A,\mathbf{b}),G, X\setminus X'\right)=\emptyset,
	\text{ where } X\setminus X'=(X_1\setminus X'_1)\times \cdots \times (X_m\setminus X'_m).$$
\end{theorem}

We can see that Theorem~\ref{t.rem_lem_ab_gr} extends Theorem~\ref{t.rem_lem_ksv13} in three ways: 
\begin{itemize}
	\item The coprimality condition between the determinantal of $A$ 
	and the order of the group is not needed.
	\item The systems induced by homomorphisms are more general than the systems induced by integer matrices. In particular, if $G=\prod_{i=1}^t\Z_{n_i}$, $n_{i+1}|n_i$, and we let $x=(x_1,\ldots,x_t)\in G$ with $x_i\in \Z_{n_i}$, then the homomorphism systems allow for linear equations between the components $x_i$ and $x_j$. This fact is used 
	to prove the multidimensional version of Szemer\'edi's Theorem. See Section~\ref{s.intro_applic}.
	\item The sizes of the deleted sets $X_i'$ are an $\varepsilon$-proportion of $|S_i((A,\mathbf{b}),G)|$ and not of $|G|$. 
	In particular, if $|S_i((A,\mathbf{b}),G)|=1$, $\varepsilon<1$, and $X_i$ contains that element, then $X_i'=\emptyset$. 
	This makes the result best possible in the following sense: if $S((A,\mathbf{b}),G,X)>\delta |S((A,\mathbf{b}),G)|$, then, in order to delete all the solutions, there should exist an $i$ with $|X_i'|> |S_i((A,\mathbf{b}),G)| \delta/m$. Therefore, we remove, at most, an $\varepsilon$-proportion of the right order of magnitude.
	\item The set of variables $x_i$ with $X_i=S_i(A,G)$ and for which no element should be removed is arbitrary. The argument leading to \cite[Theorem~1]{ksv13}, allows the existence of a set of indices $I$ of full sets from which no element is removed. However, the argument from \cite{ksv13} imposes an upper bound on the size of $I$. The argument in Section~\ref{s.adding_variables} remove those bounds on $I$.\footnote{The argument of Section~\ref{s.adding_variables} and of Observation~\ref{lem:red2_ext_3} could be applied to \cite[Theorem~1]{ksv13} to add this extra property.}
\end{itemize}

\subsubsection{Removal lemma for finite abelian subgroups}

As Theorem~\ref{t.rem_lem_ab_gr} can be applied to any finite abelian group $G$ and any homomorphism, we can rephrase the result in terms of subgroups.
\begin{theorem}[Removal lemma for subgroups] \label{t.rem_lem_subgroups}
	 For every $\epsilon>0$ and every positive integer $m$, there exists a $\delta=\delta(\epsilon,m)>0$ such that the following holds. Let $G_1,\ldots,G_m$ be finite abelian groups. Let $S$ be a subgroup of $G_1\times \cdots\times G_m$ and let $s\in \prod_{i\in[m]} G_i$.
	Let $X_i$ be a subset of $G_i$ for each $i\in[1,m]$. If
	$|\left[X_1\times\cdots\times X_m\right] \cap s+S|<\delta |S|$ then there exist
	$X_1',\ldots,X_m'$, with $|X_i'|<\epsilon \pi_i (S)$ for all $i\in[1,m]$, such that
	$\{\left[X_1\setminus X_1' \times\cdots\times X_m\setminus X_m'\right] \cap s+S\}=\emptyset$. If $\pi_i (s+S)\subset X_i$, for some $i\in[1,m]$ then we can assume
	$X_i'=\emptyset$.
\end{theorem}
Indeed, any subgroup of a finite abelian group is the kernel of a homomorphism (namely the quotient map $\prod_{i\in[m]} G_i\to \left[\prod_{i\in[m]} G_i\right]/S$, $S$ being our subgroup of interest.) Notice also that, instead of $\prod_{i\in[m]} G_i$, we could consider our domain to be a supergroup $G^m>\prod_{i\in[m]} G_i$, for some suitable finite abelian group $G$, as $S$ is also a subgroup of $G^m$.
Moreover, we can assume $G^m/S<G^k$ for some $k$ (take, for instance $k=m$.) Therefore, Theorem~\ref{t.rem_lem_ab_gr} suffices to show Theorem~\ref{t.rem_lem_subgroups}. Since the kernel of a homomorphism generates
a subgroup of $G^m$, Theorem~\ref{t.rem_lem_subgroups} implies Theorem~\ref{t.rem_lem_ab_gr}. Hence the version of the result for subgroups
Theorem~\ref{t.rem_lem_ab_gr}, and the version for systems of homomorphisms Theorem~\ref{t.rem_lem_subgroups}, are equivalent.

\subsubsection{Removal lemma for permutations} \label{s.intro_permutations}

In \cite{cooper06}, Cooper introduced a regularity lemma and a removal lemma for permutations. Let $\mathcal{S}(i)$ denote the set of bijective maps from $[0,i-1]$ to $[0,i-1]$. Slightly modifying the notation in \cite{cooper06}, let $\Lambda^{\tau}(\sigma)$, for $\tau\in \mathcal{S}(m)$ and $\sigma\in \mathcal{S}(n)$, be the set of occurrences of the pattern $\tau$ in $\sigma$. That is to say, the set of index sets $\{x_0<\cdots < x_{m-1}\}\subset [0,n-1]$
such that $\sigma(x_i)<\sigma(x_j)$ if and only if $\tau(i)<\tau(j)$.

\begin{proposition}[Proposition~6 in \cite{cooper06}]\label{p.rem_lem_permutations}
	Suppose that $\sigma\in \mathcal{S}(n)$, $\tau\in \mathcal{S}(m)$. For every $\epsilon>0$ there exist a $\delta=\delta(\epsilon,m)>0$ such that, if $|\Lambda^{\tau}(\sigma)|<\delta n^m$, then we may delete at most $\epsilon n^2$
	index pairs to destroy all copies of $\tau$ in $\sigma$.
\end{proposition}

In Section~\ref{s.perm_rep_and_rem_lem} we can find a representation where the valid configurations are given by the set $\Lambda^{\tau}(\sigma)$, and where we shall delete pairs of indices to destroy them. Hence Proposition~\ref{p.rem_lem_permutations} follows from Theorem~\ref{t.rep_sys_rem_lem}.

\subsection{Applications} \label{s.intro_applic}

\paragraph{Multidimensional Szemer\'edi.}

One of the main applications of Theorem~\ref{t.rem_lem_ab_gr} is a new proof
of the counting version of the multidimensional Szemer\'edi's Theorem for finite abelian groups.

 The original proof of the multidimensional Szemer\'edi theorem for the integers was given by Furstenberg and Katznelson \cite{furskatz78} and uses ergodic theory. Solymosi \cite{sol04} observed that a removal lemma for hypergraphs would imply the multidimensional Szemer\'edi theorem (a detailed construction can be found in \cite{gow07}). Solymosi's geometric argument uses hypergraphs and follows the lines of the argument by Ruzsa and Szemer\'edi \cite{ruzsze78} to obtain Roth's Theorem \cite{rot53} from the Triangle Removal Lemma. With the development of the Regularity Method for Hypergraphs \cite{gow07,nagrodsch06,rodsko04,tao06} and the corresponding Removal Lemma for Hypergraphs \cite{gow07,nagrodsch06,tao06}, 
a combinatorial proof of the multidimensional version of Szemer\'edi's Theorem for the integers could be pushed forward \cite{gow07,sol04}.

In \cite{tao12}, Tao uses the same construction as Solymosi \cite{sol04} to show \cite[Theorem~B.2]{tao12} and a lifting trick to obtain a generalized version of the multidimensional Szemer\'edi theorem for finite abelian groups, \cite[Theorem~B.1]{tao12}.
 Theorem~\ref{t.rem_lem_ab_gr} can be used to prove both \cite[Theorem~B.1]{tao12} and \cite[Theorem~B.2]{tao12}. 

The argument to deduce \cite[Theorem~B.2]{tao12} (Theorem~\ref{t.multi_szem_ab_gr} in this paper) from Theorem~\ref{t.rem_lem_ab_gr} explicitly shows that the dependencies in \cite[Theorem~B.2]{tao12} and in \cite[Theorem~B.1]{tao12} are independent of the dimension of the space and depend only on the number of points required in the configuration. On the other hand, the relation between $\delta$ and $\epsilon$ obtained using Theorem~\ref{t.rem_lem_ab_gr} is worse than the direct construction of \cite[Theorem~B.2]{tao12} due to the larger uniformity of the hypergraph used.

\begin{theorem}[Multidimensional Szemer\'edi for finite abelian groups, Theorem~B.2 in \cite{tao12}] \label{t.multi_szem_ab_gr}
	Let $\varepsilon >0$. Let $G^m$ be a finite abelian group and let
	$S\subset G^m$ be such that $|S|/|G^m|\geq \varepsilon$. There exists $\delta=\delta(\varepsilon,m+1)>0$ such that the number of configurations of the type
	$\{(x_1,\ldots,x_m),(x_1+a,x_2,\ldots,x_m),\ldots,(x_1,x_2,\ldots,x_m+a)\}\subset S$, for some $a\in G$, is at least $\delta |G|^{m+1}$.
\end{theorem}

\begin{proof}[Proof of Theorem~\ref{t.multi_szem_ab_gr}]
	Consider the abelian group $P=G^m$, $X_i=S\subset P$, for $i\in[1,m+1]$. Let $\mathbf{x}_i=(x_{i,1},\ldots,x_{i,m})$, $i\in[1,m+1]$, be the variables of the homomorphism system that can be derived from the following linear equations:
	\begin{equation}\label{eq.system_multidim}
		\left\{\begin{array}{cl}
		x_{1,1}-x_{2,1}= x_{1,j}-x_{j+1,j} & \text{ for } j\in[2,m] \\
		x_{1,j}=x_{i,j} & \text{ for all }(i,j)\in[1,m+1]\times[1,m], i\neq j+1.\\
	\end{array}\right. 
	\end{equation}
Indeed, $\mathbf{x}_1$ is thought of as the centre of the configuration. The first equations state that the difference between the $j$-th coordinate of $\mathbf{x}_{j+1}$ and the $j$-th coordinate of $\mathbf{x}_1$ is the same regardless of $j$; this is achieved by setting all the differences to be equal
to the difference between the first coordinate of $\mathbf{x}_2$ and the first coordinate of $\mathbf{x}_1$. The second set of equations treat the other coordinates, imposing that all the other coordinates of $\mathbf{x}_{j+1}$, except the $j$-th, should be equal to those of $\mathbf{x}_1$. Therefore, $(\mathbf{x_1},\ldots,\mathbf{x}_{m+1})$ is a solution to the system defined by 
(\ref{eq.system_multidim}) if and only if $\mathbf{x}_1=(y_1,\ldots,y_m)$, 
$\mathbf{x}_2=(y_1+a,\ldots,y_m)$, \ldots, $\mathbf{x}_{m+1}=(y_1,\ldots,y_m+a)$
for some $y_1,\ldots,y_m,a\in G$.

By adding some trivial equations, like $0=0$, the system induces a homomorphism $A:P^{m+1} \to P^{m}$, with $S(A,P)\cong G^{m+1}$. Observe that $S_i(A,P)\cong G^m$ as any point in $P=G^m$ can be the $i$-th element in the configuration. 

Consider the $\delta=\delta_{\text{Theorem~\ref{t.rem_lem_ab_gr}}}(m+1,\varepsilon/(m+1))$ coming from Theorem~\ref{t.rem_lem_ab_gr} applied with $\varepsilon/(m+1)$ and $m+1$.
Let us proceed by contradiction and assume that the number of solutions is less than $\delta S(A,P)=\delta |G|^{m+1}$. Now we apply Theorem~\ref{t.rem_lem_ab_gr} and find sets $X_1',\ldots,X_{m+1}'$ with $|X_i'|<|G^m|\varepsilon/(m+1)$ such that the sets $X_i=S\setminus X_i'$ bear none of the desired configurations.

Observe that any point $\mathbf{x}\in S\subset G^m$ generates a solution to the linear system as $(\mathbf{x},\ldots,\mathbf{x})\in P^{m+1}$ is a valid configuration with $a=0_{G}$. Consider $S'=S\setminus (\cup X_i')$. Since $|S|>\varepsilon|G^m|$ and $|X_i'|<|G^m|\varepsilon/(m+1)$, there exists an element $\mathbf{s}$ in $S'$, as $S'$ is non-empty. Therefore $\mathbf{s}\in X_i$ for every $i\in[1,m+1]$. Thus $(\mathbf{s},\ldots,\mathbf{s})\in P^{m+1}$ is a solution that still exists after removing the sets $X_i'$ of size at most $\varepsilon/(m+1)$ from $S$. This contradicts Theorem~\ref{t.rem_lem_ab_gr}. Therefore, we conclude that at least $\delta |G|^{m+1}$ solutions exist.
\end{proof}

\paragraph{Other linear configurations.}

More generally, we can show the following corollary of Theorem~\ref{t.rem_lem_ab_gr}.

\begin{corollary}\label{c.homotetic_solutions2}
	Let $G$ be a finite abelian group, let $A$ be a $k\times m$ homomorphism for $G$ and let $\mathbf{b}\in G^k$. Assume that $S(A,G)=S((A,\mathbf{b}),G)\subset G^m$ contains a set $R$ satisfying the following conditions.
	\begin{enumerate}[(i)]
	\item The projection of $R$ onto the $i$-th coordinate of $G^m$  is $S_i(A,G)$. This is, $\pi_i(R)=S_i(A,G)$.
	\item For each $i\in[1,m]$ and for each pair $g_1,g_2\in S_i(A,G)$,
	$|\pi_i^{-1}(g_1)\cap R|=|\pi_i^{-1}(g_2)\cap R|$.
\end{enumerate}
 Then, for every $\varepsilon>0$ there exists a $\delta=\delta(\varepsilon,m)>0$ such that, for any $S\subset G$ with $|S^m\cap R|>\varepsilon |R|$, we have $\left|S(A,G,S^m)\right|\geq\delta \left|S(A,G)\right|$.
\end{corollary}

\begin{proof}
	We proceed by contradiction.
Choose $\delta=\delta_{\text{Theorem~\ref{t.rem_lem_ab_gr}}}(m,\varepsilon/(m+1)$ and assume that $\left|S(A,G,S^m)\right|<\delta \left|S(A,G)\right|$. Then there are sets $X_i'$, with $|X_i'|< \varepsilon/(m+1)$, such that $S(A,G,\prod_{i=1}^m S\setminus X_i')=\emptyset$. However, by $(i)$ and $(ii)$ we delete at most $\epsilon\frac{m}{m+1}|R|$ hence $R\cap \prod_{i=1}^m S\setminus X_i'\neq \emptyset$, thus $S(A,G,\prod_{i=1}^m S\setminus X_i')\neq \emptyset$ reaching a contradiction.
\end{proof}

In particular, if the linear system $(A,G)$ satisfies $S_i(A,G)=G$ for all $i\in[1,m]$ and $(x,\ldots,x)\in S(A,G)$ for each $x\in G$, then Corollary~\ref{c.homotetic_solutions} shows that any set $S\subset G$ with $|S|\geq \epsilon |G|$, satisfies that $|S(A,G,S^m)|>\delta |S(A,G)|$ for some $\delta>0$ depending on $\epsilon$ and $m$. That is, any set with positive density will contain a positive proportion of the solutions.
Corollary~\ref{c.homotetic_solutions2} can be particularized as
Corollary~\ref{c.homotetic_solutions} which presents a perhaps more directly applicable form.

\begin{corollary}\label{c.homotetic_solutions}
	Let $G$ be a finite abelian group, let $G_1,\ldots G_s$ be subgroups of $G$. Let $\Phi_1,\ldots,\Phi_t$ be group homomorphisms
\begin{align}
		\Phi_i: G_1\times \cdots\times G_s &\to G \nonumber \\
		(x_1,\ldots,x_s) &\mapsto \Phi_i(x_1,\ldots,x_s). \nonumber
\end{align}
For every $\epsilon>0$ there exists a $\delta=\delta(\epsilon,t)>0$
such that, for every $S\subset G$ with $S\geq \epsilon |G|$,
\begin{multline}
	\left|\left\{x\in G, \mathbf{x}\in \prod_{i=1}^s G_i \; |\; (x+\Phi_1(\mathbf{x}),
	\ldots,x+\Phi_t(\mathbf{x}))\in S^t\right\}\right|>\\ 
	\delta 
	\left|\left\{x\in G, \mathbf{x}\in \prod_{i=1}^s G_i  \;|\; (x+\Phi_1(\mathbf{x}),
	\ldots,x+\Phi_t(\mathbf{x}))\in G^t\right\}\right|. \nonumber
\end{multline}
\end{corollary}

\begin{proof}
	Consider $R=\{(x,\ldots,x)\}_{x\in G}$.
	Observe that the configuration set $\{x\in G, \mathbf{x}\in \prod_{i=1}^s G_i  \;|\; (x+\Phi_1(\mathbf{x}),
	\ldots,x+\Phi_t(\mathbf{x}))\in G^t\}$ is a subgroup of $G^t$, whence there exists a homomorphism $A$ such that $S(A,G)$ is the configuration set and, in this case, $S_i(A,G)=G=\pi_i(R)$ for all $i$. Thus the hypotheses of Corollary~\ref{c.homotetic_solutions2} are fulfilled and the result follows.
\end{proof}

Corollary~\ref{c.homotetic_solutions} encompasses the simplex-like configurations from the multidimensional version of Szemer\'edi's theorem with the evaluation: $G=\Z_p^k$, $G_1=\Z_p$, $\Phi_t=0$ and $\Phi_1,\ldots,\Phi_{k}$ being the coordinate homomorphisms 
\begin{align}
	\Phi_i: \Z_p&\to \Z_p^t \nonumber \\
	x&\mapsto (0,\ldots,0,\overbrace{x}^{i},0,\ldots,0). \nonumber
\end{align}

Additionally Corollary~\ref{c.homotetic_solutions} generalizes \cite[Theorem~B.1]{tao12} which asserts that given a finite abelian group $G$, for every $\epsilon>0$ and $t,m$ positive integers, there are, in any set $S\subset G^m$ with $|S|>\epsilon |G|^m$,
 $\delta(\epsilon,t,m) |G|^{m+1}$ configurations
\begin{displaymath}
	\{y\in G^m, x\in G \; | \; (y+\Phi_1(x),\ldots,y+\Phi_{(2t+1)^m}(x))\in S^{(2t+1)^m}\}
\end{displaymath}
with 
\begin{align}
	\Phi_i(x)=(\chi_1(i) x,\ldots,\chi_m(i) x) \nonumber
\end{align}
where $(\chi_1(i),\ldots,\chi_m(i))$ are the components of $i$ in base $2t+1$ shifted by $-t$ so their values lie in $[-t,t]$ instead of the usual $[0,2t]$.
An example of an extra configuration that Corollary~\ref{c.homotetic_solutions} covers are the ``rectangles'' $(x,x+x_1,x+x_2,x+x_1+x_2)\in S^4$, for $S\subset G=\Z_3^n$, with $x_1\in G_1$ and $x_2\in G_2$ two subgroups of $G$, isomorphic to $\Z_3^{n-\log(n)}$ and $\Z_3^{\sqrt{n}}$ respectively, and such that $G_1+G_2=G$.

The arguments to show Corollary~\ref{c.homotetic_solutions}, Corollary~\ref{c.homotetic_solutions2}, or Theorem~\ref{t.multi_szem_ab_gr} exemplify that Theorem~\ref{t.rem_lem_ab_gr} presents a comprehensive approach to the asymptotic counting of homothetic-to-a-point structures found in dense sets of products of finite abelian groups. More precisely, the constants involved in the lower bound of the number of configurations depend only on the number of points of the configuration  and on the density of the set, but not on the configuration itself nor on the structure of the finite abelian group.

If we ask for configurations in the integers, the constant does depend on the configuration as we are not interested in solutions that occur due to the cyclic nature of the components of the finite abelian group. Therefore, we should reduce the density of the sets to allow only the desired solutions. This affects the total number of configurations found in the finite abelian group.

\paragraph{Monochromatic solutions.}

Theorem~\ref{t.rem_lem_ab_gr} also allows us to extend the results in \cite{serven14} regarding a counting statement for the monochromatic solutions of bounded torsion groups. In particular, we ensure that there are $\Omega\left(|S(A,G)|\right)$  monochromatic solutions, thus improving the asymptotic behaviour $\Omega\left(|G|^{m-k}\right)$ stated in \cite{serven14}. Here $S(A,G)$ represents the solution set of $A\mathbf{x}=0$, $\mathbf{x}\in G^m$, when $A$ is a $k\times m$ full rank integer matrix and the asymptotic behaviour depends on the number of colours.

\paragraph{Hypergraph containers.}

Using the hypergraph containers tools from \cite{saxtho13+},
 Theorem~\ref{t.rem_lem_ab_gr} can be used to extend  \cite[Theorem~10.3]{saxtho13+} or \cite[Theorem~2.10]{saxtho13+2}, regarding the number of subsets free of solutions of a given system of equations, and show for instance Theorem~\ref{t.saxtom+}, where homomorphism systems are considered. Following the notation in \cite{saxtho13+}, a homomorphism system $A$ is said to be \emph{full rank} if there exists a solution to $Ax=\mathbf{b}$ for any $\mathbf{b}\in G^k$. A full rank $k\times m$ homomorphism system $A$ (or with coefficients over a finite field) is said to be \emph{abundant} if any $k\times m-2$ subsystem of $A$ formed using $m-2$ columns of homomorphisms also has full rank. Given a set $Z\subset G^m$ of discounted solutions and $\mathbf{b}\in G^k$, a set $X\subset G$ is said to be $Z$-solution-free if there is no $x\in X^m-Z$ with $Ax=\mathbf{b}$. Let $\text{ex}(A,\mathbf{b},G)$ denote the size of the maximum $Z$-solution-free set.

\begin{theorem}[Saxton, Thomason, Theorem~10.3/2.10 in \cite{saxtho13+}/\cite{saxtho13+2} with Theorem~\ref{t.rem_lem_ab_gr}] \label{t.saxtom+}
	Let $\{G_i\}_{i\in I}$ be a sequence of finite abelian groups. Let $A_i$ be a sequence of abundant $k\times m$ homomorphism systems and $\mathbf{b}_i\in G_i^k$ a sequence of independent vectors such that $|S((A_i,\mathbf{b}_i),G_i)|=|G_i|^{m-k}$. Let 
	$Z$ be such that $Z\subset S((A_i,\mathbf{b}_i),G_i)$ and $|Z|=o(|G_i|^{m-k})$. Then the number of $Z$-solution-free subsets of $G_i$ is $2^{\text{ex}(A_i,\mathbf{b}_i,G_i)+o(|G_i|)}$.
\end{theorem}

\subsection{Outline of the paper}

The main results of the paper are Theorem~\ref{t.rep_sys_rem_lem} and Theorem~\ref{t.rem_lem_ab_gr}. To prove Theorem~\ref{t.rep_sys_rem_lem}, we observe that the notion of representation, Definition~\ref{d.rep_sys}, is sufficient to transfer the hypergraph removal lemma, Theorem~\ref{t.rem_lem_edge_color_hyper} in this paper, to the representable setting. The argument can be found in Section~\ref{s.representables_systems}.
Some examples of representable systems and their correspondent removal lemmas are presented.

In Section~\ref{s.equiv_and_rep} we introduce the notion of $\mu$-equivalent linear systems (see Section~\ref{s.equiv_systems}). In Section~\ref{s.oper_between_representable}, we show some relations between the representability of the systems $(A_1,G_1)$ and $(A_2,G_2)$ whenever $(A_2,G_2)$ is $\mu$-equivalent to $(A_1,G_1)$. These results are used in the proof of Theorem~\ref{t.rem_lem_ab_gr}. Indeed, the strategy of the proof can be summarized as finding a suitable sequence of $\mu$-equivalent systems, from the system of our interest, to a representable one. As Section~\ref{s.oper_between_representable} shows, we can then find a representation for our original system.

In Section~\ref{s.proof_rl-lsg-1} and Section~\ref{s.representability_product_cyclics} we prove Theorem~\ref{t.rem_lem_ab_gr} by arguing that the systems involved in the statement of the theorem are representable. 
Section~\ref{s.finish_rem_lem_dkA1} is devoted to show the cases where $m\leq k+1$ and to prove the second part of the result 
involving the sets $X_i$ for which $X_i\supset\pi_i(S(A,G))$.

The sketch of the construction for the representation is as follows.
Given $G=\prod_{i=1}^t \Z_{n_i}$ with $n=n_1$ and $n_i|n_j$ for $i\geq j$, we interpret the homomorphism $A:G^m\to G^t$ as a homomorphism $A'$ from  $(\Z_{n}^t)^m$ to $(\Z_{n}^t)^k$ 
in a natural way. Then any solution of $S(A,G)$ is related to $|S(A',\Z_{n}^t)|/|S(A,G)|$ solutions of $S(A',\Z_{n}^t)$. This reduction process is detailed in Section~\ref{s.proof_rl-lsg-1}.

As Section~\ref{s.hom_mat_to_integer_mat} shows,
the homomorphism matrix $A'$ can be thought of as an integer matrix from $\Z_{n}^{t m}$ to $\Z_{n}^{t k}$ with $t m$ variables and $t k$ equations in $\Z_{n}$.
This interpretation as an integer matrix allows for the construction of the representation by using the ideas in the proof of \cite[Lemma~4]{ksv13}.
The construction is detailed in Section~\ref{s.representability_product_cyclics} 
and involves several transformations to the pair $(A',G')$ to address the 
different issues like the determinantal being larger than $1$. The main 
characteristics of those transformations are described in the statements of 
Section~\ref{s.oper_between_representable}. The $\Gamma$-representability with $\gamma_j(i)> 1$ involves the generation of several systems. The construction of such systems is detailed in Section~\ref{s.gamma-effective} and they are combined in Section~\ref{s.final_composition} to create a single $1$-strongly-representable system.\footnote{See Definition~\ref{d.rep_sys} for the additional conditions of the strong-representability.} A summary of all the transformations can be found in a table in Section~\ref{s.unwrap_const}.

%%%%%%%%%%%%%%%%%%%%%%%%%%%%%%%%%%%%%%%%%%%%%%%%%%%%%%%%%%%%%%%%%%%%%%%%
%%%%%%%%%%%%%%%%%%      Representable systems      %%%%%%%%%%%%%%%%%%%% 
%%%%%%%%%%%%%%%%%%%%%%%%%%%%%%%%%%%%%%%%%%%%%%%%%%%%%%%%%%%%%%%%%%%%%%%%

\section{Representable systems}\label{s.representables_systems}

In this work $[a,b]$ stands for the integers between $a$ and $b$, both included. If $\mathbf{x}\in G^m$, then $(\mathbf{x})_i$ denotes the $i$-th component of $\mathbf{x}$.
Let us recall some notions regarding hypergraphs. Given a hypergraph $K=(V,E)$, $V=V(K)$ denotes the vertex set, $E=E(K)$ denotes the edge set and $|K|=|V(K)|$ denotes the size of the vertex set. A hypergraph $K$ with vertex set $V=V(K)$ and edge set $E=E(K)$ is said to be $s$-uniform
if each edge in $E$ contains precisely $s$ vertices. Throughout this paper, we consider hypergraphs with edges coloured by integers. A hypergraph $K$ is said to be $m$-coloured if each edge in $K$ bears a colour in $[1,m]$.  If $K$ is an $m$-coloured hypergraph, $E_i(K)$ denotes the set of edges coloured $i\in[1,m]$ in $K$.
By a copy of $H$ in $K$ we understand an injective homomorphism of colored hypergraphs of $H$ into $K$ respecting the colors of the edges (the map is from vertices to vertices, injective, and maps edges colored $i$ to edges with color $i$). We use $C(H,K)$ to denote the set of colored copies of $H$ in $K$. If $H$ has $m$ edges $\{e_1,\ldots,e_m\}$ with $e_i$ colored $i$ then $H$ can be identified with $(e_1,\ldots,e_m)$.

\subsection{Representability} \label{s.rep_tech_def}

The definition of a representable system, Definition~\ref{d.rep_sys}, is a generalized notion of the one formalized in \cite{sha10} that suffices to obtain a removal lemma; in our case Theorem~\ref{t.rep_sys_rem_lem}. These representability notions have been used in several works like \cite{can_tesis_09,kraserven08,kraserven09,kraserven12,ksv13,sha10,sze10} 
to translate the conclusion of the removal lemma for graphs or hypergraphs to linear systems of equations.
 The representable system notion could potentially be used in more general contexts than the homomorphism systems described in this work.

Recall that a system $(A,G)$ is a pair given by a finite set $G$ and a property $A:G^m\to \{0,1\}$. $S(A,G)$ denotes the preimage of $1$ by $A$. $\gamma$ denotes a tuple of $m$ positive real numbers $(\gamma_1,\ldots,\gamma_m)$.
$(\mathcal{A},\mathcal{G},m)$ denotes a family of systems and $\Gamma$ a collection of $\gamma$'s, one for each system.

\begin{definition}[(strongly) representable system] \label{d.rep_sys}
	The family of finite systems $(\mathcal{A},\mathcal{G},m)$ 
	is said to be $\Gamma$-representable if there are positive real numbers $\chi_1, \chi_2$, depending on the family $(\mathcal{A},\mathcal{G},m)$, and for each $(A,G)\in (\mathcal{A},\mathcal{G},m)$ and the $\gamma=(\gamma_1,\ldots,\gamma_m)\in \Gamma$ associated with $(A,G)$, there exists a pair of coloured hypergraphs $(K,H)$ 
	with the following properties RP\ref{prop_rep1}, RP\ref{prop_rep2} and RP\ref{prop_rep3}.
\begin{enumerate}[RP1]
	\item \label{prop_rep1} 
	\begin{itemize}
		\item $K$ and $H$ are $s$-uniform $m$-colored hypergraphs. 
		\item $H$ has $m$ different edges $\{e_1,\ldots,e_m\}$ and the edge $e_i$ is coloured $i$. Moreover  $\chi_1\geq|V(H)|=h>s\geq2$. 
		\item Each edge in $K$ bears a label in $G$ given by $l:E(K)\to G$. 
	\end{itemize}

\item \label{prop_rep2} There exist a positive integer $p$, a set $Q$, and 
a 
surjective map $r$
	\begin{align}
		r:C(H,K) &\longrightarrow  S(A,G)\times Q\nonumber \\
		H=\{e_1,\ldots,e_m\}&\longmapsto (r_0(H),r_q(H)) \nonumber
	\end{align}
	such that $r_0(H)=(l(e_1),\ldots,l(e_m))$, and, for any given $\mathbf{x}\in S(A,G)$ and $q\in Q$, 
the set $r^{-1}(\mathbf{x},q)$ has size 
 		 \begin{displaymath}
 		 	|r^{-1}(\mathbf{x},q)|=p \lambda \prod_{i=1}^m \gamma_i \text{ with } \lambda= c\frac{|K|^s}{|G|} 
 		 \end{displaymath}
 		for some $c\geq \chi_2$.
\item \label{prop_rep3} If $e_i$ is an edge coloured $i$ in a copy $H \in r^{-1}(\mathbf{x},q)$, then $p\frac{\prod_{j=1}^m\gamma_j}{\gamma_i}$ copies of $H$ in $r^{-1}(\mathbf{x},q)$ contain $e_i$.

\end{enumerate}
If, additionally, 
\begin{enumerate}[RP1]
	  \setcounter{enumi}{3}
	\item \label{prop_rep4} For any edge $e_i$ coloured $i$ and $l(e_i)=\mathbf{x}_i$, there exists a copy of $H\in r^{-1}(\mathbf{x},q)$, with $(\mathbf{x})_i=\mathbf{x}_i$, containing $e_i$,
\end{enumerate}
then the family is said to be strongly $\Gamma$-representable.
\end{definition}

If $H\in r^{-1}(\mathbf{x},q)$ we say that $H$ is related to $\mathbf{x}$ through $q$. 
If a system $(A,G)$ belongs to a $\Gamma$-representable family of systems and has $\gamma$ as its associated parameters then $(A,G)$ is said to be $\gamma$-representable. If $\gamma_1=\cdots=\gamma_m=1$ we say that the system is $1$-representable. The vector $(K,H,\gamma,l,r,Q,p,c)$ defines the $\gamma$-representation and the key parameters are $\chi_1$ and $\chi_2$.

\paragraph{Comments on Definition~\ref{d.rep_sys}.} \label{s.rep_oper_def}
In the definition, the hypergraphs $H$ and $K$ could have also been asked to be directed.

By choosing $Q=\{1\}$, $p=1$ and $\gamma_1=\cdots=\gamma_m=1$ for all the systems $(A,G)$, Definition~\ref{d.rep_sys} covers the representation notions in \cite{kraserven09,ksv13,kraserven12,sha10}. 
The main purpose of the introduction of the set $Q$ is to accommodate the determinantal condition from \cite[Theorem~1]{ksv13}. The different $p$ and $\gamma$ allow for removing different proportions for different sets $S_i(A,G)$, the projections of the solution set to the coordinates of $G^m$.

Asking for the bounds on $s$, $h$ and $c$ to depend on $m$ and on the family of systems as a whole is one of the key points in the representability notion. 
	The existence of $r$ in RP\ref{prop_rep2} and the definition of $r_0(q)$, imply that the labels of the edges of each copy of $H$ in $K$, ordered by colours, form a solution of the system $(A,G)$. 

For each solution $\mathbf{x}=(x_1,\ldots,x_m)\in S(A,G)$, the set $Q$ equipartitions the copies of $H$ in $K$ related to $\mathbf{x}$. The conditions RP\ref{prop_rep2} and RP\ref{prop_rep3} guarantee, for each $\mathbf{x}$, $q$ and $i\in[1,m]$, the existence of a set of $i$-colored edges with size
	 \begin{displaymath}
	 	|E_i(\mathbf{x},q)|=\lambda \gamma_i=c\frac{|K|^s}{|G|} \gamma_i,
	 \end{displaymath}
where $c$ is lower bounded by a function of $m$ such that the following holds. For each edge $e\in E_i(\mathbf{x},q)$, there are $p \frac{\prod_{j=1}^m \gamma_j}{\gamma_i}$ copies of $H$ in $K$ related to $(\mathbf{x},q)$ containing $e$. $p$ is independent on $\mathbf{x}$, $i$, $q$ or $e$. By the existence of $r$ in RP\ref{prop_rep2}, any copy of $H$ in $K$ related to $\mathbf{x}$ through $q$ intersects $E_i(\mathbf{x},q)$ for all $i\in [1,m]$.
	
If the system is strongly representable, then $E_{i}(\mathbf{x},q)$ is the set of edges labelled with $(\mathbf{x})_i$.

In Definition~\ref{d.rep_sys}, we could have made the constants $c$ to depend on the pair $(\mathbf{x},q)$ as long as $c_{\mathbf{x},q}\geq \chi_2$ for any $(\mathbf{x},q)\in S(A,G)\times Q$. The proof of Theorem~\ref{t.rep_sys_rem_lem} in Section~\ref{s.rep_sys_proof} can be adapted to this case by using the bound $\chi_2$ instead of $c$.

If the system is $\gamma$-strongly-representable, then the new set $Q$ can be considered to be $\{1\}$ at the expense of increasing the value of $p$ to $p|Q|$. Indeed, for any $q$, the set of hypergraphs $H$ in $K$ related to $(\mathbf{x},q)$ contains all the edges labelled $(\mathbf{x})_i$. Therefore any edge labelled $(\mathbf{x})_i$ contains $p\frac{\prod_{j=1}^m\gamma_j}{\gamma_i}|Q|$ copies of $H$ related to $\mathbf{x}$ in $\cup_{q\in Q} \;r^{-1}(\mathbf{x},q)$.

\subsection{Representable systems and the removal lemma} \label{s.rep_sys_proof}

The proof of the removal lemma for representable systems, Theorem~\ref{t.rep_sys_rem_lem}, uses the coloured version of the hypergraph removal lemma, Theorem~\ref{t.rem_lem_edge_color_hyper} in this work.
Theorem~\ref{t.rem_lem_edge_color_hyper} can be deduced from Austin and Tao's \cite[Theorem~1.5]{austao10}. Alternatively, the coloured version of the hypergraph removal lemma can be proved using the arguments that lead to the colourless version of the hypergraph removal lemma \cite{elesze12,gow07,rodletal,tao06}, or it can be found in Ishigami's \cite{ishi09}.

\begin{theorem}[Removal lemma for colored hypergraphs \cite{austao10}] \label{t.rem_lem_edge_color_hyper}
 For any positive integers $r$, $h$, $s$ with  $h\ge s\ge 2$ and every $\varepsilon >0$ there exists $\delta>0$
 depending on $r$, $h$, $s$ and $\varepsilon$  such that
the following holds.

Let $H$ and $K$ be $r$-colored $s$-uniform hypergraphs with $h=|V(H)|$ and   $M=|V(K)|$ vertices respectively.
If the number of copies of $H$ in $K$ (preserving the colors of the
edges) is at most $\delta M^h$, then there is a set $E'\subseteq
E(K)$ of size at most $\varepsilon M^{s}$ such that the hypergraph
$K'$ with edge set $E(K)\setminus E'$ is $H$--free.
\end{theorem}

\begin{proof}[Proof of Theorem~\ref{t.rep_sys_rem_lem}]
	Let $(K,H)$ be the hypergraph pair that  $\gamma$-represents the system $(A,G)$, with $\gamma=(\gamma_1,\ldots,\gamma_m)$. Let us denote the labelling by $l:E(K)\to G$ and the representation function by
	$r:C(H,K)\to S(A,G)\times Q$.
The components of $r$ are given by $r_0:C(H,K)\to S(A,G)$ and $r_q:C(H,K)\to Q$. Recall that, by Definition~\ref{d.rep_sys}, if $H_0=\{e_1,\ldots,e_m\}$ is a copy of $H$ in $K$, then $r_0(H)=(l(e_1),\dots,l(e_m))$. 
	Let $K_X$ be the  subhypergraph of $K$ with the same vertex set as $K$ and the edges belonging to  $r_0^{-1}(S(A,G,X))$. In other words, $K_X\subset K$ is the hypergraph containing only the edges whose labels belong to the restricted solution set.

By the property RP\ref{prop_rep2} of the $\gamma$-representability of the system, the total number of copies of $H$ in $K$ is, for the $c$ and $p$ provided by the representation, at most
	$$
	 c\frac{|K|^s}{ |G|} p |Q||S(A, G)| \prod_{i=1}^m \gamma_i.
	%=\lambda p |Q| |S(A, G)| \prod_{i=1}^m \gamma_i.
	$$
	Let $\lambda=c\frac{|K|^s}{ |G|}$. 
	Since $H$ has $h$ vertices, it follows that
	$$
	 \lambda p  |Q| |S(A, G)| \prod_{i=1}^m \gamma_i<|K|^h.
	$$
	On the other hand, the hypothesis $|S(A,G,X)|<\delta |S(A,G)|$, $\delta$ to be determined later, implies that the total number of copies of $H$ in $K_X$ is at most
	$$
	\lambda p |Q| |S(A, G, X)| \prod_{i=1}^m \gamma_i < \delta\lambda p|Q| |S(A, G)| \prod_{i=1}^m \gamma_i < \delta |K|^h.
	$$
	We apply the Removal Lemma for colored hypergraphs, Theorem~\ref{t.rem_lem_edge_color_hyper}, with $\varepsilon'=c\varepsilon/m$. By setting $\delta$ according to $\varepsilon'$ and $H$ in Theorem~\ref{t.rem_lem_edge_color_hyper}, we obtain a set of edges $E'\subset  E(K_X)$ with cardinality  at most $\varepsilon'|K|^s$   such that $K_X \setminus E'$  has no copy of $H$. We note that $\delta$ depends on $s,h,m$ and $\varepsilon'$, which in our context and by the representability, all depend on $m$ and $\varepsilon$.

	We next define the sets $X'_i\subset X_i$ as follows. The element $x$ is in $X_i'$ ($x$ is removed from $X_i$) if $E'$ contains at least  $\lambda \gamma_i /m $ edges labelled $x$ and colored $i$. We observe that 
\begin{displaymath}
	|X'_i|\le \frac{|E'|}{(\lambda\gamma_i/m)}  =\frac{m|G|}{c|K|^s}|E'| \frac{1}{\gamma_i}\le \varepsilon \frac{|G|}{\gamma_i}.
\end{displaymath}

	We claim that $S(A,G,X\setminus X')$, with $X\setminus X'=\prod_{i=1}^m X_i\setminus X_i'$, is empty. Indeed, pick one element  $\mathbf{x}=(x_1,\ldots ,x_m)\in S(A,G,X)$ and $q\in Q$. By RP\ref{prop_rep2} there are $p\lambda \prod_{i=1}^m \gamma_i$ copies of $H$ in $r^{-1}(\mathbf{x},q)$. Since  $\mathbf{x}\in S(A,G,X)$, all these copies belong to $K_X$. On the other hand, by RP\ref{prop_rep3}, every edge of $K$ coloured $i$ is contained in at most $p\prod_{j\in[1,m]\setminus \{i\}} \gamma_j$ copies of $H$ in $r^{-1}(\mathbf{x},q)$. Let $E'_{i,x_i}$ denote the set of edges in $E'$ labelled with $x_i$ and colored $i$. Then
	\begin{displaymath}
		\sum_{i=1}^m \left[|E'_{i,x_i}| p\prod_{j\in[1,m]\setminus \{i\}} \gamma_j \right]\geq  p \lambda\prod_{j\in[1,m]} \gamma_j
	\end{displaymath}
as there are no copies related to $(\mathbf{x},q)$ after $E'$ has been removed.
By the pigeonhole principle, at least one of the sets $E'_{i,x_i}$ is such that $|E'_{i,x_i}|>\lambda \gamma_i/m$. By the definition of $X'_i$, the element $x_i$ belongs to $X'_i$ and thus $\mathbf{x}\not\in  X\setminus X' \supset S(A,G,X\setminus X')$. This proves the claim and finishes the proof of the result.
\end{proof}

\subsection{Examples of representable systems and removal lemmas} \label{s.some_rep_syst}

\subsubsection{Subhypergraph copies}

As expected, the coloured hypergraph removal lemma can be retrieved from
Theorem~\ref{t.rem_lem_edge_color_hyper}.
The system of configurations induced by ``the copies of an $r$-coloured $k$-uniform hypergraph $H_0$ in an $r$-coloured $k$-uniform hypergraph $K_0$'' can be represented by Definition~\ref{d.rep_sys} as follows. Order the edges of $H$ arbitrarily. $H=H_0$ and $K=K_0$ as the pair of hypergraphs that represents the system. The property $A$ is the map from $E(K)^{|E(H)|}$ to $\{0,1\}$ such that $A(e_1,\ldots,e_{|E(H)|})=1$ if and only if the edges $(e_1,\ldots,e_{|E(H)|})$ conform a copy of $H$ in $K$
in which $e_i$ is the $i$-th edge of $H$ with the chosen order. The map $l$
is given by the identity map of the edge in $K$, $r_0$ is the identity map induced by the property $A$, $Q=\{1\}$, $\lambda=c=\gamma_i=1$.
The sets $X_i$ in the removal lemma Theorem~\ref{t.rep_sys_rem_lem} are the edges in $K_0$ coloured using the colour of the $i$-th edge in $H$.

\subsubsection{Permutations}\label{s.perm_rep_and_rem_lem}

The copies of $\tau\in \mathcal{S}(t)$ in $\sigma\in \mathcal{S}(n)$, as defined by the set $\Lambda^{\tau}(\sigma)$ in Section~\ref{s.intro_permutations}, can be represented using directed and coloured graphs $H$ and $K$ in Definition~\ref{d.rep_sys} as follows.
Given a finite set $V$, let ${V\choose i}$ denote the set of subsets of $i$ different elements of $V$. Let $A$ be the property
$A:{[0,n-1]\choose 2}^{\left|{[0,t-1] \choose 2}\right|} \to \{0,1\}$
such that $A(e_1,\ldots,e_{t(t-1)/2})=1$ if and only if 
the collection of endpoints of the edges $\{e_1,\ldots,e_{t(t-1)/2}\}$ configure an $m$-element set $\{x_0<\cdots<x_{t-1}\}$ in $[0,n-1]$ belonging to $\Lambda^{\tau}(\sigma)$.

Given a permutation $\sigma\in \mathcal{S}(n)$, let us define the loopless bicolored directed graph $G_{\sigma}$
as follows. The vertex set $V(G_{\sigma})$ is given by the $n$-element set $[0,n-1]$. The directed edge $e=(i,j)$ or $e=\{i\to j\}$, from $i$ to $j$, belongs to $E(G_{\sigma})$ if and only if $\sigma(i)<\sigma(j)$.
The edge $e=\{i\to j\}$ is painted blue if $i<j$ and painted red if $i>j$.
Observe that $|E(G_{\sigma})|={n \choose 2}$.

We claim that the system for the permutations involved in Proposition~\ref{p.rem_lem_permutations} is representable with $A$ as above, $H=G_{\tau}$, $K=G_{\sigma}$, $m={t \choose 2}$, $Q=\{1\}$, $\gamma_i=c=\lambda=1$ and $r_0$ given as follows.
If $\{x_0<\ldots<x_{t-1}\}\subset [0,n-1]=V(K)$ is a set of indices that generates a copy of $H$ in $K$, then $r_0(\{x_0<\ldots<x_{t-1}\})=\{x_0<\ldots<x_{t-1}\}$.

\begin{claim}\label{cl.0}
	If $H_0$, with $V(H_0)=\{x_0<\cdots<x_{t-1}\}$, is a copy of $G_{\tau}=([0,t-1],E(G_{\tau}))$
in $G_{\sigma}$, then the only map (homomorphism) from $f:V(G_{\tau}) \to V(H_0)$
with the property ``if $e=\{i\to j\}\in E(G_{\tau})$ and is coloured $c$, then
$\{f(i)\to f(j)\}\in E(G_{\sigma}{|V(H_0)})$ and is coloured $c$'' is the map
$f(i)=x_i$ for all $i\in [0,t-1]$.
\end{claim}

\begin{proof}
	The map $f$ must be bijective. Indeed, since $G_{\tau}$ is a complete graph
if $f$ were not bijective, then the graph induced by $V(H_0)$ would contain a loop as $f$ is a homomorphism, but $G_{\sigma}$ is loopless.

If $f$ is not the map $f(i)=x_i$, then there exist a pair $i,j\in[1,t-1]$ with $i<j$ but $f(i)>f(j)$. If the edge between $i$ and $j$ is $e=\{i\to j\}$, then
$f(e)=\{f(i)\to f(j)\}$. In such case, $e$ is painted blue as $i<j$ and $f(e)$ is painted red as $f(i)>f(j)$, hence $f$ is not an homomorphism. If the edge between $i$ and $j$ is $e'=\{j\to i\}$, then $e$ is coloured red but $f(e')$ is blue. Therefore, if $f$ is a homomorphism, it has to be the isomorphism with $f(i)=x_i$.
\end{proof}

\begin{claim}\label{cl.1}
	If $H_0$, with $V(H_0)=\{x_0<\cdots<x_{t-1}\}$ is a copy of
$G_{\tau}$ in $G_{\sigma}$ where $x_i\in V(G_{\sigma})$ corresponds to the $i$-th vertex of $G_{\tau}$, then $\{x_0<\cdots <x_{m-1}\}\in \Lambda^{\tau}(\sigma)$.
\end{claim}

\begin{proof}
	Let $e=\{x_i\to x_j\}$ be an edge in $G_{\sigma}$, then $\sigma(x_i)<\sigma(x_j)$. Since $H_0$ is a copy of $G_{\tau}$ where $x_i$ corresponds to the $i$-th vertex of $G_{\tau}$,  $e'=\{i\to j\}$ is an edge in $G_{\tau}$ meaning that $\tau(i)<\tau(j)$ as wanted. Since the reverse implication also holds, the result is shown.
\end{proof}

\begin{claim}\label{cl.2}
	If $\{x_0<\cdots <x_{t-1}\}\in \Lambda^{\tau}(\sigma)$ then the graph induced by $x_0,x_1,\ldots,x_{t-1}$ in $G_{\sigma}$ is a copy of $G_{\tau}$
with the map from $V(G_{\tau})=[0,t-1]$ to $\{x_0,x_1,\ldots,x_{t-1}\}\subset V(G_{\sigma})$ given by $i\mapsto x_i$ for $i\in[0,t-1]$.
\end{claim}

\begin{proof}
Assume the pair $\{i,j\}\in {[0,t-1]\choose 2}$, with $i<j$, is such that $\sigma(x_i)<\sigma(x_j)$. By the construction of $G_{\sigma}$
we have the edge $\{x_i\to x_j\}$ and is painted blue (as $x_i<x_j$).
Since $\{x_0<\cdots <x_{t-1}\}\in \Lambda^{\tau}(\sigma)$, then $\tau(i)<\tau(j)$. Hence $G_\tau$ has the edge $\{i\to j\}$ coloured blue (as $i<j$).

Assume now that the pair $\{i,j\}\in {[0,t-1]\choose 2}$, with $i<j$, is such that $\sigma(x_i)>\sigma(x_j)$. $G_{\sigma}$ contains the edge $\{x_j\to x_i\}$ painted red (as $x_j>x_i$). On the other side we have $\tau(i)>\tau(j)$ as $\{x_0<\cdots <x_{t-1}\}\in \Lambda^{\tau}(\sigma)$.  Hence $G_\tau$ has the edge $\{j\to i\}$ coloured red (as $j>i$).

Therefore, the map $i\mapsto x_i$, for $i\in[0,t-1]$, is a graph homomorphism preserving the colours and the directions of the edges as claimed.
\end{proof}

Combining claims \ref{cl.0}-\ref{cl.2}, we observe that $r_0$ is well defined and the representation of $\Lambda^{\tau}(\sigma)$ is given by the pair $(G_{\tau},G_{\sigma})$ with the parameters described above.
Proposition~\ref{p.rem_lem_permutations} is shown by using
Theorem~\ref{t.rep_sys_rem_lem} with $X_i={[0,n-1] \choose 2}$ for all $i\in [t(t-1)/2]$. In this case the proof of Theorem~\ref{t.rep_sys_rem_lem} should use, instead of Theorem~\ref{t.rem_lem_edge_color_hyper}, a removal lemma for directed and coloured graphs that can by obtained by combining the arguments from \cite[Lemma~4.1]{alosha04} with \cite[Theorem~1.18]{komsim96}.\footnote{For a detailed argument of how to obtain a removal lemma for directed and coloured graphs, the reader may refer to \cite{vena_master}.}

\section{Equivalent systems and representability} \label{s.equiv_and_rep}

In this section we assume that the systems are defined by a homomorphism. The definition for $\mu$-equivalent systems is introduced in Section~\ref{s.equiv_systems} and in Section~\ref{s.oper_between_representable} the relations between $\mu$-equivalent systems and their representations are explored.

\subsection{Equivalent systems} \label{s.equiv_systems}

Let $\mu$ be a positive integer. The homomorphism system $(A_2,G_2)$ with $A_2:G_2^{m_1}\to G_2^{k_2}$ is said to be \emph{$\mu$-equivalent} to the homomorphism system $(A_1,G_1)$, $A_1:G_1^{m_1}\to G_1^{k_2}$, with $m_2\geq m_1$, if 
\begin{itemize}
	\item $\mu |S(A_1,G_1)|=|S(A_2,G_2)|$.
	\item There exist an injective map $\sigma: [1,m_1] \to [1,m_2]$ and affine homomorphisms $\phi_1,\ldots,\phi_{m_1}$, $\phi_i: G_2\to G_1$ such that the  map
	\begin{displaymath}
		\phi(x_1,\ldots,x_{m_2})=\left(\phi_1(x_{\sigma(1)}),\ldots,\phi_{m_1}(x_{\sigma(m_1)})\right)
		\end{displaymath}
induces a $\mu$-to-$1$ surjective map $\phi: S(A_2,G_2) \to S(A_1,G_1)$. 
\end{itemize}

An affine homomorphism is a map $\phi_i:G_2\to G_1$ with $\phi_i(x)=b+\phi_i'(x)$, where $\phi_i'$ is a homomorphism and $b$ is a fixed element in $G_1$. 
Observe that, if necessary, we can restrict $\phi_i$ to map from the subgroup $S_i(A_2,G_2)$ (or a coset of the subgroup $S_i((A_2,\mathbf{0}),G_2)$) to $S_i(A_1,G_1)$.
If $G_1=G_2$ and the $\{\phi_i\}_{i\in [1,m_1]}$ are affine automorphisms, then $\phi_i\left( S_i(A_2,G_2)\right)=S_i(A_1,G_1)$ and their sizes are the same. In this case the systems are said to be \emph{auto-equivalent}.

\subsection{Operations on equivalent systems and representability} \label{s.oper_between_representable}

The propositions \ref{p.1-auto-equiv-rep} through \ref{p.mu-equivalent_2} proved in this section expose how the property of equivalence between systems, as defined in Section~\ref{s.equiv_systems}, is related with their representability properties, Definition~\ref{d.rep_sys}. For this section $G$, $G_1$ and $G_2$ are finite abelian groups and the systems are homomorphism systems.

%%%%%%%%%%%%%%%%%%%%%%%%%%%%%%%%%%%%%%%%%%%%%%%%%%%%%%%%%%%%%%%
%%%%%%%%%%%%%    $1$-auto-equivalent systems       %%%%%%%%%%%%
%%%%%%%%%%%%%%%%%%%%%%%%%%%%%%%%%%%%%%%%%%%%%%%%%%%%%%%%%%%%%%%

\subsection{$1$-auto-equivalent systems}

\begin{proposition}[$1$-auto-equivalent systems]\label{p.1-auto-equiv-rep}
	Let $((A_2,\mathbf{b}_2),G)$ be a $k_2\times m_2$ system $1$-auto-equivalent to $((A_1,\mathbf{b}_1),G)$, a $k_1\times m_1$ system. Assume $((A_2,b_2),G)$ is $\gamma'$-representable by $(K',H')$ with constants $\chi_1,\chi_2$. If the edges coloured by $\sigma(1),\ldots,\sigma(m_1)$ cover all the vertices of $H'$, then  $(A_1,G)$ is $\gamma$-representable %by $(K,H)$ and 
	with the same constants $\chi_1,\chi_2$ and $\gamma_i=\gamma'_{\sigma(i)}$. If $((A_2,\mathbf{b}_2),G)$ is strongly representable, then so is $((A_1,\mathbf{b}_1),G)$.
\end{proposition}

\begin{proof}[Proof of Proposition~\ref{p.1-auto-equiv-rep}]
Let $\phi$ be the map that defines the $1$-auto-equivalence $\phi: S(A_2,G) \to S(A_1,G)$ with $\phi(x_1,\ldots,x_{m_2})=\left(\phi_1(x_{\sigma(1)}),\ldots,\phi_{m_1}(x_{\sigma(m_1)})\right)$.
Let $(K',H', \gamma',\linebreak[1] l', r', Q', p', c')$ be the vector defining the $\gamma'$-representation for $(A_2,G)$. Let $s$ be the uniformity of the edges of $H'$. 

The vector $(K,H,\gamma,l,r,Q',p',c')$ defines the $\gamma$-representation of $((A_1,\mathbf{b}_1),G)$  
as follows.
$\gamma_i=\gamma'_{\sigma(i)}$ for $i\in[1,m_1]$. $H$ and $K$ are the hypergraph on the same vertex set of $H'$ and $K'$ respectively, and with the edges given by the colours $\sigma(1),\ldots,\sigma(m_1)$. Repaint the edge coloured $\sigma(i)$ with colour $i$. If $e=\{v_1,\ldots,v_s\}$ is an edge coloured $\sigma(i)$ in $K'$ and labelled $l'(e)$, then $e$ is an edge coloured $i$ in $K$ and labelled $l(e)=\phi_i(l'(e))$. $r_q(H_0)=r_q'(H_0')$ where $H_0'$ is the unique copy of $H'$ in $K'$ spanned by the vertices of $H_0$ seen as vertices of $K'$.

Each copy of $H$ in $K$ induces a unique copy of $H'$ in $K'$ and vice-versa. Moreover, $\phi$ is a bijection between the solution sets and $(K',H',\gamma',l',r',Q',p',c')$ is a $\gamma'$-representation for $((A_2,\mathbf{b}_2),G)$. Therefore, $(K,H,\gamma,l,r,Q,p,c)$ as defined above induces a $\gamma$-representation for $((A_1,\mathbf{b}_1,G)$ and have the same constants $\chi_1$ and $\chi_2$ as $(K',H',\gamma',l',r,Q',p',c')$. Since $\phi_i$ are affine automorphisms, if the representation for $((A_2,\mathbf{b}_2),G)$ is strong, the so is the representation for $((A_1,\mathbf{b}_1,G)$ here presented. 
\end{proof}

%%%%%%%%%%%%%%%%%%%%%%%%%%%%%%%%%%%%%%%%%%%%%%%%%%%%%%%%%%%%%%%%
%%%%%%%%%%%%   $\mu$-auto-equivalent systems       %%%%%%%%%%%%%
%%%%%%%%%%%%%%%%%%%%%%%%%%%%%%%%%%%%%%%%%%%%%%%%%%%%%%%%%%%%%%%%

\subsection{$\mu$-auto-equivalent systems}

\begin{proposition}[$\mu$-auto-equivalent systems] \label{p.mu-auto-equivalent}
Let $((A_2,\mathbf{b}_2),G)$ be a $k_2\times m_2$ system $\mu$-auto-equivalent to the $k_1\times m_1$ system $((A_1,\mathbf{b}_1),G)$, $m_2\geq m_1$. Let
\begin{align}
		\phi:S((A_2,\mathbf{b}_2),G)&\longrightarrow S((A_1,\mathbf{b}_1),G) \nonumber \\
	(x_1,\ldots,x_{m_2}) &\longmapsto  \left(x_{1},\ldots,x_{m_1}\right) \nonumber
\end{align}
be the map that defines the $\mu$-auto-equivalence.
Assume $((A_2,\mathbf{b}_2),G)$ is $(\gamma'_1,\ldots,\gamma'_{m_2})$-representable by $(K',H')$ 
with constants $\chi_1,\chi_2$. If 
	the edges coloured by $1,\ldots,m_1$ cover all the vertices of $H'$, 	 then  $((A_1,\mathbf{b}_1),G)$ is $(\gamma'_{1},\ldots,\gamma'_{m_1})$-representable 
	with $\chi_1,\chi_2$ as constants.
	If $((A_2,\mathbf{b}_2),G)$ is strongly representable, then so is $((A_1,\mathbf{b}_1),G)$.
\end{proposition}

\begin{proof}[Proof of Proposition~\ref{p.mu-auto-equivalent}]
	
			Let $\iota$ be a map from $S((A_2,\mathbf{b}_2),G)$ to $[1,\mu]$ where, given $\mathbf{x}_1, \mathbf{x}_2\in S((A_2,\mathbf{b}_2),G)$ such that  $\phi(\mathbf{x}_1)=\phi(\mathbf{x}_2)$ and $\mathbf{x}_1\neq \mathbf{x}_2$, then $\iota(\mathbf{x}_1)\neq \iota(\mathbf{x}_2)$. If $\phi$ is a $\mu$-to-$1$ map, such $\iota$ exist, is exhaustive and induces an equipartition in $S((A_2,\mathbf{b}_2),G)$.
	
	Let $(K',H',\gamma',l',r',Q',p',c')$ be the vector defining the $\gamma'$-representation for $((A_2,\mathbf{b}_2),G)$. Let $s$ be the uniformity of the edges of $H'$. The candidate vector $(K,H,\gamma,l,r,Q,p,c)$ 
	is defined as follows.
	\begin{itemize}
		\item $Q=Q'\times [1,\mu]$, $p=p'\prod_{i=m_1+1}^{m_2} \gamma'_i$, $c=c'$, $\gamma$ is such that $\gamma_i=\gamma'_i $ for $i\in[1,m_1]$.
		\item $H$ and $K$ are the hypergraphs on the vertex sets of $H'$ and $K'$ respectively. $e=\{v_1,\ldots,v_s\}$ is an edge in $K$ coloured $i\in[1,m_1]$ if and only $e$ is an edge coloured $i\in[1,m_1]$ in $K'$. 
		\item $l$ is defined by $l(e)=l'(e)$ for $e$ an edge coloured $i\in[1,m_1]$.
		\item 
If $H_0\in C(H,K)$, then
$r_q(H_0)=(r_q'(H_0'),\iota(r_0'(H_0')))$ where $H_0'$ is the unique copy of $H'$ in $K'$ spanned by the vertices of $H_0$, seen as vertices of $K'$.
	\end{itemize}

	Selecting $\mathbf{x}\in S((A_1,\mathbf{b}_1),G)$ and $q=(q',j)\in Q=Q'\times[1,\mu]$ is equivalent to select the $\mathbf{y}\in S((A_2,\mathbf{b}_2),G)$, with $\mathbf{y}\in \phi^{-1}(\mathbf{x})$ 
	 such that $\iota(\mathbf{y})=j$, 
and $q'\in Q'$, first coordinate of $q$. Moreover, each copy of $H$ in $K$ induces a unique copy of $H'$ in $K'$ and vice-versa. Therefore, the class of copies of $H$ related to $(\mathbf{x},q)$ is the same as the copies of $H'$ related to $(\mathbf{y},q')$.
	
	Since each edge $e_i\in E(K')$, $i\in [1,m_1]$, is contained in $p'\frac{\prod_{j=1}^{m_2}\gamma'_j}{\gamma'_i}$ copies of $H'$ related to $(\mathbf{y},q')$, then it also contains, seen as an edge in $K$, $p'\frac{\prod_{j=1}^{m_2}\gamma'_j}{\gamma'_i}=p\frac{\prod_{j=1}^{m_1}\gamma'_j}{\gamma_i}$ copies of $H$ related to $(\mathbf{x},q)$. Therefore, $(K,H,\gamma,l,r,Q,p,c)$ as defined above induces a $\gamma$-representation for $(A_1,G)$ and have the same constants $\chi_1$ and $\chi_2$ as $(K',H',\gamma',l',r',Q',p',c')$. Moreover, since $\phi_i$ is the identity map for each $i$, if the representation for $((A_2,\mathbf{b}_2),G)$ is strong, then so is the presented representation for $((A_1,\mathbf{b}_1),G)$. 
\end{proof}

\subsection{$\mu$-equivalent systems}

\begin{proposition}[$\mu$-equivalent systems 1] \label{p.mu-equivalent_1}
Let $((A_2,\mathbf{b}_2),G_2)$ be a $k_2\times m_1$ system $\mu$-equivalent to the $k_1\times m_1$ system $((A_1,\mathbf{b}_1),G_1)$ with
\begin{align}
	\phi:S((A_2,\mathbf{b}_2),G_2)&\longrightarrow S((A_1,\mathbf{b}_1),G_1) \nonumber \\
	(x_1,\ldots,x_{m_1}) &\longmapsto  \left(\phi_1(x_{1}),\ldots,\phi_1(x_{m_1})\right) \nonumber
\end{align}
be the map that defines the $\mu$-equivalence.
Assume $(A_2,G_2)$ is $(\gamma_1,\ldots,\gamma_{m_1})$-representable by $(K',H')$ with constants $\chi_1,\chi_2$. If 
$\phi_1:G_2\to G_1$ is surjective and $\phi^{-1}(\mathbf{x})=\prod_{i=1}^{m_1} \phi_1^{-1}((\mathbf{x})_i)$ then  $((A_1,\mathbf{b}_1),G_1)$ is $(\gamma_{1},\ldots,\gamma_{m_1})$-representable %by $(K,H)$ and 
	with the same constants $\chi_1,\chi_2$. If $((A_2,\mathbf{b}_2),G)$ is strongly representable, then so is $((A_1,\mathbf{b}_1),G)$.
\end{proposition}

\begin{proof}[Proof of Proposition~\ref{p.mu-equivalent_1}]

 Observe that, for $i\in[1,m_1]$, $\phi_1(S_i((A_2,\mathbf{b}_2),G_2))=S_i((A_1,\mathbf{b}_1),G_1)$ as $\phi$ is surjective. Since $\phi_1$ is affine, $|\{y_i\in S_i((A_2,\mathbf{b}_2),G_2) : \phi_1(y_i)=x_i\}|$ is the same for each $x_i\in S_i((A_1,\mathbf{b}_1),G_1)$.
Since $\phi^{-1}(\mathbf{x})=\prod_{i=1}^{m_1} \phi_1^{-1}((\mathbf{x})_i)$ and $\phi_1$ is affine, then we can let
 $\beta=|S_i((A_2,\mathbf{b}_2),G_2)|/|S_i((A_1,\mathbf{b}_1),G_1)|$, as its value is independent of $i\in[1,m_1]$. Therefore, $\mu=\beta^{m_1}$. Additionally, since $\phi_1$ is surjective, $\beta=|G_2|/|G_1|$.

	Let $\iota$ be a map from $G_2$ to $\Z_{\beta}$ such that, if $y_1,y_2\in G_2$ with $\phi_1(y_1)=\phi_1(y_2)$ and $y_1\neq y_2$, then $\iota(y_1)\neq \iota(y_2)$. Since $\phi_1$ is a $\beta$-to-$1$ map between $G_2$ and $\phi_1(G_2)=G_1$, then such $\iota$ exist, is exhaustive and induces an equipartition of $G_2$ in $\beta$ classes. Moreover, $\iota$ induces the bijections
	\begin{displaymath}
	\begin{array}{cl}
		G_2 &\longrightarrow \phi_1(G_2)\times \Z_{\beta} \\ % \nonumber \\
		y &\longmapsto  (\phi_1(y),\iota(y)) %\nonumber
	\end{array}
	\text{ and }
	\begin{array}{cl}
		S((A_2,\mathbf{b}_2),G_2) &\longrightarrow S((A_1,\mathbf{b}_1),G_1)\times \Z_{\beta}^{m_1} \\ % \nonumber \\
		\mathbf{y} &\longmapsto  (\phi(\mathbf{y}),\iota(\mathbf{y})) %\nonumber
	\end{array}
\end{displaymath}
where $\iota((\mathbf{y}_1,\ldots,\mathbf{y}_{m_1}))=(\iota(\mathbf{y}_1),\ldots,\iota(\mathbf{y}_{m_1}))$.
	Let $\pi:\Z_{\beta}^{m_1}\to \Z_{\beta}^{m_1}/\langle 1,\ldots,1\rangle$ be the quotient map.

		Let $(K',H',\gamma',l',r',Q',p',c')$ be the vector defining the $\gamma'$-representation for $((A_2,\mathbf{b}_2),G_2)$. Let $s$ be the uniformity of the edges of $H'$. The candidate vector $(K,H,\gamma,l,r,Q,p,c)$ 
		is defined as follows.
		\begin{itemize}
			\item $Q=Q'\times \left[\Z_{\beta}^{m_1}/\langle 1,\ldots,1\rangle\right]$, $p=p'$, $c=c'$, $\gamma_i=\gamma'_i $ for $i\in[1,m_1]$.
			\item $H$ and $K$ are the hypergraphs on the same vertex sets and edge sets as $H'$ and $K'$ respectively. $e=\{v_1,\ldots,v_s\}$ is an edge in $K$ coloured $i\in[1,m_1]$ if and only $e$ is an edge coloured $i\in[1,m_1]$ in $K'$.

			\item $l(e)=\phi_1(l'(e))$ if $e$ is an edge coloured $i\in[1,m_1]$ as an edge in $K'$.
			\item
Given $H_0\in C(H,K)$, let $H_0'$ be the unique copy of $H'$ in $K'$ spanned by the vertices of $H_0$ and let $\mathbf{y}=(\mathbf{y}_1,\ldots,\mathbf{y}_{m_1})=r_0'(H_0')\in S((A_2,\mathbf{b}_2),G_2)$ the solution spanned by $H_0'$. Then
		 $r_q(H_0)=(r_q'(H_0'),\pi(\iota(\mathbf{y})))$. 
		\end{itemize}

	Property RP\ref{prop_rep1} is fulfilled with the same parameters and each edge bears a label given by $l$. The function $r=(r_0,r_q)$ goes from $C(H,K)$ to $S((A_1,\mathbf{b}_1),G_1)\times Q$ by the definition of $r'=(r_0',r_q')$, $\iota$, $\phi_1$ and $\pi$. $r$ is surjective because $r'$ is surjective and $S((A_2,\mathbf{b}_2),G_2)$ is in bijection with $S((A_1,\mathbf{b}_1),G_1)\times \Z_{\beta}^{m_1}$. Observe that $r^{-1}(\mathbf{x},q)$ is the union of those $r'^{-1}(\mathbf{y},q')$, with $\mathbf{y}\in S((A_2,\mathbf{b}_2),G_2)$, 
	%and $q'$ being the projection of $q$ onto $Q'$, 
	such that $\phi(\mathbf{y})=\mathbf{x}$ and $q=(q',\pi(\iota(\mathbf{y})))$. This union has $\beta$ elements, as this is the size of each class in the quotient $\Z_\beta^{m_1}/\langle(1,\ldots,1)\rangle$. Therefore,
	% Since $\beta$ is the size of  then
	\begin{align}
		\left|r^{-1}(\mathbf{y},q)\right|=\beta \left|r'^{-1}(\mathbf{x},q')\right|=\beta p'c' \frac{|K'|^s}{|G_2|} \prod_{i=1}^{m_1}  \gamma_i' =p c \frac{|K|^s}{|G_1|} \prod_{i=1}^{m_1}  \gamma_i, \nonumber
	\end{align}
	which shows RP\ref{prop_rep2}.
	
All the solutions $(\mathbf{y}_1,\ldots,\mathbf{y}_m)=\mathbf{y}\in S((A_2,\mathbf{b}_2),G_2)$ that conform the union just mentioned have the property that any component $\mathbf{y}_i$ takes all the possible $\beta$ values of $\phi_1^{-1} ((\mathbf{x})_i)$. Indeed, from all the solutions $(\mathbf{y}_1,\ldots,\mathbf{y}_m)=\mathbf{y}$ that, along with the $q'$, conform the sets of copies of $H$ given by $r^{-1}(\mathbf{x},q)$, there is only one solution $\mathbf{y}$ with $\mathbf{y}_i$ having a particular value in $\phi_1^{-1}((\mathbf{x})_i)$. Therefore,
	if two copies of $H$ in $K$ share an edge $e_i\in H_0\in r^{-1}(\mathbf{x},q)$, then they belong to the same set $r'^{-1}(\mathbf{y},q')$ if seen as copies of $H'$ in $K'$. Thus, there are $p'\frac{\prod_{j=1}^m  \gamma'_j}{\gamma'_i}=p\frac{\prod_{j=1}^m  \gamma_j}{\gamma_i}$ copies of $H$ in $K$ sharing $e_i$. This shows RP\ref{prop_rep3}.

	Let $\mathbf{x}\in S((A_1,\mathbf{b}_1),G_1)$ and $q=(q',j)\in Q$. Pick $e_i$ with $l(e_i)=(\mathbf{x})_i$ for some $i\in[1,m_1]$. Let $y_i=l'(e_i)$ by seen $e_i$ as an edge in $K'$. Let $\mathbf{y}$ be the unique solution to  $S((A_2,\mathbf{b}_2),G_2)$ such that $\phi(\mathbf{y})=\mathbf{x}$, $(\mathbf{y})_i=y_i$, and $\pi(\iota(\mathbf{y}))=j$. If $((A_2,\mathbf{b}_2),G_2)$ is strongly representable, there exists a $H_0'$, $H_0'\in r'^{-1}(\mathbf{y},q')$, with $e_i\in H_0'$. If $H_0$ is the unique copy of $H$ in $K$ on the vertices of $H_0'$, then $e_i\in H_0$ and $H_0\in r^{-1}(\mathbf{x},q)$. This shows RP\ref{prop_rep4} for the system $((A_1,\mathbf{b}_1),G_1)$ when $((A_2,\mathbf{b}_2),G_2)$ is strongly representable and finishes the proof of the proposition.
\end{proof}

\begin{proposition}[$\mu$-equivalent systems 2] \label{p.mu-equivalent_2}
Let $((A_2,\mathbf{b}_2),G_2)$ be a $k_2\times m_2$ system $\mu$-equivalent to the $k_1\times m_1$ system $((A_1,\mathbf{b}_1),G_1)$ with $m_2\geq m_1$. Let
	\begin{align}
		\phi:S((A_2,\mathbf{b}_2),G_2)&\longrightarrow S((A_1,\mathbf{b}_1),G_1) \nonumber \\
		(x_1,\ldots,x_{m_2}) &\longmapsto  \left(\phi_1(x_{1}),\ldots,\phi_{m_1}(x_{m_1})\right) \nonumber
	\end{align}
	be the map that defines the $\mu$-equivalence.
	Assume $((A_2,\mathbf{b}_2),G_2)$ is 
$\gamma'$-strongly-representable
	by $(K',H')$ with constants $\chi_1,\chi_2$. Assume the following.
 \begin{enumerate}[(i)]
 \item \label{h.1} The edges coloured by $[1,m_1]$ cover all the vertices of $H'$.
\item \label{h.2} Given $\mathbf{x}\in S((A_1,\mathbf{b}_1),G_1)$ and $i\in[1,m_1]$,
then 
\begin{displaymath}
	\left\vert \{\mathbf{y}\in S((A_2,\mathbf{b}_2),G_2)\; : \;  \phi(\mathbf{y})=\mathbf{x} \text{ and }  (\mathbf{y})_i = y_i \}\right|
	\end{displaymath} 
is constant for any $y_i\in S_i((A_2,\mathbf{b}_2),G_2)$ with $\phi_i(y_i)=(\mathbf{x})_i$.
 \end{enumerate} 
Then  $((A_1,\mathbf{b}_1),G)$ is strongly $\gamma$-representable 
with
$\gamma_i=\gamma_i'\frac{|S_i((A_2,\mathbf{b}_2),G_2)|}{|S_i((A_1,\mathbf{b}_1),G_1)|}\frac{|G_1|}{|G_2|}$ for $i\in[1,m_1]$ and the
%by $(K,H)$ and the 
 same constants $\chi_1,\chi_2$.
\end{proposition}

The condition (\ref{h.2}) is not superfluous. If $A_2=(1,1)$, $A_1=(1,2)$, $b_1=b_2=0$, $G_1=G_2=\Z_2$ and $\phi(x_1,x_2)=(2x_1,x_2)$, then $\phi$ is one to one but $\phi_1$ does not satisfy (\ref{h.2}) for the solution $(0,0)\in S(A_1,G_1)$. However, the number of solutions $\mathbf{y}\in S(A_2,G_2)$ with $\phi(\mathbf{y})=\mathbf{x}$ and $(\mathbf{y})_i=y_i$ is either zero (if there is no such solution $\mathbf{y}$ with $(\mathbf{y})_i=y_i$), or it is a positive fixed value for any $i\in[1,m_1]$ if there exist some solution $\mathbf{y}\in \phi^{-1}(\mathbf{x})$ with $(\mathbf{y})_i=y_i$. The reason being that $\phi$ and $\phi_i$ are affine homomorphisms and the preimage by $\phi$ and $\phi_i$ has a coset/subgroup-like structure. Therefore, the condition (\ref{h.2}) can be rephrased as
\begin{enumerate}[(i')]
		  \setcounter{enumi}{1}
	\item \label{h.2'} Given $\mathbf{x}\in S((A_1,\mathbf{b}_1),G_1)$ and $i\in[1,m_1]$
then, for any  $y_i\in S_i((A_2,\mathbf{b}_2),G_2)$ with $\phi_i(y_i)=(\mathbf{x})_i$, there exists a $\mathbf{y}\in S((A_2,\mathbf{b}_2),G_2)$ with $\phi(\mathbf{y})=\mathbf{x}$ and $(\mathbf{y})_i=y_i$.
\end{enumerate}

\begin{proof}[Proof of Proposition~\ref{p.mu-equivalent_2}]
	Since $\phi$ is surjective, so is $\phi_i:S_i((A_2,\mathbf{b}_2),G_2)\to S_i((A_1,\mathbf{b}_1),G_1)$ for $i\in[1,m_1]$. 
	Since $\phi_i$ is affine, $|\{y_i\in S_i((A_2,\mathbf{b}_2),G_2) : \phi_i(y_i)=x_i\}|$ is the same for each $x_i\in S_i((A_1,\mathbf{b}_1),G_1)$. As $\phi$ and $\phi_i$ are affine homomorphisms, given $\mathbf{x}\in S((A_1,\mathbf{b}_1),G_1)$, the solutions $\mathbf{y}\in S((A_2,\mathbf{b}_2),G_2)$ such that $\phi(\mathbf{y})=\mathbf{x}$ can be partitioned into
	\begin{multline} \label{e.union}
		\left\{\mathbf{y}\in S((A_2,\mathbf{b}_2),G_2) : \phi(\mathbf{y})=\mathbf{x}\right\}= \\ \bigcup_{\substack{y_i \in S_i((A_2,\mathbf{b}_2),G_2)\\ \phi_i(y_i)=(\mathbf{x})_i}} \left\{\mathbf{y}\in S((A_2,\mathbf{b}_2),G_2)\; :\; \phi(\mathbf{y})=\mathbf{x} \text{ and } (\mathbf{y})_i=y_i \right\}.
\end{multline}
By the assumptions, the size of the sets $\left\{\mathbf{y}\in S((A_2,\mathbf{b}_2),G_2)\; :\; \phi(\mathbf{y})=\mathbf{x} \text{ and } (\mathbf{y})_i=y_i  \right\}$ is independent of each $y_i$ with $\phi_i(y_i)=(\mathbf{x})_i$ and we denote it by $\mu_i$. Therefore (\ref{e.union}) is an equipartition.
	Let $\beta_i=|S_i((A_2,\mathbf{b}_2),G_2)|/|S_i((A_1,\mathbf{b}_1),G_1)|$ be the number of preimages by $\phi_i$ of each $x_i\in S_i((A_1,\mathbf{b}_1),G_1)$ in $S_i((A_2,\mathbf{b}_2),G_2)$. Then $\mu_i$ is such that $\mu_i\beta_i=\mu$.

	Let $(K',H',\gamma',l',r',Q',p',c')$ be the vector defining the $\gamma'$-strong-representation for $((A_2,\mathbf{b}_2),G_2)$. Let $s$ be the uniformity of the edges of $H'$. The candidate vector $(K,H,\gamma,l,r,Q,p,c)$ 
	is defined as follows.
	\begin{itemize}
		\item $Q=Q'$, 
		$c=c'$, 
		\item 
		$\gamma_i=\gamma'_i \frac{|S_i((A_2,\mathbf{b}_2),G_2)|}{|S_i((A_1,\mathbf{b}_1),G_1)|}\frac{|G_1|}{|G_2|}$ for $i\in[1,m_1]$. $p=\mu p'\frac{|G_2|^{m_1-1}}{|G_1|^{m_1-1}}\left[\prod_{i=m_1+1}^{m_2}\gamma_i'\right] \left[\prod_{i=1}^{m_1} \frac{|S_i((A_1,\mathbf{b}_1),G_1)|}{|S_i((A_2,\mathbf{b}_2),G_2)|}\right]$.

\item $H$ and $K$ are hypergraphs on the same vertex sets as $H'$ and $K'$ respectively. $e=\{v_1,\ldots,v_s\}$ is an edge in $K$ (respectively $H$) coloured $i\in[1,m_1]$ if and only if $e$ is an edge coloured $i\in[1,m_1]$ in $K'$ (respectively $H'$.)
			\item $l(e)=\phi_i(l'(e))$ if $e$ is an edge coloured $i\in[1,m_1]$ as an edge in $K'$.
		\item Given $H_0\in C(H,K)$, let $H_0'$ be the unique copy of $H'$ in $K'$ spanned by the vertices of $H_0$. Then $r_q(H_0)=r_q'(H_0')$.	\end{itemize}
	
RP\ref{prop_rep1} is satisfied for $(K,H)$ with the same bounds and the labelling function $l'$.
By the hypothesis (\ref{h.1}), each copy of $H'$ in $K'$ spans a unique copy of $H$ in $K$ and vice-versa.
Since
\begin{equation} \label{e.345}
	r^{-1}(\mathbf{x},q)=\bigcup_{\mathbf{y}\in \phi^{-1}(S((A_1,\mathbf{b}_1),G_1))\cap S((A_2,\mathbf{b}_2),G_2)}  r'^{-1}(\mathbf{y},q)
\end{equation}
and there are $\mu$ different $\mathbf{y}\in S((A_2,\mathbf{b}_2),G_2)$ with $\phi(\mathbf{y})=\mathbf{x}$, then the union (\ref{e.345}) is disjoint and
\begin{displaymath}
		|r^{-1}(\mathbf{x},q)|=\mu|r'^{-1}(\mathbf{y},q)|=\mu p' c\frac{|K'|^s}{|G_2|}\prod_{i=1}^{m_2} \gamma_i'.
\end{displaymath}
as each set $r'^{-1}(\mathbf{y},q)$ contains  $p' c\frac{|K'|^s}{|G_2|}\prod_{i=1}^{m_2} \gamma_i'$ copies of $H'$ in it.

By the definition of $\gamma$ and $p$ we have
\begin{multline}
		|r^{-1}(\mathbf{x},q)|=\mu p' c\frac{|K'|^s}{|G_2|}\prod_{i=1}^{m_2} \gamma_i'= \\ \mu p' \frac{|K|^{s}}{|G_2|} \left[\prod_{i=m_1+1}^{m_2} \gamma_i' \right]  \left[\prod_{i=1}^{m_1}\gamma_i\frac{|S_i((A_1,\mathbf{b}_1),G_1)|}{|S_i((A_2,\mathbf{b}_2),G_2)|} \right] \frac{|G_2|^{m_1}}{|G_1|^{m_1}}  =p c\frac{|K|^s}{|G_1|}\prod_{i=1}^{m_1} \gamma_i.
\end{multline}
Since $r'$ is a $\gamma'$-representation function and $\phi=(\phi_1,\ldots,\phi_{m_1})$ defines the $\mu$-equivalence between systems (in particular, is surjective),
RP\ref{prop_rep2} is satisfied for $r$.

Given $\mathbf{x}\in S((A_1,\mathbf{b}_1),G_1)$ and $q\in Q$, let $e_i$ be an edge coloured $i$ and with  $l(e_i)=(\mathbf{x})_i$. $H_0$, a copy of $H$ in $K$, belongs to $r( \mathbf{x},q)$ and contains $e_i$ if and only if $H_0$, as a copy of $H'$ in $K'$, contains $e_i$ and belongs to one of the $r'^{-1}(\mathbf{y},q)$ with $\phi(\mathbf{y})=\mathbf{x}$.
Since $((A_2,\mathbf{b}_2),G_2)$ is $\gamma'$-strongly-represented, each set $r'^{-1}(\mathbf{y},q)$ with $(\mathbf{y})_i=l'(e_i)$ contains an $H_0'$, a copy of $H'$ in $K'$, with $e_i\in H_0'$. By RP\ref{prop_rep3}, there are $p' \frac{\prod_{j=1}^{m_2} \gamma_j'}{\gamma_i'}$ copies of $H'$ containing $e_i$ in any set $r'^{-1}(\mathbf{y},q)$ whenever $(\mathbf{y})_i=l'(e_i)$.

There are $\mu_i=\mu/\beta_i$ solutions $\mathbf{y}\in S((A_2,\mathbf{b}_2),G_2)$ such that $\phi(\mathbf{y})=\mathbf{x}$ with $(\mathbf{y})_i=l'(e_i)$. Therefore, there is a total of
\begin{align}
	\mu_i &p' \frac{\prod_{j=1}^{m_2} \gamma_j'}{\gamma_i'}
	=\frac{\mu}{\beta_i} p'
	\left[\prod_{j=m_1+1}^{m_2} \gamma_j' \right]\frac{\prod_{j=1}^{m_1} \gamma_j'}{\gamma_i'}\nonumber \\
	&= \mu {p'} \left[\prod_{j=m_1+1}^{m_2} \gamma_j'\right] 
	\frac{\left[\prod_{j=1}^{m_1} \gamma_j' \right] \left[\prod_{j=1}^{m_1}\frac{ |S_j((A_2,\mathbf{b}_2),G_2)|}{ |S_j((A_1,\mathbf{b}_1),G_1)|}\right] \frac{|G_1|^{m_1}}{|G_2|^{m_1}}}{\gamma_i' \beta_i \frac{|G_1|}{|G_2|} } \frac{|G_2|^{m_1-1}}{|G_1|^{m_1-1}} \prod_{j=1}^{m_1} \frac{ |S_j((A_1,\mathbf{b}_1),G_1)|}{ |S_j((A_2,\mathbf{b}_2),G_2)|} \nonumber \\
	&= p \frac{\prod_{j=1}^{m_1} \gamma_j}{\gamma_i} \nonumber
\end{align}
copies of $H'$ in $K'$ through $e_i$ that, seeing as copies of $H$ in $K$, belong to $r^{-1}(\mathbf{x},q)$. Hence, the vector $(K,H,\gamma,l,r,Q,p,c)$ fulfills RP\ref{prop_rep3}. 

To show RP\ref{prop_rep4}, choose $q$ and 
let $e_i$ be an edge in $K$ and 
let $\mathbf{x}$ be a solution to $((A_1,\mathbf{b}_1),G_1)$ such that $(\mathbf{x})_i=l(e_i)$. By the surjectivity of $\phi$ there exists a $\mathbf{y}\in S((A_2,\mathbf{b}_2),G_2)$ with $\phi(\mathbf{y})=\mathbf{x}$. By the assumption (\ref{h.2}), we can choose the solution $\mathbf{y}$ such that $(\mathbf{y})_i=l'(e_i)$. Since $r'^{-1}(\mathbf{y},q)$ contains a copy $H_0'$ of $H'$ with $e_i\in H_0'$, then 
$r^{-1}(\mathbf{x},q)$ contains $H_0$, the copy of $H$ over the vertices of $H_0'$, and satisfies $e_i\in H_0$. This shows RP\ref{prop_rep4} and finishes the proof of Proposition~\ref{p.mu-equivalent_2}.
\end{proof}

\paragraph{Comment.}
 In Propositions \ref{p.mu-auto-equivalent}, \ref{p.mu-equivalent_1}, and \ref{p.mu-equivalent_2}, the permutation $\sigma$ has been omitted as the variables are assumed to be properly ordered so that $\sigma(i)=i$ for $i\in[1,m_1]$.

%%%%%%%%%%%%%%%%%%%%%%%%%%%%%%%%%%%%%%%%%%%%%%%%%%%%%%%%%%%%%%%%%%%%%%%%%%%% 
%%%%%%%%                                                              %%%%%% 
%%%%%%%%   Proof of the Removal Lemma for linear systems in groups    %%%%%%
%%%%%%%%                                                              %%%%%% 
%%%%%%%%%%%%%%%%%%%%%%%%%%%%%%%%%%%%%%%%%%%%%%%%%%%%%%%%%%%%%%%%%%%%%%%%%%%% 

\section{Proof of Theorem~\ref{t.rem_lem_ab_gr}: from $G$ to $\Z_n^t$} \label{s.proof_rl-lsg-1}

\paragraph{Sketch of the proof of Theorem~\ref{t.rem_lem_ab_gr}.}

Let $((A,\mathbf{b}),G)$ be a homomorphism system with $A:G^m\to G^k$.
We will see that each element of the family of the homomorphism systems on $m$ variables and $k$ equations, with $m\geq k+2$, admits a representation where the constants $\chi_1$ and $\chi_2$ involved only depend on $m$.
For any given system $((A,\mathbf{b}),G)$,
 we find a sequence of $\mu$-equivalent systems $\{((A^{(i)},\mathbf{b}^{(i)},G^{(i)})\}$, $i\in[1,\kappa]$ for some $\kappa\in \mathbb{N}$, such that $((A^{(\kappa)},\mathbf{b}^{(\kappa)}),G^{(\kappa)})$ is strongly representable. Moreover, the sequence is equipped with affine morphisms $\phi:S((A^{(i+1)},\mathbf{b}^{(i+1)},G^{(i+1)})\to S((A^{(i)},\mathbf{b}^{(i)},G^{(i)})$ that fulfill the hypotheses of
an appropriate proposition from Section~\ref{s.oper_between_representable}. By concatenating these propositions, we obtain the final result Proposition~\ref{p.repr_hom}. The final argument of the construction is summarized in Section~\ref{s.unwrap_const}. For the cases regarding  $m< k+2$ and to show the second part of Theorem~\ref{t.rem_lem_ab_gr}, the additional argument from Section~\ref{s.finish_rem_lem_dkA1} is used.

The sequence of systems $\{((A^{(i)},\mathbf{b}^{(i)},G^{(i)})\}_{i\in[1,\kappa]}$ deals with different features of the solution set $S((A,\mathbf{b}),G)$ so that, for the last element of the sequence, a $1$-strong-representation can be found using the methods from \cite{ksv13}. In Section~\ref{s.rep_indep_vector} the case of non-homogeneous systems is reduced to the homogeneous case $(A,G)$. In Section~\ref{s.repr_for_Zt_implies_G}, we observe that the representation for any abelian group can be reduced to the homocyclic case $\Z_n^t$, for some appropriate $t$ and $n$.
Section~\ref{s.representability_product_cyclics} is
devoted to the $\gamma$-representation for any system with $G=\Z_n^t$. In Section~\ref{s.hom_mat_to_integer_mat} we describe the interpretation of $A$ as an integer matrix in the case of $G=\Z_n^t$. Once we have an integer matrix, we prepare the system for any determinantal in Section~\ref{s.union_of_systems} while Section~\ref{s.gamma-effective} prepares the systems to deal with the cases where $\gamma\neq 1$. Sections~\ref{s.construction_circular_matrix} and Section~\ref{s.final_composition} are devoted to the representation by hypergraphs using the tools detailed in Section~\ref{s.circ_mat_properties}.

%%%%%%%%%%%%%%%%%%%%%%%%%%%%%%%%%%%%%%%%%%%%%%%%%%%%%%%%%%%%%%%%%%%%%%%%%%%%
%%%%%%%%%%%%%      Independent vector representation     %%%%%%%%%%%%%%%%%%%
%%%%%%%%%%%%%%%%%%%%%%%%%%%%%%%%%%%%%%%%%%%%%%%%%%%%%%%%%%%%%%%%%%%%%%%%%%%% 

\subsection{Representation and the independent vector} \label{s.rep_indep_vector}

Proposition~\ref{p.hom_to_all} below shows that we can restrict ourselves to consider homogeneous systems $A\mathbf{x}=0$.

\begin{proposition}[Representation: any independent vector] \label{p.hom_to_all}
Either there is no solution to $A\mathbf{x}=\mathbf{b}$, $x\in G^m$ or
 $((A,\mathbf{0}),G)$ is $1$-auto-equivalent to $((A,\mathbf{b}),G)$. \end{proposition}

\begin{proof}[Proof of Proposition~\ref{p.hom_to_all}]
	Assume that $A\mathbf{x}=\mathbf{b}$ has a solution $\mathbf{y}=(\mathbf{y}_1,\ldots,\mathbf{y}_m)$. The map $\phi: S((A,\mathbf{0}),G) \to S((A,\mathbf{b}),G)$ with $\phi((\mathbf{x}_1,\ldots,\mathbf{x}_m))=(\mathbf{x}_1+\mathbf{y}_1,\ldots,\mathbf{x}_m+\mathbf{y}_m)$ 
defines a $1$-auto-equivalence.
\end{proof}

%%%%%%%%%%%%%%%%%%%%%%%%%%%%%%%%%%%%%%%%%%%%%%%%%%%%%%%%%%%%%%%%%%%%%%%%%%%%
%%%   Representability for $\Z_n^t$ implies representability any $G$   %%%%%
%%%%%%%%%%%%%%%%%%%%%%%%%%%%%%%%%%%%%%%%%%%%%%%%%%%%%%%%%%%%%%%%%%%%%%%%%%%%

\subsection{Representability for $\Z_n^t$ implies representability for $G$}
\label{s.repr_for_Zt_implies_G}

This section shows how to obtain, for some system $A'$ and integers $t$ and $n$, a system $(A',\Z_n^t)$ $\mu$-equivalent to the given homogeneous system $(A,G)$. Moreover, the map defining the $\mu$-equivalence fulfills the hypothesis of Proposition~\ref{p.mu-equivalent_1}. 
Thus, a representation result for any system $(A,\Z_n^t)$ is enough.

%%%%%%%%%%%%%%%%%%%%%%%%%%%%%%%%%%%%%%%%%%%%%%%%%%%%%%%%%%%%%%%%%%%%%%%%%%%%
%%%%%%%%     Matrix of homomorphisms: from $G$ to $\Z_n^t$    %%%%%%%%%%%%%%
%%%%%%%%%%%%%%%%%%%%%%%%%%%%%%%%%%%%%%%%%%%%%%%%%%%%%%%%%%%%%%%%%%%%%%%%%%%% 

By the Fundamental Theorem of Finite Abelian Groups, $G$ can be expressed, for some $n_1,\ldots,n_t> 1$ as the product of cyclic groups
$
G=\Z_{n_1}\times \cdots \times \Z_{n_t}, \text{ with } n_i|n_j \text{ for } i\geq j.
$
Let $G'=\Z_{n_1}^t$. The group $G$ can be seen as a quotient of $G'$. Let us denote by $\tau:G'\to G$ the quotient map
\begin{displaymath}
	\tau(a_1,\ldots,a_t)= \left(\frac{n_1}{n_1}a_1,\ldots,\frac{n_1}{n_t}a_t\right)
\end{displaymath}
and let $\beta=|G'|/|G|$. Let $\tau'$ denote the extension of $\tau$ from $G'$ to $G'^m$ using the diagonal action; if $(x_1,\ldots,x_m)=\mathbf{x}\in G'^m$ then $\tau'(\mathbf{x})=(\tau(x_1),\ldots,\tau(x_m))$.

Recall that the set of homomorphisms $A:G^m\to G^k$ are in bijection with $k\times m$ homomorphism matrices $(\vartheta_{i,j})$ for some homomorphisms $\vartheta_{i,j}:G\to G$ with
\begin{equation}
\left(
	\begin{array}{ccc}
		\vartheta_{1,1} & \cdots & \vartheta_{1,m}  \\
		\vdots  &\ddots & \vdots \\
		\vartheta_{k,1} &  \cdots & \vartheta_{k,m} \\
		\end{array}\right)
		\left(\begin{array}{c}
		x_1 \\
		\vdots \\
			x_m \\
			\end{array}\right)	
	=	\left(\begin{array}{c}
			b_1 \\
			\vdots \\
				b_k \\
				\end{array}\right)
				\iff
\sum_{i=1}^m \vartheta_{j,i}(x_i)=b_j, \; \forall j\in [1,k].\nonumber
\end{equation}
See, for instance \cite[Section 13.10, p. 66]{vandW91-2}.

By considering the matrix of homomorphisms $(\vartheta_{i,j}')=(\vartheta_{i,j}\circ \tau)$, any homomorphism $A:G^m\to G^k$ induces a homomorphism $A':G'^m\to G'^k$. If we see $\mathbf{b}\in G'^k\supset G^k$, then the system $((A,\mathbf{b}),G)$ induces a system $((A',\mathbf{b}),G')$. Indeed if $\mathbf{y}\in S((A',\mathbf{b}),G')$ then $\tau'(\mathbf{y})\in S((A,\mathbf{b}),G)$ and for any $\mathbf{x}\in S((A,\mathbf{b}),G)$, then $\tau'^{-1}(\mathbf{x})\subset S((A',\mathbf{b}),G')$ and $\tau'^{-1}(\mathbf{x})\neq \emptyset$.

\begin{observation} \label{o.same_gamma} $\tau':S((A',\mathbf{b}),G')\to S((A,\mathbf{b}),G)$ is surjective.
	If $S_i((A,\mathbf{b}),G)$ is the translated subgroup obtained by projecting the solution set to the $i$-th coordinate of $G^m$, then $\tau^{-1}(S_i((A,\mathbf{b}),G))=S_i((A',\mathbf{b}),G')$ and
\begin{displaymath}
	\frac{|G|}{|S_i((A,\mathbf{b}),G)|}=\frac{|G'|}{|S_i((A',\mathbf{b}),G')|}.
\end{displaymath}
Moreover, for any $\mathbf{x}\in S((A,\mathbf{b}),G)$, $\tau'^{-1}(\mathbf{x})=\prod_{i=1}^m \tau^{-1}((\mathbf{x})_i)$.
\end{observation}

\begin{remark} \label{r.1}
Observe that $((A',\mathbf{b}),\Z_n^t)$ is $\mu$-equivalent to $((A,\mathbf{b}),G)$ with the surjective map $\tau':S(A',G')\to S(A,G)$ and that the hypotheses of Proposition~\ref{p.mu-equivalent_1} regarding the map $\phi=\tau'$ hold by Observation~\ref{o.same_gamma}. 
\end{remark}

Therefore, using Proposition~\ref{p.hom_to_all}, it is enough to find a $\gamma$-representation for $((A',0),\Z_n^t)$, alternatively denoted by $(A',\Z_n^t)$, with $\gamma_i=|\Z_n^t|/|S_i(A',\Z_n^t)|$.

%%%%%%%%%%%%%%%%%%%%%%%%%%%%%%%%%%%%%%%%%%%%%%%%%%%%%%%%%%%%%%%%%%%%%%%%%%%%
%%%%%%%%%%     From the representation of $\Z_n^t$ to $G$     %%%%%%%%%%%%%%
%%%%%%%%%%%%%%%%%%%%%%%%%%%%%%%%%%%%%%%%%%%%%%%%%%%%%%%%%%%%%%%%%%%%%%%%%%%%

\section{Proof of Theorem~\ref{t.rem_lem_ab_gr}: $\gamma$-representability of $(A,\Z_n^t)$} \label{s.representability_product_cyclics}

In this section we prove the $\gamma$-representability of $(A,\Z_n^t)$ for homomorphism systems $A$ with $m\geq k+2$. The other cases with $m<k+2$ are treated in Section~\ref{s.finish_rem_lem_dkA1}. 
Following Section~\ref{s.repr_for_Zt_implies_G}, $A$ can be seen as a $k\times m$ matrix of homomorphisms. As previously mentioned, the construction involves creating a sequence of $\mu$-equivalent sequence, each element of the sequence being a modification of the pair matrix-group from the previous one.

\subsection{From a homomorphism to an integer matrix} \label{s.hom_mat_to_integer_mat}

Let $g_i=(0,\ldots,0,\stackrel{i}{1},0,\ldots,0)$, $i\in [1,t]$, 
be the canonical generators of $G=\Z_n^t$. 
Any variable $x_i$ in $\Z_n^t$ can be decomposed into $t$ variables $x_i=(x_{i,1},\ldots,x_{i,t})$, with $x_{i,j}\in \Z_n$. Therefore, any $k\times m$ homomorphism matrix in $\Z_n^t$ can be expressed as a $t k \times  t m$ integer matrix by replacing each homomorphism $\psi:\Z_n^t\to\Z_n^t$ by a $t\times t$ integer matrix $\Psi=\left(
			\psi_{i,j}\right)$; $\psi_{i,j}$ is the coefficient of $g_i$ in the image of $g_j$ by $\psi$ expressed as a linear combination of the generators $g_1,\ldots,g_t$. Indeed, the image of $g_j$ by $\psi$ is an element of $\Z_n^t$, hence it can be thought of as a tuple in $[0,n-1]^t$; $\psi_{i,j}$ is the $i$-th component of such tuple.

With these considerations, the system $A$ can be interpreted as an integer system of dimensions $t k\times  t m$ with the variables in $\Z_n$:	$A \left(x_{1,1},\cdots,x_{1,t},
	\cdots,\linebreak[1] x_{m,1}\linebreak[1],\linebreak[1]\cdots\linebreak[1],\linebreak[1]x_{m,t} \right)^{\top}=0$ 
with
\begin{displaymath}
A=\left(
		\begin{array}{ccc}
			\Psi_{1,1} & \cdots & \Psi_{1,m} \\
			\vdots & \ddots & \vdots \\
			\Psi_{k,1} & \cdots & \Psi_{k,m} \\
		\end{array}\right), \text{ where $\Psi_{i,j}$ is a $t\times t$ block of integers.} 
\end{displaymath}

If $A^i$ is a column of zeros in $A$, we can exchange it with any column vector whose components are multiples of $n$. 
If the determinant of the $kt\times kt$ submatrix of $A$ formed by the first $kt$ columns is zero, as a matrix with coefficients in $\Z$, then we add appropriate multiples of $n$ to the main diagonal so that the modified matrix has non-zero determinant in $\Z$.\footnote{This can be done as, for instance, $n^{i kt}$ grows faster than $(kt)! n^i$ when $i$ increases and $n\geq2$.} 
The modified matrix and the original 
are equivalent in $\Z_n$.

Even though $A$ is treated as an integer matrix for most of Section~\ref{s.representability_product_cyclics}, the arguments should take in consideration the origins of $A$ as a homomorphism matrix. In particular, the $t$ variables $x_{i,1},\ldots,x_{i,t}$ coming from $x_i$ are kept consecutive as they represent a unique variable $x_i$.

%%%%%%%% Union of systems: independent vectors simulation  %%%%%%%%%

\subsection{Union of systems: independent vectors simulation}
\label{s.union_of_systems}

Let $S(A)$ denote the Smith Normal Form of $A$.
 Recall that the $i$-th determinantal divisor of $A$, denoted by $D_i(A)$ and named $i$-th determinantal for short, is the greatest common divisor of the determinants of all the $i\times i$ submatrices of $A$ (choosing $i$ rows and $i$ columns). 
The product of the first $i$ elements in the diagonal of $S(A)$, $\prod_{j=1}^i d_j$, equals the $i$-th determinantal of $A$, so $D_i(A)=\prod_{j=1}^i d_j$. $A_i$ denotes the $i$-th row of $A$ while $A^i$ its $i$-th column.

\begin{proposition}[Row multiples] \label{p.union_systems_d_k}
	Let $A$ be a $k\times m$, $m\geq k$ integer matrix. Let $d_1,\ldots, d_k$ denote its diagonal elements of the Smith Normal Form of $A$. There is a matrix $A^{\text{\tiny{(1)}}}$, equivalent to $A$ (row reduced), such that the row $A_j^\text{\tiny{(1)}}$ satisfies \[\gcd\left(\left\{A_{j,i}^{\text{\tiny{(1)}}}\right\}_{i\in[1,m]}\right)=d_j.\]
	
	 Furthermore, assume that $d_i\neq 0$ for $i\in[1,k]$. The matrix $A^{\text{\tiny{(2)}}}$, obtained from $A^{\text{\tiny{(1)}}}$ by dividing the row $A_j^\text{\tiny{(1)}}$ by $d_j$, has $k$-th determinantal one.
\end{proposition}

\begin{proof}[Proof of Proposition~\ref{p.union_systems_d_k}]
	Let $S=U^{-1}AV^{-1}$ be the Smith Normal Form of $A$, where $U$
	and  $V$ are integer unimodular matrices that convey, respectively, the row and column operations
	 that transform $A$ into $S$. We have $S=\left(D|0\right)$,
	where $D$ is a $k\times k$ diagonal integer matrix with
	$\det(D)=D_k(A)$ and $0$ is an all--zero $k\times (m-k)$
	matrix. $d_i$ is the $i$-th element in the main diagonal of $D$. Let $A^{\text{\tiny{(1)}}}=U^{-1}A=SV$. Notice that the system $A^{\text{\tiny{(1)}}}\textbf{x}=\mathbf{0}$ is equivalent to $A\textbf{x}=\mathbf{0}$. 
	
	As $A^{\text{\tiny{(1)}}}$ has been obtained from $S$ by column operations using integer coefficients, the $j$-th row $A_j^\text{\tiny{(1)}}$ is formed by integer multiples of $d_j$. Since $V$ is unimodular, then \[\gcd\left(\left\{A_{j,i}^{\text{\tiny{(1)}}}\right\}_{i\in[1,m]}\right)=d_j,\] which proves the first part of the statement. Let $A^{\text{\tiny{(2)}}}$ be the matrix obtained by dividing each row $A_j^{\text{\tiny{(1)}}}$ by $d_j$. We have $A^{\text{\tiny{(2)}}}=S^{\text{\tiny{(2)}}}V$, where $S^{\text{\tiny{(2)}}}=\left(I_k|0\right)$ is the Smith Normal Form of $A^{\text{\tiny{(2)}}}$ and $I_k$ is the $k\times k$ identity matrix. 
	This completes the proof.
	\end{proof}

The integer $d_i$ induces a homomorphism $d_i:G\to G$ with $d_i(x)=d_i x=\sum_{j=1}^{d_i}x$.
Let $\mathcal{P}_{d_i}(G)$ denote the set $d_i^{-1}(0)\subset G$, this is, the subgroup of preimages of $0$ by the homomorphism induced by $d_i$ inside $G$.

\begin{observation}[Solution set] \label{o.sol_set}
	Using Proposition~\ref{p.union_systems_d_k}: 
	\begin{displaymath}
		S(A,G)=\bigcup_{\mathbf{b} 
		} S((A^{\text{\tiny{(2)}}},\mathbf{b}),G), \; \text{for } \mathbf{b}\in\prod_{i=1}^k \mathcal{P}_{d_i}(G),
	\end{displaymath}
	 where $d_i$ is the greatest common divisor of the $i$-th row of $A^{\text{\tiny{(1)}}}$. 
\end{observation}
\begin{proof}[Proof of Observation~\ref{o.sol_set}]
	Let $\textbf{x}\in G^m$ be a solution to $A\mathbf{x}=\mathbf{0}$, or, equivalently, $A^{\text{\tiny{(1)}}}\mathbf{x}=\mathbf{0}$.
	Observe the $j$-th equation for $A^{\text{\tiny{(1)}}}$:
	\begin{displaymath}
	A^{\text{\tiny{(1)}}}_{j,1}x_1 + \cdots + A^{\text{\tiny{(1)}}}_{j,m} x_m=0
	\iff
	d_j\left(A^{\text{\tiny{(2)}}}_{j,1}x_1+ \cdots + A^{\text{\tiny{(2)}}}_{j,m} x_m\right)=0.
	\end{displaymath}
Thus, $A^{\text{\tiny{(2)}}}_{j,1}x_1+ \cdots + A^{\text{\tiny{(2)}}}_{j,m} x_m$ is an element of $\mathcal{P}_{d_j}(G)$. Doing the same for all the rows (equations) of the system gives us that $A^{\text{\tiny{(2)}}}\mathbf{x}=\mathbf{b}$ for some independent vector $\mathbf{b}$ in $\prod_{i=1}^k \mathcal{P}_{d_i}(G) \subset G^k$. Also, any solution to $A^{\text{\tiny{(2)}}}\mathbf{x}=\mathbf{b}$ for some $\mathbf{b}\in\prod_{i=1}^k \mathcal{P}_{d_i}(G)$ is a solution to $A^{\text{\tiny{(1)}}}\mathbf{x}=\mathbf{0}$ by multiplying the $i$-th equation by $d_i$.
\end{proof}

We introduce dummy variables $y_j\in G$ to account for those independent vectors that occur by Observation~\ref{o.sol_set}. The variables $y_i\in G$ are called \emph{simulating} variables.

\begin{observation}[Simulating the independent vector for $\Z_n^s$]\label{o.simul_small_tor_p-groups}
	Assume $G=\Z_n^{s}$.
 For each row $A_j^{\text{\tiny{(1)}}}$, the equation
	\begin{displaymath}
		A_{j,1}^{\text{\tiny{(2)}}}x_1+\cdots+A_{j,m}^{\text{\tiny{(2)}}}x_m -\frac{n}{\gcd(n,d_j)}y_j=0,
	\end{displaymath}
where $y_j$ is a new variable with $y_j\in G$,
	is $|\mathcal{P}_{n/\gcd(n,d_j)}(G)|$-auto-equivalent to
\begin{displaymath} 
	A_{j,1}^{\text{\tiny{(1)}}}x_1+\cdots+A_{j,m}^{\text{\tiny{(1)}}}x_m=d_j\left(A_{j,1}^{\text{\tiny{(2)}}}x_1+\cdots+A_{j,m}^{\text{\tiny{(2)}}}x_m\right)=0.
	\end{displaymath}
 The application  $(x_1,\ldots,x_m,y_i) \to (x_1,\ldots,x_m)$ gives the $|\mathcal{P}_{n/\gcd(n,d_j)}(G)|$-auto-equivalence.
	
	Moreover, for each value of the $j$-th component of the independent vector	$g\in \mathcal{P}_{d_j}(G)$, 
	there are $|G|/|\mathcal{P}_{d_j}(G)|$ values for $y_j$ with 
	\begin{displaymath}
		\frac{n}{\gcd(n,d_j)}y_j=g.
	\end{displaymath}
\end{observation}

\begin{proof}[Proof of Observation~\ref{o.simul_small_tor_p-groups}]
	Since $G=\mathbb{Z}_{n}^{s}$ then $\mathcal{P}_{d_j}(G)\cong \mathbb{Z}_{\gcd(n,d_j)}^{s}$. 
Observe that the introduction of $\overline{y}_j$ in
	\begin{displaymath}
		A_{j,1}^{\text{\tiny{(2)}}}x_1+\cdots+A_{j,m}^{\text{\tiny{(2)}}}x_m -\overline{y}_j=0
	\end{displaymath}
	with $\overline{y}_j\in \mathbb{Z}_{\gcd(n,d_j)}^{s}$ simulates the independent vector.
	
	As $\frac{n}{\gcd(n,d_j)}: \mathbb{Z}_{n}^{s} \to \mathbb{Z}_{\gcd(n,d_j)}^{s}$ with $\frac{n}{\gcd(n,d_j)}(g)=\frac{n}{\gcd(n,d_j)}g$ is a $|\mathcal{P}_{n/\gcd(n,d_j)}(G)|$-to-$1$ surjective homomorphism, we can replace the variable $\overline{y}_j\in \mathbb{Z}_{\gcd(n,d_j)}^{s}$ by the variable $y_j\in \mathbb{Z}_{n}^{s}$ multiplied by $\frac{n}{\gcd(n,d_j)}$ and obtain the two parts of the observation.
\end{proof}

	Let $A^{\text{\tiny{(3)}}}$ denote the new matrix of the system with the simulating variables. This is, $A^{\text{\tiny{(3)}}}=(A^{\text{\tiny{(2)}}} \; Y)$
	where $Y$\label{page.Y} is a collection of columns of a $k\times k$ diagonal integer matrix.

\textbf{Remark.} If $A$ is a $tk\times tm$ integer matrix coming from a homomorphism matrix, then we use Observation~\ref{o.simul_small_tor_p-groups} on each row with $G=\Z_n$ (or $s=1$). Additionally, Observation~\ref{o.sol_set} should consider the matrices as $tk\times tm$ integer matrices and $\mathbf{b}\in \Z_n^{tk}$. Adding the simulating variables is only needed when $\gcd(d_{i},n)\neq 1$. To simplify the arguments, we may add some additional columns in the matrix $Y$, with its coefficients being multiples of $n$, so that the final matrix $A^{\text{\tiny{(3)}}}$ has dimensions $tk\times tm^{\text{\tiny{(3)}}}$, with $m^{\text{\tiny{(3)}}}=m+k$. 
Since $D_{tk}(A^{\text{\tiny{(2)}}})=1$, then $D_{tk}(A^{\text{\tiny{(3)}}})=1$.

\begin{remark}\label{r.2}
	The system $(A^{\text{\tiny{(3)}}},\Z_n^t)$ is $\mu$-auto-equivalent to $(A^{\text{\tiny{(1)}}},\Z_n^t)$ with 
	\begin{align}
		\phi: S(A^{\text{\tiny{(3)}}},\Z_n^t) &\longrightarrow S(A^{\text{\tiny{(1)}}},\Z_n^t) \nonumber \\
		(x_1,\ldots,x_m,x_{m+1},\ldots,x_{m^{\text{\tiny{(3)}}}}) &\longmapsto (x_1,\ldots,x_m), \nonumber
	\end{align}
	where $\mu$ is the number of preimages by $\phi$ of each $\mathbf{x}\in S(A^{\text{\tiny{(1)}}},\Z_n^t)$ in $S(A^{\text{\tiny{(3)}}},\Z_n^t)$. If $m^{\text{\tiny{(3)}}}=m+k$ then $\mu=\prod_{i=1}^{tk}\frac{n}{\gcd(d_i,n)}$. \end{remark}

%%%% Completion of the matrix: from the determinantal to the determinant  %%%%

\subsection{From the determinantal to the determinant} \label{s.determinantal_to_determinant}

\begin{lemma}[Matrix extension, Lemma~9 in \cite{ksv13}]
 \label{lem:ext-mat}
	Let $A$ be a $k\times m$ integer matrix, $m\geq k$. 
There is an $m\times m$ integer matrix $N$ that contains $A$ in its   first $k$ rows and is such that $\det(N)=D_k(A)$.
\end{lemma}

Let us include a proof for completeness.
\begin{proof}[Proof of Lemma~\ref{lem:ext-mat}]
	Let $S=UAV=(D|0)$ be the Smith Normal Form of $A$, where $U$ and $V$ are unimodular matrices and $D$ is a $k\times k$ diagonal matrix. Consider 
\begin{displaymath}
	S'=\begin{pmatrix}
	D & 0\\
	0 & I_{m-k}\\
\end{pmatrix}\; \text{ and }\;
U'=\begin{pmatrix}
U & 0 \\
0 & I_{m-k} \\
\end{pmatrix}.
\end{displaymath}
Then $N=U'^{-1} S' V^{-1}$ is an integer matrix as $U'$ is unimodular and satisfy the thesis of the lemma.
\end{proof}

As $D_{tk}(A^{\text{\tiny{(2)}}})=1$, we use Lemma~\ref{lem:ext-mat} to extend the $tk\times tm$ integer matrix $A^{\text{\tiny{(2)}}}$ to
 a $tm\times tm$ determinant $1$ integer matrix
\begin{displaymath}
	N=\begin{pmatrix}
	A^{\text{\tiny{(2)}}}\\
	M\\
\end{pmatrix},\;
		\det\begin{pmatrix}
		A^{\text{\tiny{(2)}}}\\
		M\\
	\end{pmatrix}=\det(N)=1, 
\end{displaymath}
which is a part of the matrix
\begin{displaymath}
	A^{\text{\tiny{(4)}}}=\begin{pmatrix}
	A^{\text{\tiny{(2)}}} & 0 & Y\\
	M & I_{tm-tk} & 0\\
	\end{pmatrix}.
\end{displaymath}
Therefore, the matrix $A^{\text{\tiny{(4)}}}$ can be row reduced into a new matrix $A^{\text{\tiny{(5)}}}$ in such a way that
\begin{displaymath}
	A^{\text{\tiny{(5)}}}=\begin{pmatrix}I_{tm} &B\end{pmatrix}\sim A^{\text{\tiny{(4)}}}
\end{displaymath}
for some $tm\times\left[(tm-tk)+(tm^{\text{\tiny{(3)}}}-tm) \right]= tm\times t\overline{m}^{\text{\tiny{(4)}}}$ integer matrix $B$. 
Moreover, we can assume that the columns of the matrix $I_{tm}$ from $A^{\text{\tiny{(5)}}}$ correspond to the ordered original variables $((x_{1,1},\ldots,x_{1,t}),\linebreak[1]\cdots,(x_{m,1},\ldots,x_{m,t}))$.
 Observe that $A^{\text{\tiny{(5)}}}$ has $tm$ rows and $tm^{\text{\tiny{(5)}}}$ columns, where  $m^{\text{\tiny{(5)}}}=m+\overline{m}^{\text{\tiny{(4)}}}$.

\begin{remark} \label{r.3}
The system $((A^{\text{\tiny{(4)}}},\mathbf{0}),G)$, hence $((A^{\text{\tiny{(5)}}},\mathbf{0}),G)$, is $1$-auto-equivalent to 
$((A^{\text{\tiny{(3)}}},\mathbf{0}),G)$.
Indeed, for any solution $\mathbf{y}\in S(A^{\text{\tiny{(4)}}},G)$ there exists one, and only one, solution $\mathbf{x}\in S(A^{ \text{\tiny{(3)}}},G)$ such that the projection 
\begin{displaymath}
	\mathbf{y}=(\mathbf{y}_1,\ldots,\mathbf{y}_{tm^{\text{\tiny{(5)}}}}) \longmapsto (\mathbf{y}_1,\ldots,\mathbf{y}_{tm},\mathbf{y}_{tm+(tm-tk)+1},\ldots,\mathbf{y}_{tm^{\text{\tiny{(5)}}}} )
\end{displaymath}
gives $\mathbf{x}$.\footnote{The coordinates to be omitted correspond to the columns of $\begin{pmatrix}0 & I_{tm-tk}\end{pmatrix}^{\top}$ for $A^{\text{\tiny{(4)}}}$. The value of these variables is determined by the values on the first $m$ coordinates.}
\end{remark}

Let us show an observation that is helpful in Section~\ref{s.final_composition}.

\begin{observation} \label{o.number_of_solutions}
	Let $A=\begin{pmatrix}A' &B\end{pmatrix}$, with $B$ being a $k\times m$, $m\geq k$ integer matrix and $A'$ denotes a square matrix of dimension $k$. Let $n$ be a positive integer and assume that $\gcd(D_{k}(B),n)=1$. Then, for any value of $x_1,\ldots,x_k$, $x_i\in \Z_n$, there are $n^{m-k}$ values for $(x_{k+1},\ldots,x_{k+m})\in \Z_n^{m}$ with $A\mathbf{x}=0$.
\end{observation}
\begin{proof}[Proof of Observation~\ref{o.number_of_solutions}]
	Extend the matrix $A$ with Lemma~\ref{lem:ext-mat} to a $1$-auto-equivalent system
\begin{displaymath}
	A'=
	\begin{pmatrix}
		A' & B & 0\\
		0 & M & I_{m-k}\\
\end{pmatrix} \text{ with } \gcd\left(\det\begin{pmatrix} B \\ M \end{pmatrix},n\right)=1.
\end{displaymath}
Select a value for $x_1,\ldots,x_k$ and any value for the last $m-k$ variables of $A'$. Then the value of the variables $x_{k+1},\ldots,x_{k+m}$ in $\Z_n$ is uniquely determined as the determinant is coprime with $n$.
\end{proof}

\subsection{Grouping the variables: on the matrix $B$} \label{s.group_on_B}

In Section~\ref{s.hom_mat_to_integer_mat} we have assigned an integer matrix in $\Z_n$ to a given homomorphism matrix. Let us partially reverse this transformation. Consider $A^{\text{\tiny{(5)}}}$ to be formed by $mm^{\text{\tiny{(5)}}}$ blocks of size $t\times t$, where $m^{\text{\tiny{(5)}}}=\overline{m}^{\text{\tiny{(4)}}}+m$. 
Let $\mathcal{A}^{\text{\tiny{(5)}}}_i$ be the matrix formed by the $i$-th row of blocks. Omitting the blocks of zeroes from the $I_{tm}$ part of $A^{\text{\tiny{(5)}}}$, $\mathcal{A}^{\text{\tiny{(5)}}}_i$ can be written as
\begin{displaymath}
	\mathcal{A}^{\text{\tiny{(5)}}}_i=\begin{pmatrix}I_t & \mathcal{B}_i\end{pmatrix},
\end{displaymath}
where $\mathcal{B}_i$ corresponds to the rows $B_{(i-1)t+1},\ldots,B_{(i-1)t+t}$ from $A^{\text{\tiny{(5)}}}$. We can assume that $D_t(\mathcal{B}_i)\neq 0$ in $\Z$. Otherwise, we can add an appropriate multiple of $n$ to each of the elements of $\mathcal{B}_i$; the new matrix is equivalent in $\Z_n$ and has non-zero determinantal in $\Z$.

By Proposition~\ref{p.union_systems_d_k}, $\mathcal{B}_i$ has an equivalent, row reduced, matrix $\mathcal{B}_i^{\text{\tiny{(1)}}}$ where the greatest common divisors of the rows are the elements in the diagonal of the Smith Normal Form of $\mathcal{B}_i$.
By performing such row reductions into $\mathcal{A}^{\text{\tiny{(5)}}}_i$, or in the whole $A^{\text{\tiny{(5)}}}$ using the corresponding rows, the matrix $I_t$ turns into a unimodular matrix $U_i$ related with the row operations conducted on $\mathcal{B}_i$ to obtain $\mathcal{B}_i^{\text{\tiny{(1)}}}$.  Since $D_t(\mathcal{B}_i^{\text{\tiny{(1)}}})=D_t(\mathcal{B}_i)\neq 0$, $\mathcal{B}_i^{\text{\tiny{(1)}}}$ has no zero row in $\Z$.

As $U_i$ is unimodular, it induces an automorphism in $G=\Z_n^t$ denoted by $\phi_i^{\text{\tiny{(1)}}}: G\to G$, with 
\begin{displaymath}
	\phi_i^{\text{\tiny{(1)}}}(x)=\phi_i^{\text{\tiny{(1)}}}((x_1,\ldots,x_t))=\left[U_i^{-1}\begin{pmatrix}
	x_1\\
	\vdots\\
	x_t\\
\end{pmatrix}\right]^{\top}.
\end{displaymath}

Consider the  matrix 
$A^{\text{\tiny{(6)}}}=\begin{pmatrix}I_{tm}& B^{\text{\tiny{(1)}}}\end{pmatrix}$ where $B^{\text{\tiny{(1)}}}$ is formed by collecting all the rows from $\mathcal{B}_i^{\text{\tiny{(1)}}}$, $i\in[1,m]$.

\begin{remark} \label{r.4}
 $(A^{\text{\tiny{(6)}}},\Z_n^t)$ is $1$-auto-equivalent to $(A^{\text{\tiny{(5)}}},\Z_n^t)$ with
\begin{align}
	\phi:S(A^{\text{\tiny{(6)}}},\Z_n^t) &\longrightarrow S(A^{\text{\tiny{(5)}}},\Z_n^t) \nonumber \\
	\mathbf{x}=(\mathbf{x}_1,\ldots,\mathbf{x}_{m^{\text{\tiny{(5)}}}}) &\longmapsto (\phi_1^{\text{\tiny{(1)}}}(\mathbf{x}_1),\ldots,\phi_{m^{\text{\tiny{(5)}}}}^{\text{\tiny{(1)}}}(\mathbf{x}_{m^{\text{\tiny{(5)}}}})) \nonumber
\end{align}
being the map between the solutions sets.
\end{remark}

\subsection{Towards $\gamma\neq 1$: constructing several systems} \label{s.gamma-effective}

We create several auxiliary systems to achieve an appropriate $\gamma\neq 1$ that are combined in Section~\ref{s.final_composition}. The purpose of its combination is to create a strongly $1$-representable system $(A^{\text{\tiny{(7)}}},G^{\text{\tiny{(7)}}})$ with $S_i(A^{\text{\tiny{(7)}}},G^{\text{\tiny{(7)}}})=G^{\text{\tiny{(7)}}}$ for any $i$. $(A^{\text{\tiny{(7)}}},G^{\text{\tiny{(7)}}})$ is $\mu$-equivalent to $(A^{\text{\tiny{(6)}}},\Z_n^t)$ and the map of the  $\mu$-equivalence fulfills the hypotheses of Proposition~\ref{p.mu-equivalent_2}. See Remark~\ref{r.last}.

 Let $\mathcal{B}_{i}^{\text{\tiny{(2)}}}$ be the matrix obtained from $\mathcal{B}_{i}^{\text{\tiny{(1)}}}$ by dividing each row of $\mathcal{B}_i^{\text{\tiny{(1)}}}$, denoted by $\mathcal{B}_{i,[j]}^{\text{\tiny{(1)}}}$ with $j\in[1,t]$, by $d_{i,j}=\gcd(\mathcal{B}_{i,[j]}^{\text{\tiny{(1)}}})$. Therefore, the greatest common divisor of each row in $\mathcal{B}_{i}^{\text{\tiny{(2)}}}$ is one.\footnote{We could have chosen to divide the coefficients of the row $\mathcal{B}_{i,j}^{\text{\tiny{(1)}}}$ by the the minimum  $\overline{d}_{i,j}$ such that $\gcd(d_{i,j}/\overline{d}_{i,j},n)=1$.}
Let $B^{\text{\tiny{(2)}}}$ be the matrix formed by collecting the rows in $\mathcal{B}_{i}^{\text{\tiny{(2)}}}$, $i\in[1,m]$. That is to say, for $i\in[1,m]$ and $j\in[1,t]$, the $(i-1)t+j$-th row of $B^{\text{\tiny{(2)}}}$ is the $j$-th row of 
$\mathcal{B}_{i}^{\text{\tiny{(2)}}}$.

Given $i\in[1,m]$ and $j\in[1,t]$, let $\mathcal{B}^{\text{\tiny{(2)}}}_{i(j)}$ denote the matrix
\begin{displaymath}
\mathcal{B}^{\text{\tiny{(2)}}}_{i(j)}=
\begin{pmatrix}
	B^{\text{\tiny{(2)}}}_{[t(i-1)+1,t(i-1)+j-1]} \\
\mathcal{B}_{i,[j]}^{\text{\tiny{(1)}}} \\
B^{\text{\tiny{(2)}}}_{[t(i-1)+j+1,ti]}\\
\end{pmatrix}
\end{displaymath}
where $B^{\text{\tiny{(2)}}}_{[i_1,i_2]}$ denotes the set of rows with indices in $[i_1,i_2]$ from $B^{\text{\tiny{(2)}}}$ and $\mathcal{B}_{i,[j]}^{\text{\tiny{(1)}}}$ denotes the $j$-th row of $\mathcal{B}_{i}^{\text{\tiny{(1)}}}$. This is, all the rows of $\mathcal{B}^{\text{\tiny{(2)}}}_{i(j)}$ are the same as the rows of
$\mathcal{B}^{\text{\tiny{(2)}}}_{i}$ except the $j$-th, which is the same as the $j$-th row in
$\mathcal{B}_{i}^{\text{\tiny{(1)}}}$.

For $i\in [1,m]$ and $j\in[1,t]$ let $J_{(i,j)}$ be the matrix formed by
\begin{align}
	%\scriptstyle
	J_{(i,j)}'&=
		\begin{pmatrix}
			I_{t(i-1)} & 0  &0&0&0& B^{\text{\tiny{(2)}}}_{[1,t(i-1)]} & 0&0\\
			0 &  I_t &  0 & e_j &0& \mathcal{B}_{i(j)}^{\text{\tiny{(2)}}} & 0&0 \\
			0 & 0     &  I_{t(m-i)} &0 &0& B^{\text{\tiny{(2)}}}_{[ti+1,tm]} & 0&0\\
			0 & 0    & 0 &      1& 0 &  0 &1&0 \\
			0 & 0    & 0 &      0& I_{t-1} &  0 &0&I_{t-1} \\
			\end{pmatrix}\nonumber \\
			&\sim
			\begin{pmatrix}
	I_{t(i-1)} & 0   &0&0&0& B^{\text{\tiny{(2)}}}_{[1,t(i-1)]} & 0&0\\
	0 &   I_t  & 0 & 0 &0& \mathcal{B}_{i(j)}^{\text{\tiny{(2)}}} & -e_j&0 \\
	0 & 0     &  I_{t(m-i)} &0 &0& B^{\text{\tiny{(2)}}}_{[ti+1,tm]} & 0&0\\
	0 & 0    & 0 &      1& 0 &  0 &1&0 \\
	0 & 0    & 0 &      0& I_{t-1} &  0 &0&I_{t-1} \\
\end{pmatrix} \nonumber\\ &=\begin{pmatrix} I_{t(m+1)} & B_{(i,j)}^{\text{\tiny{(3)}}}\end{pmatrix}=J_{(i,j)}\nonumber 
\end{align}
where $e_j=(0,\ldots,0,\stackrel{j}{1},0,\ldots,0)^{\top}\in\Z^t$.

The variables in the system associated to $J_{(i,j)}$ take values over $G_{(i,j)}=d_{i,j}^{-1}(0)\subset 
\Z_n$, the subgroup of $\Z_n$ formed by the preimage of zero by the 
homomorphism induced by $d_{i,j}$ in $\Z_n$.\footnote{Observe that, if $\gcd(d_{i,j},n)=1$, then $G_{(i,j)}=\{0\}$.} The matrix $J_{(i,j)}$ can be considered as a homomorphisms system over $G_{(i,j)}$ or over $G_{(i,j)}^t$, by considering $J_{(i,j)}$ to be formed by blocks of size $t\times t$. In the first case the system is denoted by $(J_{(i,j)},G_{(i,j)})$ and, in the second, by  $(J_{(i,j)},G_{(i,j)}^t)$.

With respect to $A^\text{\tiny{(6)}}$, $t$ equations and $2t$ variables in $G_{(i,j)}$ have been added to $J_{(i,j)}$. The added variables form the $(m+1)$-th block of $t$ variables in $J_{(i,j)}$ and the last block of $t$ variables over $G_{(i,j)}$. The added equation corresponds to the last one in $J_{(i,j)}$ and it involves the $2t$ variables added.

Let $J_{(1,0)}$ be the system induced by the matrix
\begin{displaymath}
	J_{(1,0)}=
	\begin{pmatrix}
		I_{tm} &0 & B^{\text{\tiny{(2)}}} & 0\\
		0 &    I_t &  0 &I_t \\
	\end{pmatrix} =
	\begin{pmatrix}
			I_{t(m+1)} & B_{(1,0)}^{\text{\tiny{(3)}}} \\
		\end{pmatrix}
\end{displaymath}
that configures a system $(J_{(1,0)},\Z_n)$ or $(J_{(1,0)},\Z_n^t)$ if $J_{(1,0)}$ is seen as a block matrix. Let us denote $G_{(1,0)}=\Z_n$.

Let $\Upsilon=(1,0)\cup \left\{[1,m]\times [1,t]\right\}$. The systems $J_{\kappa}$, $\kappa\in \Upsilon$, thought of as integer matrices, have some common properties.
\begin{enumerate}[(i)]
	\item \label{prop.G0} $J_{\kappa}$ have $m^{\text{\tiny{(J)}}}=m^{\text{\tiny{(6)}}}+2=m^{\text{\tiny{(5)}}}+2$ variables and $k^{\text{\tiny{(J)}}}=k^{\text{\tiny{(6)}}}+1=m+1$  equations over $G_{\kappa}^t$.
	\item\label{prop.G1} The groups $G_{\kappa}$ are cyclic: $G_{\kappa}=\frac{n}{\gcd(d_{\kappa},n)}\cdot \Z_n \subseteq \Z_n$. $J_{\kappa}$ can be seen as a homomorphism system $(J_{\kappa},G_{\kappa}^t)$.
	\item\label{prop.G2} $J_{\kappa}$ can be displayed as $\begin{pmatrix}I_{t(m+1)} & B^{\text{\tiny{(3)}}}\end{pmatrix}$ for certain $B^{\text{\tiny{(3)}}}$ depending on $\kappa$ and with dimensions $t(m+1)\times t((m^{\text{\tiny{(5)}}}+2)-(m+1))$. All the rows $B^{\text{\tiny{(3)}}}_{[i]}$ from $B^{\text{\tiny{(3)}}}$ have $\gcd(B^{\text{\tiny{(3)}}}_{[i]},|G_{\kappa}|)=1$. Moreover, the block of $t$ consecutive rows $\mathcal{B}_i^\text{\tiny{(3)}}=B^{\text{\tiny{(3)}}}_{[(i-1)t+1,\ldots,(i-1)t+t]}$, $i\in [1,m+1]$, is such that $\gcd(D_{t}(\mathcal{B}_i^\text{\tiny{(3)}}),|G_{\kappa}|)=1$. Even more, $\gcd(D_{t}(\mathcal{B}_i^\text{\tiny{(3)}}),n)=1$
\end{enumerate}
\begin{remark}\label{r.prop.G3}
	For any $\kappa\in\Upsilon$,
	the homomorphism
	\begin{align}
		f_{\kappa}:S(J_{\kappa},G_{\kappa}^t)&\longrightarrow S(A^{\text{\tiny{(6)}}},G_{\kappa}^t)\subset S(A^{\text{\tiny{(6)}}},\Z_n^t) \nonumber \\
		\begin{pmatrix} 
			\left(x_{(1,1)},\ldots,x_{(1,t)}\right) \\
		\vdots \\
		\left(x_{(m^{\text{\tiny{(J)}}},1)},\ldots,x_{(m^{\text{\tiny{(J)}}},t)}\right)
	\end{pmatrix}
		&\longmapsto
		\begin{pmatrix}
			\left(d_{1,1}\; x_{(1,1)},\ldots, d_{1,t}\; x_{(1,t)}\right) \\
			\vdots \\
			\left(d_{k^{\text{\tiny{(6)}}},1} \; x_{(k^{\text{\tiny{(6)}}},1)},\ldots, d_{k^{\text{\tiny{(6)}}},t} \; x_{(k^{\text{\tiny{(6)}}},t)}\right) \\
			\left(x_{(k^{\text{\tiny{(J)}}}+1,1)},\ldots,x_{(k^{\text{\tiny{(J)}}}+1,t)}\right) \\
			\vdots \\
			\left(x_{(m^{\text{\tiny{(J)}}}-1,1)},\ldots,x_{(m^{\text{\tiny{(J)}}}-1,t)}\right) \\
		\end{pmatrix}
		\nonumber
	\end{align}
is surjective and $|G_\kappa^t|$-to-$1$.
\end{remark}

\begin{proof}[Proof of Remark~\ref{r.prop.G3}]
	The variables with indices $[k^{\text{\tiny{(6)}}}+1,m^{\text{\tiny{(6)}}}]$ from $(A^{\text{\tiny{(6)}}},G_{\kappa}^t)$ parameterize the solution set $(A^{\text{\tiny{(6)}}},G_{\kappa}^t)$. This is, any choice of $x_{(k^{\text{\tiny{(6)}}}+1,\cdot)},\ldots,x_{(m^{\text{\tiny{(6)}}},\cdot)}\in G_{\kappa}^t$ provide a unique solution to $(A^{\text{\tiny{(6)}}},G_{\kappa}^t)$.
	The same holds true for the variables indexed by $[k^{\text{\tiny{(J)}}}+1,m^{\text{\tiny{(J)}}}]$ in the system $(J_{\kappa},G_\kappa^t)$.
	
	Assume $(i,j)\in \Upsilon \setminus \{(1,0)\}$. The $(i,j)$-th equation for $(A^{\text{\tiny{(6)}}},G_{\kappa}^t)$ can be written as
	\begin{displaymath}
	x_{(i,j)}+d_{i,j} B_{[(i-1)t+j]}^{\text{\tiny{(2)}}}\cdot \left(x_{(k^{\text{\tiny{(6)}}}+1,j)},\ldots,x_{(m^{\text{\tiny{(6)}}},j)}\right)^{\top}=0
	\end{displaymath}
	On the other hand, the $(i,j)$-th equation in $(J_{\kappa},G_{\kappa}^t)$ is
	\begin{displaymath}
		\left\{\begin{array}{cr}
		\left.\begin{array}{c}
		y_{(i,j)} + B_{\kappa,[(i-1)t+j]}^{\text{\tiny{(3)}}}\cdot \left(y_{(k^{\text{\tiny{(J)}}}+1,j)},\ldots,y_{(m^{\text{\tiny{(J)}}},j)}\right)^{\top}= \\ =y_{(i,j)} + B_{[(i-1)t+j]}^{\text{\tiny{(2)}}}\cdot \left(y_{(k^{\text{\tiny{(J)}}}+1,j)},\ldots,y_{(m^{\text{\tiny{(J)}}}-1,j)}\right)^{\top} =0  \end{array}\right\} & \text{ if } (i,j)\neq \kappa \\
	
		\left.\begin{array}{c}
		y_{(i,j)} + B_{\kappa,[(i-1)t+j]}^{\text{\tiny{(3)}}}\cdot \left(y_{(k^{\text{\tiny{(J)}}}+1,j)},\ldots,y_{(m^{\text{\tiny{(J)}}},j)}\right)^{\top}= \\
		y_{(i,j)} + d_{i,j} B_{[(i-1)t+j]}^{\text{\tiny{(2)}}}\cdot \left(y_{(k^{\text{\tiny{(J)}}}+1,j)},\ldots,y_{(m^{\text{\tiny{(J)}}}-1,j)}\right)^{\top}-y_{(m^{\text{\tiny{(J)}}},j)}=0 
	\end{array}\right\} &  \text{ if }  (i,j)= \kappa. \\
	\end{array}\right.
	\end{displaymath}
Therefore, if we let 
$\left(x_{(k^{\text{\tiny{(6)}}}+1,\cdot)},\ldots,x_{(m^{\text{\tiny{(6)}}},\cdot)}\right)=\left(y_{(k^{\text{\tiny{(J)}}}+1,\cdot)},\ldots,y_{(m^{\text{\tiny{(J)}}}-1,\cdot)}\right)$,
 the variables $y_{(i,j)}$ and $x_{(i,j)}$ are such that
$d_{i,j}y_{(i,j)}=x_{(i,j)}$ for $(i,j)\neq \kappa$. If $(i,j)=\kappa$, then 
 $d_{i,j}G_{\kappa}=0$ and $(\mathbf{x})_{\kappa}=0=d_{i,j} (\mathbf{y})_{\kappa}$ for any pair of solutions $\mathbf{x}\in S(A^{\text{\tiny{(6)}}},G_{\kappa}^t)$ and $\mathbf{y}\in S(J_{\kappa},G_{\kappa}^t)$. This shows that the map $f_{\kappa}$ exists.

Since $f_{\kappa}$ maps the subset of parameterizing variables $(y_{(k^{\text{\tiny{(J)}}}+1,\cdot)},\ldots,y_{(m^{\text{\tiny{(J)}}}-1,\cdot)})$ to the parameterizing variables $(x_{(k^{\text{\tiny{(6)}}}+1,\cdot)},\ldots,x_{(m^{\text{\tiny{(6)}}},\cdot)})$ using the identity map, $f_{\kappa}$ is surjective. Moreover, since the image by $f_{\kappa}$ is independent of the variable $y_{(m^{\text{\tiny{(J)}}},\cdot)}$, the map is $|G_{\kappa}^t|$-to-$1$.
\end{proof}

In the following part, Section~\ref{s.circ_mat_properties}, we adapt to the case of homomorphisms matrices the properties of the $n$-circular matrices used to show $1$-strong-representations in \cite{ksv13} and \cite{can_tesis_09,kraserven08}.\footnote{Although the final representation to prove \cite[Theorem~1]{ksv13} is not strong, the one established in \cite[Lemma~4]{ksv13} has the strong property.} In particular, Proposition~\ref{p.constr_c} in Section~\ref{s.circ_mat_properties} constructs, given an $n$-circular matrix, a matrix $C$ with good properties for the representation.

In Section~\ref{s.construction_circular_matrix} an $n$-circular matrix $\overline{J}_{\kappa}$ is constructed for each matrix $J_{\kappa}$, $\kappa\in \Upsilon$. The final construction of the $1$-strong-representation is conducted in Section~\ref{s.final_composition}; it involves combining the matrices $\overline{J}_{\kappa}$, $\kappa\in \Upsilon$, in a single matrix $A^{\text{\tiny{(7)}}}$, as well as combining all the matrices $C_{\kappa}$, provided by Section~\ref{s.circ_mat_properties}, in a single matrix $C$.

\subsection{$n$-circular matrices and properties} \label{s.circ_mat_properties}

An integer matrix $A$ formed by $k\times m$ square blocks, $m\geq k$, is said to be \emph{block $n$--circular} if all the matrices formed by $k$ consecutive columns of blocks of $A$, $(A^i, \ldots, A^{i+k-1})$ (considering the indices modulo $m$) have determinant coprime with $n$. A matrix is called standard $n$--circular if it is $n$-circular and with the shape $\begin{pmatrix} I_k &B\end{pmatrix}$. When the size of the blocks is one this definition coincides with the one provided in \cite[Definition~3]{ksv13}. The properties of the $n$-circular matrices described in Proposition~\ref{p.constr_c} are used in the construction of the representation described in Section~\ref{s.final_composition}.

\begin{proposition} \label{p.constr_c}
	Let $A$ be a $kt\times mt$ integer matrix, $m\geq k$, formed by $km$ blocks of size $t\times t$.
Assume that $A$ is block $n$--circular. Then there exists a $m\times m$ block integer matrix $C=(\mathcal{C}_{i,j})$,  each block of size $t\times t$ and $(i,j)\in[1,m]^2$, with the following properties.
\begin{enumerate}[i.]
\item \label{p.p0}
$AC=0$
\item \label{p.p1} The $i$-th row of $t\times t$ blocks is such that $\mathcal{C}_{i,j}=0$, for $j\in \{i+k+1,\ldots,i-1\}$ with indices modulo $m$. So, the matrix looks like
\begin{displaymath} 
	C=	\begin{pmatrix}
\ast & \ast & \ast & \ast & 0 & 0 \\
0& \ast & \ast & \ast & \ast & 0 \\
0&0& \ast & \ast & \ast & \ast  \\
\ast&0&0& \ast & \ast & \ast   \\
\ast&\ast&0&0& \ast & \ast    \\
\ast&\ast&\ast&0&0& \ast  \\
\end{pmatrix}
\end{displaymath}

\item \label{p.p2}
$\gcd(\det(\mathcal{C}_{i,i}),n)=1$ and $\gcd(\det(\mathcal{C}_{i,k+i}),n)=1$ for all $i\in[1,m]$ with indices taken modulo $m$.
\end{enumerate}
\end{proposition}

\begin{proof}[Proof of Proposition~\ref{p.constr_c}]
	Consider the square matrix formed by the column blocks $A^{[i,i+k-1]}=(A^{[i]},\linebreak[1]\ldots,\linebreak[1]A^{[i+k-1]})$ with $i\in[1,m]$. By assumption $A^{[i,i+k-1]}$ is a square non-singular matrix as it has non-zero determianant. For the $j$-th vector in the column block $A^{[i+k]}$, $A^{[i+k],j}$, with $j\in[1,t]$, we can find rational coefficients $b_{(i-1)t+w,(i+k-1)t+j}$, $w\in[1,kt]$, with
	\begin{displaymath}
		A^{[i+k],j}=\sum_{w\in[1,kt]} b_{(i-1)t+w,(i+k-1)t+j}\; A^{[i,i+k-1],w}
	\end{displaymath}
	where $A^{[i,i+k-1],w}$ stands for the $w$-th column in $A^{[i,i+k-1]}$ and corresponds to the $((i-1)t+w)$-th column in $A$. Moreover, since the determinant is coprime with $n$, there exists an integer $c_{(i+k-1)t+j,(i+k-1)t+j}$, coprime with $n$, such that
	\begin{equation}\label{e.coeff_C}
		-c_{(i+k-1)t+j,(i+k-1)t+j}\; A^{[i+k],j}=\sum_{w\in[1,kt]} c_{(i-1)t+w,(i+k-1)t+j}\; A^{[i,i+k-1],w}
	\end{equation}
	where, for $w\in[1,kt]$, $c_{(i-1)t+w,(i+k-1)t+j}=-c_{(i+k-1)t+j,(i+k-1)t+j} \; b_{(i-1)t+w,(i+k-1)t+j}$ are integers.

The coefficients of the matrix $C$ are 
\begin{enumerate}[(a)]
	\item $c_{w_1,w_2}$ whenever the subscripts $(w_1,w_2)$ coincide with one the $c$'s found in the relations given by (\ref{e.coeff_C}) for the $mk$ column vectors of $A$.
	\item \label{c.matrix_c.2} $0$ otherwise.
\end{enumerate} 
Consider the matrix $C$ as divided into $t\times t$ blocks $\mathcal{C}_{\cdot,\cdot}\;$. $C$ satisfy property \ref{p.p1}. Indeed, given a column of $C$ indexed by $j=j_1t+j_2$, with $j_1\in[0,m-1]$ and $j_2\in[1,t]$, 
the indices $i$ of the rows involved in the relations given by (\ref{e.coeff_C}) satisfy $i\in[(j_1-k)t,(j_1+1)t]$.

The relation (\ref{e.coeff_C}) can be rearranged as 
\begin{displaymath}
	0=c_{(i+k-1)t+j,(i+k-1)t+j} A^{[i+k],j}+\sum_{w\in[1,kt]} c_{(i-1)t+w,(i+k-1)t+j} A^{[i,i+k-1],w}
\end{displaymath}
and can be extended to $\sum_{w\in[1,tm]} c_{w,(i+k-1)t+j} A^l$ considering that all the other $c$'s that appear in the sum are zero by (\ref{c.matrix_c.2}).
Thus \ref{p.p0} is satisfied.

Observe that $\mathcal{C}_{i,i}$ is a diagonal matrix where all the elements in the diagonal are coprime with $n$. Hence the first part of property \ref{p.p2} is satisfied. To show the second part observe that, for each $i\in[1,m]$, indices modulo $m$,
\begin{displaymath}
\begin{pmatrix}
	A^{[i-k]} & \cdots & A^{[i-1]}
\end{pmatrix} 
	\begin{pmatrix}
		0 & 0 & \cdots & 0 & \mathcal{C}_{i-k,i} \\
	 &  & I_{t(k-1)} &  & \vdots \\
		 &  &  &  & \mathcal{C}_{i-1,i} \\
\end{pmatrix}
= 
\begin{pmatrix}
	A^{[i-k+1]} & \cdots & A^{[i-1]} & \overline{A}^{[i]} 
\end{pmatrix}
\end{displaymath}
where the columns of $\overline{A}^{[i]}$ are multiples of the columns of $A^{[i]}$ by (\ref{e.coeff_C}). Indeed, $\overline{A}^{[i],j}= -c_{(i-1)t+j,(i-1)t+j} A^{[i],j}$. Therefore
\begin{displaymath}
\det\begin{pmatrix}
		A^{[i-k+1]} & \cdots & A^{[i-1]} & \overline{A}^{[i]} 
	\end{pmatrix}
	=
	\det\begin{pmatrix}A^{[i-k+1]} & \cdots & A^{[i-1]} & A^{[i]}\end{pmatrix} \prod_{j\in[1,t]}c_{(i-1)t+j,(i-1)t+j}
\end{displaymath}
which is a product of integers coprime with $n$. 
Since
\begin{displaymath}
		\det 	\begin{pmatrix}
				0 & 0 & \cdots & 0 & \mathcal{C}_{i-k,i} \\
			 &  & I_{t(k-1)} &  & \vdots \\
				 &  &  &  & \mathcal{C}_{i-1,i} \\
		\end{pmatrix}= \pm\det(\mathcal{C}_{i-k,i})
\end{displaymath}
and $(\pm\det(\mathcal{C}_{i-k,i}))\cdot \det\begin{pmatrix}
	A^{[i-k]} & \cdots & A^{[i-1]}
\end{pmatrix} =\det\begin{pmatrix}
		A^{[i-k+1]} & \cdots & A^{[i-1]} & \overline{A}^{[i]} 
	\end{pmatrix}$, then $\det(\mathcal{C}_{i-k,i})$
is an integer coprime with $n$. This proves the second part of \ref{p.p2} and finalizes the proof of the proposition.
\end{proof}

%%%%%%%%   construction of the $n$--circular matrix  %%%%%%%%%

\subsection{Construction of the $n$-circular matrix} \label{s.construction_circular_matrix}

Let $n$ be a positive integer and let $G$ be an abelian group of order $n$. For our purposes, we can assume $G=\Z_n$. Let $A$ be a $kt\times mt$ matrix $A=\begin{pmatrix}I_{tk}& B \end{pmatrix}$ though of as built with $km$ blocks of size $t\times t$. Moreover, we shall assume that 
$\gcd\left(D_t(\mathcal{B}_i),n\right)=1$, where $\mathcal{B}_i$ is the $i$-th 
block of $t$ rows of the submatrix $B$, $i\in[1,k]$. In this section we build a 
$tk^{\text{\tiny{(9)}}}\times tm^{\text{\tiny{(9)}}}$ integer matrix 
$A^{\text{\tiny{(9)}}}=\begin{pmatrix}I_{tk^{\text{\tiny{(9)}}}}& B \end{pmatrix}$ such that:
\begin{itemize}
	\item $(A^{\text{\tiny{(9)}}},G^t)$ is $1$-auto-equivalent to $(A,G^t)$.
   \item $A^{\text{\tiny{(9)}}}$ is $n$-circular with blocks of length $1$, hence $n$-circular with blocks of size $t$.
\end{itemize}

We enlarge the $t\times (m-k)t$ matrix $\mathcal{B}_i$ using Lemma~\ref{lem:ext-mat} to the $(m-k)t\times (m-k)t$ matrix
\begin{displaymath}
\mathcal{B}_i^{\text{\tiny{(4)}}}=\begin{pmatrix}
	\mathcal{B}_i\\
	M_i\\
\end{pmatrix}
\text{ with } 
\det\begin{pmatrix}
	\mathcal{B}_i\\
	M_i\\
\end{pmatrix}=1.
\end{displaymath}
By adding some new variables taking values in $G$, 
$\mathcal{A}_i=\begin{pmatrix}I_t & \mathcal{B}_i\end{pmatrix}$
turns into the matrix denoted by $\mathcal{A}^{\text{\tiny{(8)}}}_i$ with
\begin{displaymath}
	\mathcal{A}^{\text{\tiny{(8)}}}_i=
	\begin{pmatrix}
I_t & 0 & \mathcal{B}_i \\
0 & I_{(m-k-1)t} & M_i \\
\end{pmatrix}
=\begin{pmatrix}I_{(m-k)t} & \mathcal{B}_i^{\text{\tiny{(4)}}}\end{pmatrix}
\end{displaymath}

Let us denote by $B^{\text{\tiny{(4)}}}$ the matrix formed by attaching together all the rows in $\{\mathcal{B}_i^{\text{\tiny{(4)}}}\}_{i\in[1,k]}$
\begin{displaymath}
B^{\text{\tiny{(4)}}}=
\begin{pmatrix}
	\mathcal{B}_1^{\text{\tiny{(4)}}}\\
	\vdots \\
	\mathcal{B}_{k}^{\text{\tiny{(4)}}}
\end{pmatrix}.
\end{displaymath}
Denote by $	A^{\text{\tiny{(8)}}}$ the matrix
	$A^{\text{\tiny{(8)}}}=\begin{pmatrix}I_{k(m-k)t} & B^{\text{\tiny{(4)}}}\end{pmatrix}$.
The variables added with respect to $A$ take values over the whole $G$.
	The system $(A^{\text{\tiny{(8)}}},G)$ and $(A^{\text{\tiny{(8)}}},G^t)$ are $1$-auto-equivalent to $(A,G)$ and $(A,G^t)$ respectively.

\paragraph{A Lemma for the building blocks.} \label{s.lemma_building_blocks}

Lemma~\ref{lem:ext3} improves \cite[Lemma~11]{ksv13} so that each block can be constructed by adding a linear number of rows with respect to the original number of columns.

\begin{lemma}\label{lem:ext3} Let $n$ and $r$ be positive integers and let $M$ be an $r\times r$ integer matrix with determinant coprime with $n$. There are $r\times r$ integer matrices $S$ and $T$ such that 
	\begin{displaymath}
		M'=
		\begin{pmatrix}
			I_r  \\
			S \\
			M \\
			T\\
			I_r\\
			\end{pmatrix}
				\end{displaymath}
is a $5r\times r$ integer matrix with the property that each $r\times r$ submatrix of $M'$, consisting of $r$ consecutive rows, has a determinant coprime with $n$.
\end{lemma}

\begin{proof}[Proof of Lemma~\ref{lem:ext3}]
	We detail the construction of $T$.
Let us define the matrices $r-i\times r$ matrices 
\begin{displaymath}
	M^i=\begin{pmatrix} M^i_{i+1} \\ \vdots \\ M^i_r \\
	\end{pmatrix},\; i\in [0,r-1],
	\end{displaymath}  together with the rows $T_{i+1}$ inductively.  Let $M^0=M$. Let $d_i=\gcd(M^{i-1}_{i,i},\ldots,M^{i-1}_{r,i})$, $i\in[1,r]$, be the greatest common divisor of the column $M_{\cdot,i}^{i-1}$. Let
\begin{displaymath}
	T_{i}=\lambda_{i}^{i} M^{i-1}_{i}+ \cdots + \lambda_r^{i} M^{i-1}_{r}
\end{displaymath}
where 
$\lambda_{i}^{i}, \ldots,\lambda_r^{i}$ are such that
\begin{equation} \label{eq.6}
	\lambda_{i}^{i} M^{i-1}_{i,i}+ \cdots + \lambda_r^{i} M^{i-1}_{r,i}=d_{i}
\end{equation}
and where $\lambda_{i}^{i}$ is some prime, $p_{i}$, larger than $n$.
 This $p_{i}$ exists, subjected to the constrain (\ref{eq.6}), by the Dirichlet theorem regarding the containment of infinitely many primes in the arithmetic progressions $a+b\Z$ with $\gcd(a,b)=1$. Observe that
\begin{displaymath}
		\det
		\begin{pmatrix}
			M^{i-1}_{i} \\
%			M_3\\
			\vdots \\
			M^{i-1}_r\\
		T_1\\
		\vdots\\
		T_{i-1}
		\end{pmatrix}
=p_{i}
				\det
				\begin{pmatrix}
					M^{i-1}_{i+1} \\
		%			M_3\\
					\vdots \\
					M^{i-1}_r\\
				T_1\\
				\vdots\\
				T_{i}
			\end{pmatrix} 
\end{displaymath}

The rows of the matrix $M^i$, denoted by $M^i_j$, are
$M^i_{j}=M^{i-1}_j - (M^{i-1}_{j,i}/d_i) T_i$, for $j\in [i+1,r]$ and with $T_0=\mathbf{0}$.
The first $i$ columns of $M^i$ are the zero columns. 

Observe that $\gcd(d_i,n)=1$ as
\begin{displaymath}
	\det\begin{pmatrix}
		M^{i-1}_{i+1}\\
		\vdots\\
		M^{i-1}_{r}\\
	T_1\\
	\vdots \\
%	T_{i-1}\\
	T_i\\
		\end{pmatrix}=\det\begin{pmatrix}
			M^i_{i+1} \\
			\vdots \\
			M^i_r\\
		T_1\\
		\vdots\\
%		T_{i-1}\\
		T_i\\
			\end{pmatrix}
\end{displaymath}
is coprime with $n$ and the original matrix $M$ has determinant coprime with $n$. Therefore the equivalent matrix
\begin{displaymath}
 \begin{pmatrix}
		M^i_{i+1} \\
		\vdots \\
		M^i_r\\
	T_1\\
	\vdots\\
%		T_{i-1}\\
	T_i\\
		\end{pmatrix}
		\sim
		\begin{pmatrix}
			M_{i+1} \\
%			M_{i+2}\\
			\vdots \\
			M_r\\
		T_1\\
		\vdots\\
%		T_{i-1}\\
		T_i\\
		\end{pmatrix}
\end{displaymath}
also has a determinant coprime with $n$. This shows the property regarding the coprimality of the determinant of consecutive rows for the first $r$ rows constructed in this way, $T_1,\ldots,T_r$. Observe that
\begin{displaymath}
T_i=(\overbrace{0 \; \cdots \; 0}^{i-1}\; d_i \; \ast \; \cdots \; \ast).
\end{displaymath}
Since each $d_i$ is coprime with $n$, we can add the identity matrix after the matrix $T$ and the claimed properties are satisfied.
The matrix $S$
is built similarly but we start from the last column and we construct a lower diagonal matrix $S$.
\end{proof}

\paragraph{Attaching building blocks.}\label{page.attaching_building}

We use Lemma~\ref{lem:ext3} on each matrix $\mathcal{B}_i^{\text{\tiny{(4)}}}$ to obtain matrices
\begin{displaymath}
\begin{pmatrix}
		I_{t(m-k)}\\
		S_i\\
		\mathcal{B}_i^{\text{\tiny{(4)}}}\\
		T_i\\
		I_{t(m-k)}\\
\end{pmatrix}=
\begin{pmatrix}
				I_{t(m-k)}\\
				\mathcal{B}_i^\text{\tiny{(5)}}\\
				I_{t(m-k)}\\
		\end{pmatrix}
\end{displaymath}
that are put together into a large matrix
\begin{displaymath}
	A^{\text{\tiny{(9)}}}=
	\begin{pmatrix}
	I_{(4k+1)t(m-k)}
\begin{array}{c}
		I_{t(m-k)}\\
%		S_1\\
		\mathcal{B}_1^\text{\tiny{(5)}}\\
%		T_1\\
		I_{t(m-k)}\\
%		\mathcal{B}_2''\\
		% \mathcal{B}_2'\\
		% T_2\\
		% I_{t(m-k)}\\
		 \vdots\\
%		T_{m-1}\\
		I_{t(m-k)}\\
%		S_m\\
		\mathcal{B}_k^\text{\tiny{(5)}}\\
%		T_m\\
		I_{t(m-k)}\\
\end{array}
\end{pmatrix}=\begin{pmatrix}I_{(4k+1)t(m-k)} & B^{\text{\tiny{(9)}}}\end{pmatrix}
\end{displaymath}
that is $n$-circular. This is, any $r=(4k+1)t(m-k)$ consecutive columns form a matrix with determinant coprime with $n$. Indeed, the matrix formed by the first $r$ columns is the identity matrix. On the other cases, some columns of the left most identity matrix $I_{r}$ are selected, along with some other columns from the $B^{\text{\tiny{(9)}}}$ part. Therefore the determinant is, up to a sign, the determinant of the submatrix formed by the columns selected in $B^{\text{\tiny{(9)}}}$ and the rows corresponding to the indices of the columns not picked from $I_{r}$. If the set of columns selected are consecutive and contains all the columns of $B^{\text{\tiny{(9)}}}$, the determinant is coprime with $n$ as so is the determinant formed with  $t(m-k)$ consecutive rows from $B^{\text{\tiny{(9)}}}$.
Since the first and the last square blocks of $B^{\text{\tiny{(9)}}}$ are identity matrices, the remaining cases are shown.

For $G^t=\Z_n^t$, the equations induced by the new rows in $B^{\text{\tiny{(9)}}}$ with respect to $B^{\text{\tiny{(4)}}}$ are 
\begin{displaymath}
	x_i+\sum_{j=1}^{t(m-k)} B^{\text{\tiny{(9)}}}_{i,j}\; x_{r+j}=0, \; x_i\in \Z_n.
	\end{displaymath} Therefore, $(A^{\text{\tiny{(9)}}},G^t)$ is a $k^{\text{\tiny{(9)}}}\times m^{\text{\tiny{(9)}}}$ homomorphism system $1$-auto-equivalent to $(A^{\text{\tiny{(8)}}},G^t)$.

\begin{remark}\label{r.ext_matrix_auto}
The system $(A^{\text{\tiny{(9)}}},\Z_n^t)$ is $1$-auto-equivalent to $(A,\Z_n^t)$ by projecting onto the original coordinates using maps $\phi_i$ equal to the identity map. Indeed, any solution to $(A,\Z_n^t)$ can be extended uniquely to a solution in $(A^{\text{\tiny{(9)}}},\Z_n^t)$ as
the last $m-k$ variables in both systems parameterize the solutions in both cases.
\end{remark}

\subsection{Final composition for $\gamma\neq 1$ and representation for  $(A^{\text{(7)}},G)$} \label{s.final_composition}

\paragraph{Joining the matrices and groups.}
Let $\left\{\overline{J}_{\kappa}\right\}_{\kappa\in \Upsilon}$, be the $n$--circular integer matrices obtained from $\left\{ J_{\kappa}=\begin{pmatrix}I_{t(m+1)} & B_{\kappa}^{\text{\tiny{(3)}}}\end{pmatrix} \right\} _{\kappa\in \Upsilon}$ using the procedure in Section~\ref{s.construction_circular_matrix}. This applies by (\ref{prop.G2}) in Section~\ref{s.gamma-effective} regarding $\gcd(D_{t}(\mathcal{B}_{\kappa,i}^\text{\tiny{(3)}}),n)=1$.
All the matrices $\overline{J}_{\kappa}$ have the same dimensions $tk^{\text{\tiny{(J')}}}=(4k^{\text{\tiny{(J)}}}+1)t(m^{\text{\tiny{(J)}}}-k^{\text{\tiny{(J)}}})$, $tm^{\text{\tiny{(J')}}}=(4k^{\text{\tiny{(J)}}}+2)t(m^{\text{\tiny{(J)}}}-k^{\text{\tiny{(J)}}})$ over $G_{\kappa}$.

Consider $\Upsilon$ to be ordered lexicographically; given $(\kappa_1,\kappa_2),(\kappa_3,\kappa_4)\in \Upsilon$, $(\kappa_1,\kappa_2)<(\kappa_3,\kappa_4)$ if and only if $\kappa_1 m+\kappa_2<\kappa_3 m + \kappa_4$. The columns of the matrix $A^{\text{\tiny{(7)}}}$ correspond to the columns $\left(\overline{J}_{\kappa}\right)^v$ using the lexicographical order for the ordered set $[1,t m^{\text{\tiny{(J')}}}] \times \Upsilon \owns (v,\kappa)$. The rows $A^{\text{\tiny{(7)}}}$ correspond to the rows $\left(\overline{J}_{\kappa}\right)_w$ using the lexicographically ordered set $[1,t k^{\text{\tiny{(J')}}}]\times \Upsilon\owns (w,\kappa)$. The coefficients of $A^{\text{\tiny{(7)}}}$ are zero wherever the intersection of a column and a row does not appear in any of the matrices $\overline{J}_{\kappa}$. This is, the $(i,j)$ element of $A^{\text{\tiny{(7)}}}$ is given by
\begin{displaymath}
	(A^{\text{\tiny{(7)}}})_{i,j}=
	\left\{\begin{array}{cl}
	\left(\overline{J}_{\kappa}\right)_{\lambda_1,\mu_1} &\begin{array}{ll}
	 \text{if} &\left\{
	\begin{array}{c}
		 i=(\lambda_1-1)(1+mt)+1+(\kappa_1-1)t+\kappa_2\\
		j=(\mu_1-1)(1+mt)+1+(\kappa_1-1)t+\kappa_2\\
	\end{array}\right.\\
	\phantom{.}&\text{for some } 
	\left\{\begin{array}{c}
		\lambda_1\in[1,tk^{\text{\tiny{(J')}}}],\mu_1\in[1,tm^\text{\tiny{(J')}}]\\
	\kappa=(\kappa_1,\kappa_2)\in \Upsilon\\
		\end{array}\right.\\
\end{array}\\
	0 & \text{otherwise.}
\end{array}\right. 
\end{displaymath}

Consider $\overline{A^{\text{\tiny{(7)}}}}$ to be the block-diagonal matrix containing the matrices $\{\overline{J}_{\kappa}\}_{\kappa\in\Upsilon}$ as the blocks in the diagonal.
The matrix $A^{\text{\tiny{(7)}}}$ can be seen as an appropriate permutation of  rows and columns of the block matrix $\overline{A^{\text{\tiny{(7)}}}}$. Let $P_1$ and $P_2$ denote, respectively, the row and column permutations so that $A^{\text{\tiny{(7)}}}=P_1 \overline{A^{\text{\tiny{(7)}}}} P_2$.

$A^{\text{\tiny{(7)}}}$ can be considered as formed by $t^2 k^{\text{\tiny{(J')}}} m^{\text{\tiny{(J')}}}$ blocks of size $(1+tm)\times (1+tm)$ over the groups $G=\prod_{\kappa \in \Upsilon} G_{\kappa}$. Furthermore, $t^2$ of the $(1+tm)\times (1+tm)$ blocks can be grouped in a single block of size $t(1+tm)\times t(1+tm)$. This allows us to interpret $A^{\text{\tiny{(7)}}}$ as formed by $k^{\text{\tiny{(9)}}} m^{\text{\tiny{(9)}}}$ blocks of size $t(1+tm)\times t(1+tm)$ over the groups $G^t=\left(\prod_{\kappa \in \Upsilon} G_{\kappa}\right)^t$. Therefore, if we denote 
$k^{\text{\tiny{(7)}}}=k^{\text{\tiny{(J')}}}$ and $m^{\text{\tiny{(7)}}}=m^{\text{\tiny{(J')}}}$, $A^{\text{\tiny{(7)}}}$ can be considered as a $k^{\text{\tiny{(7)}}}\times m^{\text{\tiny{(7)}}}$ homomorphism system over $G^t$ denoted by  $(A^{\text{\tiny{(7)}}},G^t)$. $(A^{\text{\tiny{(7)}}},G^t)$ has the particularity that the solution set of the system $A^{\text{\tiny{(7)}}}$, $S(A^{\text{\tiny{(7)}}},G^t)$, is the cartesian product of the solution sets $\left\{S\left(\overline{J}_{\kappa},G_{\kappa}^t\right)\right\}_{\kappa\in \Upsilon}$.

\paragraph{Matrix $C$ for $A^{\text{\tiny{(7)}}}$.}

Since each of the matrices $\{\overline{J}_{\kappa}\}_{\kappa\in \Upsilon}$ is $n$--circular with block size $1$, we use Proposition~\ref{p.constr_c} to find band-shape matrices  $C_{\kappa}$ related to $\overline{J}_{\kappa}$ for $\kappa\in\Upsilon$.

All the  $\{C_{\kappa}\}_{\kappa\in \Upsilon}$ are joined into a single $C$ fulfilling the properties stated in Proposition~\ref{p.constr_c} for $A=A^{\text{\tiny{(7)}}}$.  Indeed, let $\overline{C}$ be the block matrix with the matrices  $\{C_{\kappa}\}_{\kappa\in \Upsilon}$ in the diagonal and zeros everywhere else. Observe that  $\overline{A^{\text{\tiny{(7)}}}} \; \overline{C}=\mathbf{0}$. Let $C=P_2^{-1}\overline{C}P_2$, where $P_2$ is the column permutation from  $\overline{A^{\text{\tiny{(7)}}}}$ into  $A^{\text{\tiny{(7)}}}$. Then the equality $A^{\text{\tiny{(7)}}}C=\mathbf{0}$ follows.

If we group the consecutive rows and columns of $C$ by blocks of size $t(1+mt)\times t(1+mt)$, then $C$ can be considered as a $(t(1+mt))^2$-sized-block matrix $C=\left(\mathcal{C}_{i,j}\right)$ with $(i,j)\in [1,m^{\text{\tiny{(7)}}}]\times [1,m^{\text{\tiny{(7)}}}]$. Moreover, $C$ has the band-shape inherited from $\{C_{\kappa}\}_{\kappa\in \Upsilon}$. In particular. 
\begin{itemize}
	\item $\mathcal{C}_{i,i}$, $i\in[1,m^{\text{\tiny{(7)}}}]$, is a $t(1+mt)\times t(1+mt)$ upper triangular matrix where each coefficient in the diagonal is coprime with $n$.
	\item $\mathcal{C}_{i,i+k^{\text{\tiny{(7)}}}}$, $i\in[1,m^{\text{\tiny{(7)}}}]$, is a $t(1+mt)\times t(1+mt)$ lower triangular matrix where each coefficient in the diagonal is coprime with the order of the group on which it is acting.\footnote{The matrix is lower triangular and not only block lower triangular (with the blocks in the diagonal having determinant coprime with $n$.) Indeed, the matrices $\overline{J}_{\kappa}$ built using Lemma~\ref{lem:ext3} are $n$-circular for blocks of size $1$.}
	\item $\mathcal{C}_{i,j}=0$ for $i\in[1,m^{\text{\tiny{(7)}}}]$ and $j\in [i+k^{\text{\tiny{(7)}}}+1,\ldots,i-1]$, indices modulo $m^{\text{\tiny{(7)}}}$.
\end{itemize}

\paragraph{Strongly $1$-representation for $(A^{\text{\tiny{(7)}}},G^t)$.} \label{s.strongly_1_rep}

We use a similar machinery as in \cite[Lemma~4]{ksv13} to construct, assuming $m^{\text{\tiny{(7)}}}\geq k^{\text{\tiny{(7)}}}+2$, a
strongly $1$-representation for $(A^{\text{\tiny{(7)}}},G)$ denoted by $(K,H)$.  

Let the hypergraph $K$ have the vertex set $V(K)=\left(\prod_{\kappa\in \Upsilon} G_{\kappa}\right)^t\times[1,m^{\text{\tiny{(7)}}}]$. $H$ has $[1,m^\text{\tiny{(7)}}]$ as its vertex set and $m^\text{\tiny{(7)}}$  edges $e_i=\{i,\ldots,i+k^\text{\tiny{(7)}}\} \mod m^{\text{\tiny{(7)}}}$ for $i\in[1,m^\text{\tiny{(7)}}]$ with $e_i$ coloured $i$. Since $m^{\text{\tiny{(7)}}}\geq k^{\text{\tiny{(7)}}}+2$, $H$ has $m^{\text{\tiny{(7)}}}$ pair-wise distinct edges.

The edges in $K^{\text{\tiny{(7)}}}$ form a $m^{\text{\tiny{(7)}}}$-partite $(k^{\text{\tiny{(7)}}}+1)$-uniform hypergraph. The edge $e_i=\{g_i,\ldots,g_{i+k^{\text{\tiny{(7)}}}}\}$, with $g_j\in G^t\times \{j\}\subset V(K)$, is coloured $i$ and is labelled $g$ if and only 
\begin{displaymath}
	\sum_{j=i}^{i+k^{\text{\tiny{(7)}}} } \mathcal{C}_{i,j}(g_j)=g.
\end{displaymath}
These labels on the edges define a labelling function $l:E(K)\to G^t$. Furthermore, the uniformity of the edges is $k^{\text{\tiny{(7)}}}+1$ which is bounded by $m^{\text{\tiny{(7)}}}=|V(H)|$. This shows RP\ref{prop_rep1} from Definition~\ref{d.rep_sys} with $\chi_1=m^{\text{\tiny{(7)}}}$.

For RP\ref{prop_rep2}, observe that all the copies of $H$ in $K$ should contain one, and exactly one, vertex in each of the vertex clusters $G^t\times \{j\}$ in $V(K)$.  Given $H_0\in C(H,K)$ with $V(H_0)=\{g_1,\ldots,g_{m^{\text{\tiny{(7)}}}}\}$, the labels on the edges $e_1,\ldots,e_{m^{\text{\tiny{(7)}}}}$ of $H_0$ are given by $(l(e_1),\ldots,l(e_{m^{\text{\tiny{(7)}}}}))^{\top}=C(g_1,\ldots,g_{m^{\text{\tiny{(7)}}}})^{\top}$. Since $A^{\text{\tiny{(7)}}}C=0$, then $(l(e_1),\ldots,l(e_{m^{\text{\tiny{(7)}}}}))\in S(A^{\text{\tiny{(7)}}},G^t)$.  Let $Q=\{1\}$, $r_q(H_0)=1$  for each $H_0\in C(H,K)$ and $p=1$. Then we define
	\begin{align}
		r:C(H,K) &\longrightarrow  S(A^{\text{\tiny{(7)}}},G^t)\times Q\nonumber \\
		H_0=\{e_1,\ldots,e_{m^{\text{\tiny{(7)}}}}\}&\longmapsto (r_0(H_0),r_q(H_0))=\left((l(e_1),\ldots,l(e_{m^{\text{\tiny{(7)}}}})),1\right). \nonumber 	\end{align}

We claim that $r$ induces a strong $1$-representation for some $c$ bounded by a function of $m^{\text{\tiny{(7)}}}$. Let $(\mathbf{x}_1,\ldots,\mathbf{x}_{m^{\text{\tiny{(7)}}}})=\mathbf{x}\in S(A^{\text{\tiny{(7)}}},G^t)$ be a solution. For each $i\in[1,m^{\text{\tiny{(7)}}}]$, there are $|G^t|^{k^{\text{\tiny{(7)}}}}$ edges labelled $\mathbf{x}_i$. Indeed, as $\mathcal{C}_{i,i}$ is a block with determinant coprime with the order of the group $G^t$, for any choice of the vertices $\{g_{i+1},\ldots,g_{i+k}\}$ there is a unique vertex $g_{i}\in G\times \{i\}$ such that 
		$\sum_{j=i}^{i+k^{\text{\tiny{(7)}}} } \mathcal{C}_{i,j}(g_j)=\mathbf{x}_i$, namely $g_i=\mathcal{C}^{-1}_{i,i} (\mathbf{x}_i-\sum_{j=i+1}^{i+k^{\text{\tiny{(7)}}} } \mathcal{C}_{i,j}(g_j))$.\footnote{All the rational numbers appearing in $\mathcal{C}^{-1}_{i,i}$ have denominators co-prime with $|G^t|$, hence inducing automorphisms in $G$.}

Given the solution $\mathbf{x}$ and an edge $e_i$ coloured $i$ and labelled $\mathbf{x}_i$ in $K$, there is a unique $H_0$, copy of $H$ in $K$, with $r_0(H_0)=\mathbf{x}$ and $e_i\in H_0$. Indeed, if $e_i$ has the vertices $\{g_i,\ldots,g_{i+k^{\text{\tiny{(7)}}} }\}$ as its support, then there exists a unique $g_{i-1}\in G^t\times \{i-1\}\subset V(K)$ such that $l(\{g_{i-1},\ldots,g_{i+k^{\text{\tiny{(7)}}} -1}\})=\mathbf{x}_{i-1}$. This process can be repeated a total of
 $m^{\text{\tiny{(7)}}}-k^{\text{\tiny{(7)}}}-1$ times. Indeed, the vertices $g_{i-1},g_{i-2},\ldots,g_{i+k^{\text{\tiny{(7)}}}+1 }$ that complete, together with $g_i,\ldots,g_{i+k^{\text{\tiny{(7)}}}}$, a copy of $H$ in $K$ can be uniquely determined using the coordinate values $\mathbf{x}_{i-1},\ldots,\mathbf{x}_{i+k^{\text{\tiny{(7)}}}+1}$ of the solution $\mathbf{x}$ and a subset of the previously determined vertices. 

Let $H_0$ denoted this copy of $H$ in $K$. By the existence of $r$, $r_0(H_0)\in S(A^{\text{\tiny{(7)}}},G^t)$. Even more, we claim that $r_0(H_0)=\mathbf{x}$. Indeed, by the $n$--circularity of $A^{\text{\tiny{(7)}}}$, any $m^{\text{\tiny{(7)}}}-k^{\text{\tiny{(7)}}}$ consecutive values of $\{\mathbf{x}_j\}$ determines the solution. The copy $H_0$ has been constructed so that it contains edges labelled with the $m^{\text{\tiny{(7)}}}-k^{\text{\tiny{(7)}}}$ consecutive values $\mathbf{x}_i,\mathbf{x}_{i-1},\ldots,\mathbf{x}_{i+k^{\text{\tiny{(7)}}}+1}$ and coloured appropriately with $\{i,i-1,\ldots,i+k^{\text{\tiny{(7)}}}+1\}$.
Since the only solution $\mathbf{y}\in S(A^{\text{\tiny{(7)}}},G^t)$ that satisfies $(\mathbf{y})_j=\mathbf{x}_j$ for $j\in\{i,i-1,\ldots, i+k^{\text{\tiny{(7)}}}+1\}$ is $\mathbf{y}=\mathbf{x}$ and $r_0(H_0)\in S(A^{\text{\tiny{(7)}}},G^t)$, then $r_0(H_0)=\mathbf{x}$.

Hence, given $\mathbf{x}\in S(A^{\text{\tiny{(7)}}},G^t)$ and $i\in[1,m^{\text{\tiny{(7)}}}]$, each copy of $e_i$ coloured $i$ and labelled $\mathbf{x}_i$ can be extended to a unique copy of $H$ in $K$ related to $\mathbf{x}$. This shows RP\ref{prop_rep3} and RP\ref{prop_rep4}. Moreover, the number of copies of $H$ related to $\mathbf{x}$ in $K$ is
\begin{displaymath}
|r^{-1}(\mathbf{x},1)|=|G^t|^{k^{\text{\tiny{(7)}}}}=\frac{|G^t|^{k^{\text{\tiny{(7)}}}+1}}{|G^t|}=\left(\frac{1}{m^{\text{\tiny{(7)}}}}\right)^{k^{\text{\tiny{(7)}}}+1} \frac{|K|^{k^{\text{\tiny{(7)}}}+1}}{|G^t|}=c \frac{|K|^{k^{\text{\tiny{(7)}}}+1}}{|G^t|}
\end{displaymath}
as there are $|G^t|^{k^{\text{\tiny{(7)}}}}$ edges labelled $(\mathbf{x})_i$ for any $i\in[1,m^{\text{\tiny{(7)}}}]$. This shows RP\ref{prop_rep2} and finishes the strong $1$-representation for $(A^{\text{\tiny{(7)}}},G^t)$ by $(K,H)$.

\paragraph{Relation with previous systems.}

Observe that each of the systems $\{(J_\kappa,G_\kappa^t)\}_{\kappa\in \Upsilon}$ have one more equation and two more variables than $(A^{\text{\tiny{(6)}}},\Z_n^t)$ and that $(A^{\text{\tiny{(7)}}},G^t)$ have the same number of equations and variable as any $(\overline{J}_\kappa,G_{\kappa}^t)$. Thus, $A^{\text{\tiny{(7)}}}$ has dimensions
$k^{\text{\tiny{(7)}}}=(4k^{\text{\tiny{(6)}}}+5)(m^{\text{\tiny{(6)}}}-k^{\text{\tiny{(6)}}}+1)$ and $m^{\text{\tiny{(7)}}}=(4k^{\text{\tiny{(6)}}}+6)(m^{\text{\tiny{(6)}}}-k^{\text{\tiny{(6)}}}+1)$ over $G^t$.

If $x_i$ is a variable in $(A^{\text{\tiny{(7)}}},G^t)$, then it can be decomposed into $t$ variables
$x_{(i,j)}\in G$ with $j\in [1,t]$. Furthermore, each $x_{(i,j)}$ can be understood as formed by $mt+1$ variables
$x_{(i,j),\kappa}\in G_{\kappa}$, where $\kappa\in\Upsilon$.

\begin{remark}\label{r.last}
The system $(A^{\text{\tiny{(7)}}},G^t)$ is $\mu$-equivalent to $(A^{\text{\tiny{(6)}}},\Z_n^t)$ with the injection
\begin{align}
	\sigma: [1,m^{\text{\tiny{(6)}}}] &\longrightarrow [1,m^{\text{\tiny{(7)}}}] \nonumber \\
	i & \longmapsto 
	\left\{
	\begin{array}{ll}
	4(i-1) (m^{\text{\tiny{(6)}}}-k^{\text{\tiny{(6)}}}+1)+ 3(m^{\text{\tiny{(6)}}}-k^{\text{\tiny{(6)}}}+1)+1 & \text{ if }\; i\in[1,k^{\text{\tiny{(6)}}}] \\
	(4 k^{\text{\tiny{(6)}}}+5)(m^{\text{\tiny{(6)}}}-k^{\text{\tiny{(6)}}}+1)+i-k^{\text{\tiny{(6)}}} & \text{ if }\;  i\in[k^{\text{\tiny{(6)}}}+1,m^{\text{\tiny{(6)}}}]\\
\end{array} \right. \nonumber
\end{align}
and
\begin{align}
	\phi:S(A^{\text{\tiny{(7)}}},G^t)
 &\longrightarrow S(A^{\text{\tiny{(6)}}},\Z_n^t) \nonumber \\
(x_1,\ldots,x_{m^{\text{\tiny{(6)}}}}) &\longmapsto (\phi_1(x_{\sigma(1)}),\ldots,\phi_{m^{\text{\tiny{(6)}}}}(x_{\sigma(m^{\text{\tiny{(6)}}})})) \nonumber
\end{align}
with
\begin{displaymath}
		\phi_i(x_{(\sigma(i),j)})= \left\{
		\begin{array}{cl}
		d_{i,j} \sum_{\kappa\in \Upsilon} x_{(\sigma(i),j),\kappa} & \text{ for }\; i\in[1,k^{\text{\tiny{(6)}}}]\\
\sum_{\kappa\in \Upsilon} x_{(\sigma(i),j),\kappa} & \text{ for }\; i\in[k^{\text{\tiny{(6)}}}+1,m^{\text{\tiny{(6)}}}] \\
	\end{array}\right. ,
\end{displaymath}
where $d_{i,j}$ is the greatest common divisor of the $j$-th row of the block $\mathcal{B}_i^{\text{\tiny{(1)}}}$ from $A^{\text{\tiny{(6)}}}=\begin{pmatrix} I_{tm} & B^{\text{\tiny{(1)}}} \end{pmatrix}$.\footnote{See Section~\ref{s.group_on_B}.}

Additionally, for any $\mathbf{x}\in S(A^{\text{\tiny{(6)}}},\Z_n^t)$ and $i\in[1,m^{\text{\tiny{(6)}}}]$, the number of $\mathbf{y}\in S(A^{\text{\tiny{(7)}}},G^t)$ with $\phi(\mathbf{y})=\mathbf{x}$ and fixed $(\mathbf{y})_{\sigma(i)}$ is independent of the $(\mathbf{y})_{\sigma(i)}\in \phi_i^{-1}((\mathbf{x})_i)$.
\end{remark}

\begin{proof}[Proof of Remark~\ref{r.last}]
	The system $(A^{\text{\tiny{(7)}}},G^t)$ is formed by joining the systems $\{(\overline{J}_\kappa,G_{\kappa}^t)\}_{\kappa\in \Upsilon}$ together. The $k^{\text{\tiny{(J')}}}\times m^{\text{\tiny{(J')}}}$ system $(\overline{J}_\kappa,G_{\kappa}^t)$ is $1$-auto-equivalent to $(J_\kappa,G_{\kappa}^t)$ by Remark~\ref{r.ext_matrix_auto} for any  $\kappa\in \Upsilon$.\footnote{
	Indeed, $(\overline{J}_{\kappa},G_{\kappa}^t)$ is built from $(J_{\kappa},G_{\kappa}^t)$ by adding some variables and the same number of equations in a way that the new equations are: new variable equal linear equation involving old variables. Therefore, a projection onto the right variables using the identity as maps $\phi_i$ configure the application $\phi$ of the $1$-auto-equivalence.}
	Let $\phi'_{\kappa}$ be the map defining the $1$-auto-equivalence from $S(\overline{J}_{\kappa},G_{\kappa}^t)$ to $S(J_{\kappa},G_{\kappa}^t)$

Any solution $\mathbf{x}\in S(A^{\text{\tiny{(7)}}},G^t)$ induces $(mt+1)$ solutions $\overline{\mathbf{x}}_\kappa\in S(\overline{J}_{\kappa},G_{\kappa}^t)$ and vice-versa. By the $1$-auto-equivalence, these solutions can be seen in $(J_{\kappa},G_{\kappa}^t)$ considering $\mathbf{x}_\kappa=\phi_{\kappa}'(\overline{\mathbf{x}}_\kappa)\in S(J_{\kappa},G_{\kappa}^t)$.
We use the maps $f_{\kappa}$ from Remark~\ref{r.prop.G3} to conclude that
\begin{displaymath}
	\phi(\mathbf{x})=\sum_{\kappa\in \Upsilon} f_{\kappa}(\mathbf{x}_\kappa)\in S(A^{\text{\tiny{(6)}}},\Z_n^t)
\end{displaymath}
as $\phi$ is the sum over $\Upsilon$ of the compositions of the homomorphisms $f_{\kappa}$ with the $\phi_{\kappa}'$.

Since the homomorphism $f_{(1,0)}$ associated to
$(J_{(1,0)},\Z_n^t)=(J_{(1,0)},G_{(1,0)}^t)$ is surjective, so is $\phi$.\footnote{Observe that $S(J_{\kappa},G_{\kappa}^t)$ contains the trivial solution $\mathbf{0}$ and $\{f_{\kappa}\}_{\kappa\in\Upsilon}$ are homomorphisms.} As $\phi$ is a homomorphism, $\phi$ is $\mu$-to-$1$ for $\mu=|S(A^{\text{\tiny{(7)}}},G^t)|/|S(A^{\text{\tiny{(6)}}},\Z_n^t)|$.

Let us show the second part of the result for $\mathbf{x}\in S(A^{\text{\tiny{(6)}}},\Z_n^t)$. 
Given $\mathbf{y}_{\kappa}\in S(J_{\kappa},G_{\kappa}^t)$, 
 $(\mathbf{y}_{\kappa})_{i}\in G_{\kappa}^t$, $i\in[1,m^{\text{\tiny{(J)}}}]$, denotes the $i$-th coordinate of $\mathbf{y}_{\kappa}$ and $(\mathbf{y}_{\kappa})_{i,j}\in G_{\kappa}$ denotes the $(i,j)$-th coordinate of the solution with $(i,j)\in[1,m^{\text{\tiny{(J)}}}]\times[1,t]$.

Any collection of solutions $\mathbf{y}_{\kappa}\in S(J_{\kappa},G_{\kappa}^t)$ induce a unique $\mathbf{y}\in(A^{\text{\tiny{(7)}}},G^t)$. The variables indexed in $[k^{\text{\tiny{(6)}}}+1,m^{\text{\tiny{(6)}}}]$ parameterize the solutions in $S(A^{\text{\tiny{(6)}}},\Z_n^t)$. Therefore, if we have
\begin{equation}\label{eq.x_as_sum_y}
	(\mathbf{x})_i=\sum_{\kappa\in \Upsilon} (\mathbf{y}_{\kappa})_{i+1}, \text{ for all } i\in[k^{\text{\tiny{(6)}}}+1,m^{\text{\tiny{(6)}}}]
\end{equation} 
for our selected $\mathbf{x}$, then $\phi(\mathbf{y})=\mathbf{x}$. The condition (\ref{eq.x_as_sum_y}) is also necessary; if the collection of solutions $\{\mathbf{y}_{\kappa}\}_{\kappa\in \Upsilon}$ inducing $\mathbf{y}$ does not satisfy (\ref{eq.x_as_sum_y}) for some index $i\in [k^{\text{\tiny{(6)}}}+1,m^{\text{\tiny{(6)}}}]$, then $\phi(\mathbf{y})\neq \mathbf{x}$. 
The variables that parameterize the solutions for any system $(J_{\kappa},G_{\kappa}^t)$ are those indexed in $[k^{\text{\tiny{(J)}}}+1,m^{\text{\tiny{(J)}}}]$; once the value of $(\mathbf{y}_{\kappa})_i$ is selected for $i\in[k^{\text{\tiny{(J)}}}+1,m^{\text{\tiny{(J)}}}]$, the solution $\mathbf{y}_\kappa \in S(J_{\kappa},G_{\kappa}^t)$ exists and is unique.
As the proof of Remark~\ref{r.prop.G3} highlights, the variable $m^{\text{\tiny{(J)}}}$ does not appear in (\ref{eq.x_as_sum_y}); the value of $\phi(\mathbf{y})$ is independent of the values $(\mathbf{y}_{\kappa})_{m^{\text{\tiny{(J)}}}}$ for $\kappa\in \Upsilon$.

Pick an $i\in[k^{\text{\tiny{(6)}}}+1,m^{\text{\tiny{(6)}}}]$ and a value for $(\mathbf{y})_{i+1}\in G^t$ such that $\phi_i((\mathbf{y})_{i+1})=(\mathbf{x})_i$. All the solutions $\mathbf{y}\in S(A^{\text{\tiny{(7)}}},G^t)$ with $\phi(\mathbf{y})=\mathbf{x}$ can be found by selecting a value for the remaining parameterizing variables $(\mathbf{y})_{i_1}$ with $i_1\in [k^{\text{\tiny{(J)}}}+1,m^{\text{\tiny{(J)}}}]\setminus\{i+1\}$ appropriately to configure a solution in $S(A^{\text{\tiny{(7)}}},G^t)$ with  $\phi(\mathbf{y})=\mathbf{x}$.
For $i_1\in [k^{\text{\tiny{(J)}}}+1,m^{\text{\tiny{(J)}}}-1]\setminus\{i+1\}$, we can select any $(\mathbf{y})_{i_1}\in G^t$ as long as
\begin{displaymath}
(\mathbf{x})_{i_1-1}=\sum_{\kappa \in\Upsilon} (\mathbf{y}_{\kappa})_{i_1}=\phi_{i_1-1}((\mathbf{y})_{i_1}).
\end{displaymath}
Additionally, we can select any value for $(\mathbf{y})_{m^{\text{\tiny{(J)}}}}$. Observe that the number of choices is independent on the particular value $(\mathbf{y})_{l+1}$ with $\phi_l((\mathbf{y})_{l+1})=(\mathbf{x})_l$ we have picked.

Given $\mathbf{x}\in S(A^{\text{\tiny{(6)}}},\Z_n^t)$, $i\in[1,k^{\text{\tiny{(6)}}}]$ and any $y^{(i)}\in G^t$ with $\phi_i(y^{(i)})=(\mathbf{x})_i$ we shall find all the solutions $\mathbf{y}\in S(A^{\text{\tiny{(7)}}},G^t)$ with $\phi(\mathbf{y})=\mathbf{x}$ and $(\mathbf{y})_i=y^{(i)}$. For $\kappa\in[1,m]\times[1,t]$, select any solution $\mathbf{y}_{\kappa}'''\in S(J_{\kappa},G_{\kappa}^t)$ such that
\begin{equation}\label{eq.23}
	(\mathbf{y}_{\kappa}''')_{i}=y^{(i)}_{\kappa}.
	\end{equation} Pick an auxiliary solution $\mathbf{y}_{(1,0)}'\in S(J_{(1,0)},G_{(1,0)}^t)$, defined by the last $m^{\text{\tiny{(J)}}}-k^{\text{\tiny{(J)}}}$ variables,
\begin{itemize}
	\item  choosing a value for $(\mathbf{y}_{(1,0)}')_{m^{\text{\tiny{(J)}}}}$  in $\Z_n^t$ (any value.)
	\item $(\mathbf{y}_{(1,0)}')_{j+1}=(\mathbf{x})_{j}-\sum_{\kappa\in\Upsilon\setminus \{(1,0)\}}(\mathbf{y}_{\kappa}''')_{j+1}$ for $j\in [k^{\text{\tiny{(6)}}}+1,m^{\text{\tiny{(6)}}}]$.
\end{itemize}
Let $\mathbf{y}'$ be the solution in $S(A^{\text{\tiny{(7)}}},G^t)$ defined by the $mt+1$ solutions $\{\mathbf{y}_{(1,0)}',\{\mathbf{y}_{\kappa}'''\}_{\kappa\in \Upsilon \setminus \{(1,0)\}}\}$. Observe that $\phi(\mathbf{y}')=\mathbf{x}$. Indeed, for $j\in [k^{\text{\tiny{(6)}}}+1,m^{\text{\tiny{(6)}}}]$,
\begin{displaymath}
	\phi_{j}(\mathbf{y}')=(\mathbf{y}_{(1,0)}')_{j+1}+\sum_{\kappa\in \Upsilon\setminus \{(1,0)\} } (\mathbf{y}_{\kappa}''')_{j+1}=(\mathbf{x})_{j} -\sum_{\kappa\in \Upsilon\setminus \{(1,0)\} } (\mathbf{y}_{\kappa}''')_{j+1}+\sum_{\kappa\in \Upsilon\setminus \{(1,0)\} } (\mathbf{y}_{\kappa}''')_{j+1}= (\mathbf{x})_{j}.
\end{displaymath}
Since $\phi(\mathbf{y}')\in S(A^{\text{\tiny{(6)}}},\Z_n^t)$ and $\mathbf{x}$ is unique given the value of its last $m^{\text{\tiny{(6)}}}-k^{\text{\tiny{(6)}}}$ variables, the claim follows and we have a solution $\mathbf{y}'\in S(A^{\text{\tiny{(7)}}},G^t)$ such that $\phi(\mathbf{y}')=\mathbf{x}$.

However, it is not yet clear that $(\mathbf{y})_i=y^{(i)}$; by (\ref{eq.23}), we know this equality holds for all coordinates in $\Upsilon\setminus (1,0)$.
If $\mathbf{y}_{(1,0)}'$ would be such that $(\mathbf{y}_{(1,0)}')_j=y^{(i)}_{(1,0)}$, then the solution $\mathbf{y}'\in S(A^{\text{\tiny{(7)}}},G^t)$ defined before satisfies the claims $\phi(\mathbf{y}')=\mathbf{x}$ and $(\mathbf{y}')_i=y^{(i)}$.
Let us define $\epsilon_{i,j}=\epsilon_{i,j}(\mathbf{x},\{\mathbf{y}_{\kappa}'''\}_{\kappa\in \Upsilon \setminus\{(1,0)\}},y^{(i)})$, for $j\in[1,t]$, as the difference between the $j$-th components of the $i$-th coordinate of $\mathbf{y}_{(1,0)}'$ and its aimed value $(y^{(i)}_{(1,0)})_j$:
\begin{displaymath}
	(\mathbf{y}_{(1,0)}')_{i,j}-(y^{(i)}_{(1,0)})_j=\epsilon_{i,j}(\mathbf{x},\{\mathbf{y}_{\kappa}'''\}_{\kappa\in \Upsilon \setminus\{(1,0)\}},y^{(i)}).
\end{displaymath}
Observe that $\epsilon_{i,j}\in G_{i,j}$. Indeed,
\begin{displaymath}
	(\phi_i(y^{(i)}))_j=d_{i,j} \left(\sum_{\kappa\in \Upsilon} (y^{(i)}_{\kappa})_{j}\right)=(\mathbf{x})_{i,j}\stackrel{\phi(\mathbf{y}')=\mathbf{x}}{=}(\phi_i(\mathbf{y}'))_{j}=d_{i,j}\left((\mathbf{y}'_{(1,0)})_{i,j}+\sum_{\kappa\in \Upsilon \setminus \{(1,0)\}} (y_{\kappa}^{(i)})_j\right).
\end{displaymath}
Therefore $d_{i,j} \epsilon_{i,j}=d_{i,j} \left((\mathbf{y}_{(1,0)}')_{i,j}-(y^{(i)}_{(1,0)})_j\right)=0$ as claimed.

Using $\epsilon_{i,j}$ we pick, one for each $j\in[1,t]$, a total of $t$ auxiliary solutions $\mathbf{y}^{(j)}_{(1,0)}\in S(J_{(1,0)},G_{(i,j)}^t)\subset S(J_{(1,0)},G_{(1,0)}^t)$, such that
\begin{itemize}
	\item  $(\mathbf{y}_{(1,0)}^{(j)})_{m^{\text{\tiny{(J)}}}}=0$.
	\item $(\mathbf{y}_{(1,0)}^{(j)})_{i,j}=\epsilon_{i,j}$.
	\item $ (\mathbf{y}_{(1,0)}^{(j)})_{i,r}=0$ for $r\in[1,t]\neq j$.\footnote{There are  $|G_{(i,j)}^t|^{m^{\text{\tiny{(J)}}}-k^{\text{\tiny{(J)}}}-2}$ such solutions for each $j\in[1,t]$ by Observation~\ref{o.number_of_solutions}. We only select one per each $j$.}
\end{itemize}
These $\{\mathbf{y}^{(j)}_{(1,0)}\}_{j\in[1,t]}$ exist because the greatest common divisor of the coefficients of the $((i-1)t+j)$-th row of $B^{\text{\tiny{(3)}}}_{(1,0)}$ from $J_{(1,0)}$, that define the $(i,j)$-th variable, is $1$.

Consider the collection of $t$ solutions $\{\mathbf{y}_{(i,j)}''\}_{j\in[1,t]}$, $\mathbf{y}_{(i,j)}''\in S(J_{(i,j)},G_{(i,j)}^t)$, that are determined by sharing the last $m^{\text{\tiny{(J)}}}-k^{\text{\tiny{(J)}}}$ variables with $\mathbf{y}_{(1,0)}^{(j)}$: $(\mathbf{y}_{(i,j)}'')_r=(\mathbf{y}_{(1,0)}^{(j)})_{r}$ for $r\in[k^{\text{\tiny{(J)}}}+1,m^{\text{\tiny{(J)}}}]$. Observe that 
$(\mathbf{y}_{(i,j)}'')_i=0$ since
\begin{itemize}
	\item $(\mathbf{y}_{(i,j)}'')_{i,j}=0$ as $(\mathbf{y}_{(i,j)}'')_{i,j}=(\mathbf{y}_{(i,j)}'')_{m^{\text{\tiny{(J)}}},1}=(\mathbf{y}_{(1,0)}^{j})_{m^{\text{\tiny{(J)}}},1}=0$.
	\item For $r\neq j$, the equations defining  $(\mathbf{y}_{(i,j)}'')_{i,r}$ and $(\mathbf{y}_{(1,0)}^{(j)})_{i,r}$ from the last $m^{\text{\tiny{(J)}}}-k^{\text{\tiny{(J)}}}$ variables are the same. Since $(\mathbf{y}_{(1,0)}^{(j)})_{i,r}=0$, then so is $(\mathbf{y}_{(i,j)}'')_{i,r}$. 
\end{itemize}

Let $\mathbf{y}$ be the solution in $S(A^{\text{\tiny{(7)}}},G^t)$ formed by
\begin{itemize}
	\item $\mathbf{y}_{\kappa}=\mathbf{y}_{\kappa}'''$ for $\kappa\in([1,m]\setminus\{i\})\times [1,t]$.
	\item $\mathbf{y}_{\kappa}=\mathbf{y}_{\kappa}''+\mathbf{y}_{\kappa}'''$ for $\kappa\in\{i\}\times [1,t]$.
	\item $\mathbf{y}_{(1,0)}=\mathbf{y}'_{(1,0)}-\sum_{j=1}^t \mathbf{y}_{(1,0)}^{(j)}$.
\end{itemize}
Observe that
\begin{itemize}
	\item for $\kappa\in \Upsilon$, $\mathbf{y}_{\kappa}\in S(J_{\kappa},G_{\kappa}^t)$ as $\mathbf{y}_{(1,0)}^{(j)}\in S(J_{(1,0)},G_{(i,j)}^t)\subset S(J_{(1,0)},G_{(1,0)}^t)$ and $\mathbf{y}_{(i,j)}''\in S(J_{(i,j)},G_{(i,j)}^t)$ for all $j\in[1,t]$.
	\item $\phi(\mathbf{y})=\mathbf{x}$ as $\phi(\mathbf{y})\in S(A^{\text{\tiny{(6)}}},\Z_n^t)$ and, for $r\in [k^{\text{\tiny{(6)}}}+1,m^{\text{\tiny{(6)}}}]$,
	\begin{align}
		\phi_r(\mathbf{y})&=(\mathbf{y}_{(1,0)})_{r+1}+\sum_{\kappa\in([1,m]\setminus\{i\})\times [1,t]} (\mathbf{y}_{\kappa})_{r+1}+\sum_{j\in[1,t]}(\mathbf{y}_{(i,j)})_{r+1} \nonumber \\
		&=
		(\mathbf{y}_{(1,0)}')_{r+1}-\sum_{j=1}^t (\mathbf{y}_{(1,0)}^{(j)})_{r+1}
		+\nonumber \\
				&\qquad
				+\sum_{\kappa\in([1,m]\setminus\{i\})\times [1,t]} (\mathbf{y}_{\kappa}''')_{r+1}+\sum_{j\in[1,t]}\left((\mathbf{y}_{(i,j)}'')_{r+1} + (\mathbf{y}_{(i,j)}''')_{r+1}\right) \nonumber \\
		&=(\mathbf{y}_{(1,0)}')_{r+1}+\sum_{\kappa\in\Upsilon \setminus\{(1,0)\}}(\mathbf{y}_{\kappa}''')_{r+1}=(\mathbf{x})_{r} \nonumber
\end{align}
\item $(\mathbf{y})_i=y^{(i)}$. Indeed, for $\kappa\in ([1,m]\setminus\{i\})\times [1,t]$, $(\mathbf{y}_{\kappa})_i=(\mathbf{y}_{\kappa}''')_i=y^{(i)}_{\kappa}$ by hypothesis. Since $(\mathbf{y}_{(i,j)}'')_i=0$, $(\mathbf{y}_{(i,j)})_i=(\mathbf{y}_{(i,j)}'')_i+(\mathbf{y}_{(i,j)}''')_i=y^{(i)}_{(i,j)}$ for $j\in[1,t]$. Additionally, for each $j\in[1,t]$,
\begin{displaymath} (\mathbf{y}_{(1,0)})_{i,j}=(\mathbf{y}'_{(1,0)})_{i,j}-\sum_{r=1}^t (\mathbf{y}_{(1,0)}^{(r)})_{i,j}=(\mathbf{y}'_{(1,0)})_{i,j}-(\mathbf{y}_{(1,0)}^{(j)})_{i,j}=(\mathbf{y}'_{(1,0)})_{i,j}-\epsilon_{i,j}=
	(y^{(i)}_{(1,0)})_j.
\end{displaymath}
\end{itemize}

Therefore, the set of solutions $\{\mathbf{y}_{\kappa}\}_{\kappa\in\Upsilon}$ provides a unique solution $\mathbf{y}$ with $\phi(\mathbf{y})=\mathbf{x}$ and $(\mathbf{y})_i=y^{(i)}$ as desired. Using Observation~\ref{o.number_of_solutions} on the matrices $J_{\kappa}$, the number of choices made to find $\mathbf{y}$ is independent of the particular $y^{(i)}$. Even more, the excesses $\epsilon_{i,j}$ define a quotient structure among the possible choices of $\{\mathbf{y}_\kappa'''\}_{\kappa\in\Upsilon\setminus \{(1,0)\}}$ and value for $(\mathbf{y}_{(1,0)}')_{m^{\text{\tiny{(J)}}}}$ (the other choices for the solutions $\mathbf{y}^{(j)}_{(1,0)}$ can be thought to be fixed depending on $\epsilon_{i,j}$.) This finishes the proof of the remark.
\end{proof}

\begin{remark}
Since the system is $n$--circular and $|G|$ is a divisor of $n^\alpha$, for some positive integer $\alpha$, $S_i(A^{\text{\tiny{(7)}}},G^t)=G^t$ for all $i\in[1,m^{\text{\tiny{(7)}}}]$.
\end{remark}

\subsection{Observation on adding variables} \label{s.adding_variables}

The number of variables in $A^{\text{\tiny{(7)}}}$, as well as its relative order, is the same as for matrix $\overline{J}_{\kappa}$. This is, the $i$-th variable in $\overline{J}_{\kappa}$, seen as a system $(\overline{J}_{\kappa},G_{\kappa}^t)$, is one of the coordinates that configure the $i$-th variable in the system $(A^{\text{\tiny{(7)}}},G^t)$.

The hypergraph $H$ used in the representation of $(A^{\text{\tiny{(7)}}},G^t)$ can be obtained in the following way. Consider the original $k\times m$ system $(A,G_0)$. Let $H_0$ be a $(k+1)$-uniform hypergraph over the vertex set $V=\{v_1,\ldots,v_{m}\}$ with edges $e_i=\{v_{i},\ldots,v_{i+k} \}$, indices modulo $m$. 
Let us pair the edge $e_i$ with the $i$-th variable of $A$, $x_i$. Let $e_i=\{v_i,\ldots,v_{i+k} \}$ and $e_j=\{v_j,\ldots,v_{j+k}\}$ be the edges associated with  $x_i$ and $x_j$ respectively. The variable $x_i$ is said to be \emph{before} $x_j$ if 
$v_j\in e_i$. If $x_i$ is before $x_j$ then the variables $x_{i+1},\ldots,x_{j-1}$ are said to be \emph{in-between} $x_i$ and $x_j$. 

Assume that $(A',G')$ is built from $(A,G_0)$ by adding a new variable and a new equation. Let $H_1$ denote the hypergraph associated with $(A',G')$ as described in the procedure above. Alternatively, $H_1$ can be constructed from $H_0$ as follows.
When a new variable $x_l$ is added to the system $(A,G)$, a new vertex $v_{m+1}$ is added to $V(H)$. If a new equation is added, the uniformity of the edges increases by $1$. 
The edges $e_1,\ldots,e_{l-1}$ in $H_1$ start at the same point as the 
corresponding edges in $H_0$ and have one more vertex in them. The edges 
$e_{l+1},\ldots,e_{m+1}$ in $H_1$ start one vertex later and have one additional 
vertex than the corresponding ones in $H_0$ (they finish two vertices later than 
$e_{l},\ldots,e_{m}$ in $H_0$.) In particular, if two edges $e_i$ and $e_j$, 
with associated variables $x_i$ and $x_j$, share a vertex in $H_0$ and $x_i$ is 
before $x_j$, then the corresponding edges in $H_1$ share the same number of 
vertices if the new variable $x_l$ is between $x_i$ and $x_j$.\footnote{Here 
$x_l$ is assumed to be between the variables corresponding to $x_i$ and $x_j$ in 
$H_1$.} If the added variable $x_l$ is not between them, then the corresponding 
edges to $e_i$ and $e_j$ in $H_1$ share an additional vertex. Moreover, some 
edges that did not share any vertex in $H_0$, may share a vertex in $H_1$. 

If $(A',G')$ is built from $(A,G_0)$ by adding only a new variable, then $H_1$ is built from $H_0$ by adding a new vertex $v_{m+1}$ but not increasing the uniformity of the edges. If $x_i$ is before $x_j$ in $H_0$ and the new variable $x_l$ is between $x_i$ and $x_j$, then $e_i$ and $e_{j+1}$, the edges corresponding to $x_i$ and $x_j$ in $(A',G')$ intersect in one vertex less. Otherwise, the intersection between the new edges does not change with respect to the intersection of the corresponding edges in $H_0$.

Observe that the hypergraph $H$  provided by the procedures from Section~\ref{s.final_composition} to represent  $(A^{\text{\tiny{(7)}}},G^t)$ can be thought of as coming from an embryonal cycle $H_0$ as described above. 
Let $H_2$ be the hypergraph obtained from $H$ by removing the edges related to the added variables from  $(A,G_0)$ to $(J_{\kappa},G_{\kappa}^t)$.
Observe that
most of the added variables involve the addition of equations. In these cases, removing the edges from $H$ represents no problems in terms of connectedness of $H_2$. However, some variables were added without the addition of any equation. Those free variables correspond to the parts of the construction dealing with $\gamma\neq 1$, Section~\ref{s.gamma-effective}, and to simulate the independent vectors in Section~\ref{s.union_of_systems}. Since the number of variables used in the simulation of the independent vectors is less than $k$, the $k+1$-uniformity of the hypergraph embryo $H_0$ allows $H_2$ to be connected.

The remaining cases involve the additional variable added in Section~\ref{s.gamma-effective} to each of the systems $(J_{\kappa},G_{\kappa}^t)$ with respect to $(A^{\text{\tiny{(6)}}},\Z_n^t)$. Since $m^{\text{\tiny{(6)}}}-k^{\text{\tiny{(6)}}}\geq 2$, the procedure from Section~\ref{s.construction_circular_matrix} ads at least four variables between any pair of the first $k^{\text{\tiny{(J)}}}$ variables from $(J_{\kappa},G_\kappa^t)$. Notice that the first $k^{\text{\tiny{(6)}}}$ variables of $J_{\kappa}$ contains the original set of $m$ variables of $(A,G_0)$.  Hence, the set of added variables is non-empty and well distributed throughout the original set of variables, with several variables placed in-between original ones and others before, increasing the connectedness. This justifies that $H_2$ is connected and supported over the same set of vertices as $H$.

\subsection{From the representation of $(A^{\text{(7)}},G^t)$ to $(A,\Z_n^t)$} \label{s.unwrap_const}

Let us summarize the steps followed to construct $(A^{\text{\tiny{(7)}}},G^t)$ and its relation with the previous systems. Recall that the original homomorphism system $(A',\prod_{i=1}^t\Z_{n_i})=(A_0,G_0)$ has dimensions $k\times m$. Moreover, we assume that $m\geq k+2$. 

\begin{center}\label{figure.summary}
{
\footnotesize
%\scriptsize
\begin{tabular}[c]{|p{1.67cm}|p{1.67cm}|l|p{2.65cm}|p{4.53cm}|l|}
	\hline
	\textbf{From} & \textbf{To} & \textbf{Relation} & \textbf{Dimensions} & \textbf{Description} & \textbf{In} \\ \hline
	$(A_0,G_0)$ & $(A,\Z_n^t)$ & $\mu$-equivalent 1 & $k$ eq., $m$ var. & From $G_0=\prod_{i=1}^t \Z_{n_i}$ to $\Z_{n_1}^t$ & \ref{s.repr_for_Zt_implies_G} \\ \hline
	
$(A,\Z_n^t)$ & $(A^{\text{\tiny{(1)}}},\Z_n^t)$ & equivalent & $k$ eq., $m$ var. & row reduction & \ref{s.union_of_systems} \\ \hline

$(A^{\text{\tiny{(1)}}},\Z_n^t)$ & $(A^{\text{\tiny{(3)}}},\Z_n^t)$ & $\mu$-auto-equiv. & $k$ eq., \newline $m+k=m^{\text{\tiny{(3)}}}$ var. % \newline $l\in[0,k]$ 
& from determinantal $n$ to $1$, independent vector simulation & \ref{s.union_of_systems} \\ \hline

$(A^{\text{\tiny{(3)}}},\Z_n^t)$ & $(A^{\text{\tiny{(4)}}},\Z_n^t)$ 
& $1$-auto-equiv. & $m^{\text{\tiny{(3)}}}$ eq., \newline $2m^{\text{\tiny{(3)}}}- k^{\text{\tiny{(3)}}}$ var. & determinantal 1 \newline to determinant 1& \ref{s.determinantal_to_determinant}\\ \hline

$(A^{\text{\tiny{(4)}}},\Z_n^t)$ & $(A^{\text{\tiny{(5)}}},\Z_n^t)$ 
& equivalent & $m^{\text{\tiny{(3)}}}\times 2m^{\text{\tiny{(3)}}}- k^{\text{\tiny{(3)}}}$& row reduction to $\left( I_{m^{\text{\tiny{(3)}}}}  \;B \right)$ & \ref{s.determinantal_to_determinant}\\ \hline

$(A^{\text{\tiny{(5)}}},\Z_n^t)$ & $(A^{\text{\tiny{(6)}}},\Z_n^t)$ 
& $1$-auto-equiv. & $k^{\text{\tiny{(5)}}}$ eq., $m^{\text{\tiny{(5)}}}$ var.& row-reduce $t$-row blocks in $B$; \newline product of $\gcd$ of the rows is \newline the determinantal of the block & \ref{s.group_on_B}\\ \hline

$(A^{\text{\tiny{(6)}}},\Z_n^t)$ & $(J_{\kappa},G_{\kappa}^t)$, \newline  with $\kappa \in\Upsilon$
& splitting  & $k^{\text{\tiny{(6)}}}+1$ eq., \newline $m^{\text{\tiny{(6)}}}+2$ var. each & find systems $J_{\kappa}=\left( I_{k^{\text{\tiny{(6)}}}} \; B \right)$ \newline with $D_{t}(B_{\text{\tiny{[ti+1,ti+t]}}})=1$ 
& \ref{s.gamma-effective}\\ \hline

$(J_{\kappa},G_{\kappa}^t)$, \newline with $\kappa\in\Upsilon$ &
$\left(\overline{J}_{\kappa},G_{\kappa}^t\right)$, \newline with $\kappa\in \Upsilon$
& $1$-auto-equiv. & $k^{\text{\tiny{(J')}}}$ eq.,\newline $m^{\text{\tiny{(J')}}}$ var. each & find $n$-circular systems \newline for $(J_{\kappa},G_{\kappa}^t)$ & \ref{s.construction_circular_matrix}\\ \hline

 $\left(\overline{J}_{\kappa},G_{\kappa}^t\right)$ \newline with $\kappa\in \Upsilon$ & $(A^{\text{\tiny{(7)}}},G^t)$
& joining & $k^{\text{\tiny{(J')}}}=k^{\text{\tiny{(7)}}}$ eq.,\newline $m^{\text{\tiny{(J')}}}=k^{\text{\tiny{(7)}}}$ var. & group the systems $(\overline{J}_{\kappa},G_{\kappa}^t)$ \newline  in a single one & \ref{s.final_composition}\\ \hline

$(A^{\text{\tiny{(6)}}},\Z_n^t)$ & $(A^{\text{\tiny{(7)}}},G^t)$ &
$\mu$-equivalent 2 & $k^{\text{\tiny{(7)}}}$ eq.,\newline $m^{\text{\tiny{(7)}}}$ var. & conclusion from joining \newline the systems & \ref{s.final_composition}\\ \hline

\end{tabular}}
\end{center}

Where $k^{\text{\tiny{(J')}}}=(4 k^{\text{\tiny{(J)}}}+1)(m^{\text{\tiny{(J)}}}-k^{\text{\tiny{(J)}}})$, $m^{\text{\tiny{(J')}}}=(4 k^{\text{\tiny{(J)}}}+2)(m^{\text{\tiny{(J)}}}-k^{\text{\tiny{(J)}}})$, and $G=\prod_{\kappa\in \Upsilon} G_{\kappa}$.

We prove the first part of Theorem~\ref{t.rem_lem_ab_gr} under the conditions $m\geq k+2$ by concatenating Proposition~\ref{p.mu-equivalent_2}, \ref{p.mu-equivalent_1}, \ref{p.mu-auto-equivalent} and \ref{p.1-auto-equiv-rep} between the different pairs of $\mu$-equivalent systems appropriately. For instance we shall use
\begin{itemize}
	\item Proposition~\ref{p.mu-equivalent_2} from $(A^{\text{\tiny{(7)}}},G^t)$ to $(A^{\text{\tiny{(6)}}},\Z_n^t)$ by Remark~\ref{r.last}.
	\item Proposition~\ref{p.1-auto-equiv-rep} from $(A^{\text{\tiny{(6)}}},\Z_n^t)$ to $(A^{\text{\tiny{(5)}}},\Z_n^t)$ by Remark~\ref{r.4}.
	\item No proposition is needed from $(A^{\text{\tiny{(5)}}},\Z_n^t)$ to $(A^{\text{\tiny{(4)}}},\Z_n^t)$ as they are equivalent.
	\item Proposition~\ref{p.1-auto-equiv-rep} from $(A^{\text{\tiny{(4)}}},\Z_n^t)$ to $(A^{\text{\tiny{(3)}}},\Z_n^t)$ by Remark~\ref{r.3}.
	\item Proposition~\ref{p.mu-auto-equivalent} from $(A^{\text{\tiny{(3)}}},\Z_n^t)$ to $(A^{\text{\tiny{(1)}}},\Z_n^t)$ by Remark~\ref{r.2}.
	\item Proposition~\ref{p.mu-equivalent_1} from $(A^{\text{\tiny{(1)}}},\Z_n^t)$ to $(A_0,G_0)$ by Remark~\ref{r.1} (with $A^{\text{\tiny{(1)}}}=A'$ and $(A_0,G_0)=(A,G)$.)
\end{itemize}
These propositions can be concatenated by Section~\ref{s.adding_variables}. Indeed, the edges related to the remaining variables, after the composition of the maps defining the $\mu$-equivalences, still cover all the vertices of the initial hypergraph $H$. Thus this shows the following proposition.

\begin{proposition}[Representation for homomorphisms] \label{p.repr_hom}
	Let $G$ be a finite abelian group and let $m,k$ be two positive integers with $m\geq k+2$. Let $A$ be a homomorphism $A:G^m\to G^k$ and let $\mathbf{b}\in G^k$ be given. Then the system of linear configurations $((A,\mathbf{b}),G)$ is $\gamma$-strongly-representable with $\gamma_i=|G|/|S_i(A,G)|$ and where $\chi_1,\chi_2$ depend only on $m$.
\end{proposition}

By means of Theorem~\ref{t.rep_sys_rem_lem}, Proposition~\ref{p.repr_hom} proves the first part of Theorem~\ref{t.rem_lem_ab_gr} when $m\geq k+2$.  The second part of Theorem~\ref{t.rem_lem_ab_gr} and the treatment of the cases when $m<k+2$ are proved in Section~\ref{s.finish_rem_lem_dkA1}.

\section{Proof of Theorem~\ref{t.rem_lem_ab_gr}: second part and the cases $m<k+2$}
\label{s.finish_rem_lem_dkA1}

 The cases where $k>m$ can be reduced to $m=k$ by eliminating the redundant equations (for instance, thinking of them as equations in $\Z_n$ using Section~\ref{s.repr_for_Zt_implies_G} and Section~\ref{s.hom_mat_to_integer_mat} and performing Gaussian elimination on the matrices only allowing integer operations.)

Let $G_0$ be a finite abelian group. If the system $A\mathbf{x}=0$, $\mathbf{x}\in \prod_{i=1}^m X_i$, with $\left|S\left(A,G_0,\prod_{i=1}^m X_i\right)\right| < \delta |S(A,G_0)|$ has $k$ equations and $m$ variables, then
$A'\mathbf{x}=0$, $\mathbf{x}\in \prod_{i=1}^m X_i \times G_0^2$, with
\begin{displaymath}
	A'=\left(A
	\left|\begin{array}{cc}
		|G_0|  &|G_0| \\
		0  & 0 \\
		\vdots & \vdots \\
		0  & 0 \\
	\end{array}\right.\right)
\end{displaymath}
has $k$ equations in $m+2$ variables and
$\left|S\left(A',G_0,\prod_{i=1}^m X_i\times G_0^2\right)\right| < \delta |S(A',G_0)|$. Therefore, the cases $m<k+2$ can be proved using the second part of Theorem~\ref{t.rem_lem_ab_gr} restricted to the case $m\geq k+2$. This is encapsulated in Observation~\ref{lem:red2_ext_3}.

\begin{observation}[Do not remove from full sets] \label{lem:red2_ext_3}
Assume that the $\gamma$-representation of the $k\times m$ homomorphism system $(A,G_0)$, with $m\geq k+2$, is constructed using a system with an $n$--circular matrix $(A^{\text{\tiny{(7)}}},G)$ by the methods exposed in sections \ref{s.rep_indep_vector} through \ref{s.unwrap_const}. Let $I\subset[1,m]$ be a set of indices such that $X_i=G_0$, $i\in I$. Then Theorem~\ref{t.rem_lem_ab_gr} holds with $X_i'=\emptyset$.
\end{observation}

\begin{proof}[Proof of Observation~\ref{lem:red2_ext_3}] Suppose $m\geq k+2$.
	By reordering the variables, assume that $I=[j,m]$ for some $j$. Represent the system $(A,G_0)$ by the $s$-uniform hypergraph pair $(K_0,H)$. By the construction in Section~\ref{s.final_composition} and Section~\ref{s.unwrap_const} and the comments in Section~\ref{s.adding_variables}, $H$ is connected and $K_0$ is $|H|$-partite. Let $V_1,\ldots,V_{|H|}$ denote the stable sets in $K_0$ and recall that $|V_i|=|V_j|$ for $i,j\in[1,|H|]$. Let $K=K_0\left(\prod_{i=1}^m X_i\right)$ represent the hypergraph where only the edges labelled $x_i\in X_i$ and colored $i$ appear. Let $K'$ and $H'$ be the hypergraphs obtained  by removing the edges colored $i\in[j,m]$ from $K$ and $H$ respectively. Assume $e_1=\{1,\ldots,s\}\subset H$ is the edge related to $x_1$.

Let $d$ be the number of vertices in $H'$ with no edges. By the construction of the representation from the $n$--circular matrix provided in the third part of Section~\ref{s.final_composition} and knowing that the deleted edges have the largest indices, the stable sets $V_i$ (clusters of vertices) of $K'$ with no edge incident with them are the consecutive ones $V_{|H|-d+1},\ldots,V_{|H|}$.
Since every subset of $s$ clusters from $K'$ span, at most, one color class of edges, the connected component of $H'$ has one vertex in each of the first $|H|-d$ clusters. The isolated vertices of $H'$ can be placed in any cluster
$V_i, i\in[1,|H|]$.
Thus, each copy of $H$ in $K\left(\prod_{i=1}^m X_i\right)$ generates ${|K|-(|H|-d) \choose d}\approx c'|K|^d$ copies of $H'$ in $K'$, for some constant $c'$.

Let $H''$ be the hypergraph built from $H'$ by removing its isolated vertices. Let $K''$ be built from $K'$ by removing the isolated stables $V_{|H|-d+1},\ldots,V_{|H|}$. Then $(K'',H'')$ is a representation of the system $(A,G_0)$ where $(x_1,\ldots,x_{j-1})$ is a solution if and only if there is some $(x_j,\ldots,x_m)\in G_0^{m-j+1}$ for which $(x_1,\ldots,x_{j-1},x_j,\ldots,x_m)$ is a solution. Indeed, we can extend any copy $H''$ in $K''$ to $\left|\prod_{i=|H|-d+1}^{|H|} V_i\right|=|V_i|^d=c|K|^d$ copies of $H$ in $K$ by selecting one additional vertex in each of the last $d$ clusters.
Since $X_i=G_0$, $K$ has all the possible edges coloured $i\in I$ in the vertices $\{(\cdot,j)\}_{j\in \{i_1,\ldots,i_s\}}$, where $\{i_1,\ldots,i_s\}$ is the support of the edge coloured $i$ in $H$. Therefore, any choice in $\prod_{i=|H|-d+1}^{|H|} V_i$ completes a copy of $H''$ in $K''$ to a copy of $H$ in $K$.
 On the other hand, any copy of $H$ in $K$ generates one copy of $H''$ in $K''$. Thus,
\begin{displaymath}
	|C(H, K)|=|C(H'', K'')| \;|V_i|^d=|C(H'', K'')|\; c|K|^d
\end{displaymath}
and the proportions $|K|^{|H|}/|C(H, K)|$ and $|K''|^{|H''|}/|C(H'', K'')|$ are such that
\begin{displaymath}
	\frac{|K|^{|H|}}{|C(H, K)|}=\frac{|K|^{|H|}}{c|K|^d |C(H'', K'')|}=c''\frac{|K''|^{|H''|}}{|C(H'', K'')|}
\end{displaymath}
where $c''$ only depends on $m$. 

Therefore, if $|S(A,G_0,\prod_{i=1}^{j-1} X_i\times G_0^{m-j+1})|<\delta |S(A,G_0)|$, then, by the representability for $(A,G_0)$ by $(K,H)$, $|C(H,K)|<\delta' |K|^{|H|}$. Hence $|C(H'',K'')|<\delta'' |K''|^{|H''|}$.
By applying the same procedure of Theorem~\ref{t.rep_sys_rem_lem} to $(K'',H'')$, we show a removal lemma by obtaining sets $X_i'\subset X_i$, $i\in[1,j-1]$, as those are the represented variables, such that $\prod_{i=1}^{j-1} X_i\setminus X_i'$ has no solution (for any value of the last variables.) Therefore, we obtain the additional property that $X_i'=\emptyset$ for  $i\in [j,m]$.
\end{proof}

Let us observe that Observation~\ref{lem:red2_ext_3} can be used to obtain a similar additional conclusion for \cite[Theorem~1]{ksv13}.

\section{Conclusions and final comments}

In this paper we have presented Definition~\ref{d.rep_sys}, a notion of representation of a system of configurations using a pair of hypergraphs that generalizes previous definitions. Additionally, the notion is strong enough to translate the combinatorial removal lemma, Theorem~\ref{t.rem_lem_edge_color_hyper}, to the context of system of configurations, Theorem~\ref{t.rep_sys_rem_lem}. We observe that the systems of configurations induced by ``copies of a hypergraph H in K'' is representable. Additionally, we present a representation for the systems induced by ``patterns of the permutation $\tau$ in the permutation $\sigma$''.

The extra flexibility given by Definition~\ref{d.rep_sys} with respect to previous notions allows us to show that the configuration systems defined by homomorphisms between finite abelian groups are representable (see Proposition~\ref{p.repr_hom}). The combination of Proposition~\ref{p.repr_hom} and Theorem~\ref{t.rep_sys_rem_lem} is reflected in the removal lemma for homomorphism systems of finite abelian groups, Theorem~\ref{t.rem_lem_ab_gr}, which is the main result of the paper. Several applications of Theorem~\ref{t.rem_lem_ab_gr} are given in the introduction.

In \cite{canszeven14+}, a more algebraic definition of representation tied with the infinite aspect of the compact abelian groups is presented as \cite[Definition~3.7]{canszeven14+}, which follows the lines of \cite{sze10}. Indeed, the strong version of Definition~\ref{d.rep_sys} can be seen as the discrete and combinatorial analogue of \cite[Definition~3.7]{canszeven14+}. Let us mention that
the construction of the representation for the homomorphisms of finite abelian groups presented in sections~\ref{s.proof_rl-lsg-1}-\ref{s.finish_rem_lem_dkA1} can be adapted to fit \cite[Definition~3.7]{canszeven14+}. It is natural to ask if a removal lemma result holds for homomorphisms in compact abelian groups. However, such result presents technical difficulties involving some aspects of the construction presented here that arise when generic compact abelian groups are considered.\footnote{These difficulties are related with the extension of matrices performed in Section~\ref{s.construction_circular_matrix}.
}

\section*{Acknowledgements}

The author wants to thank Oriol Serra, Bal\'azs Szegedy and Pablo Candela for fruitful discussions, improvement suggestions as well as carefully reading the manuscript.

\end{document}